\documentclass{article}
\usepackage{graphicx} 
\usepackage[a4paper, top=2cm,bottom=3cm,left=3cm,right=3cm,marginparwidth=1.75cm]{geometry}

\usepackage{amsfonts}
\usepackage{amsmath}
\usepackage{amssymb}
\usepackage{amsthm}
\usepackage{bbm}
\usepackage{comment}

\usepackage[dvipsnames]{xcolor}

\newtheorem{theorem}{Theorem}[section]
\newtheorem{lemma}[theorem]{Lemma}
\newtheorem{cor}[theorem]{Corollary}
\newtheorem{prop}[theorem]{Proposition}

\newtheorem{example}[theorem]{Example}

\newcommand{\n}{\mathbb{N}}

\newcommand{\R}{\mathbb{R}}

\newcommand{\rd}{\mathbb{R}^d}

\newcommand{\rpluszero}{[0,\infty)}  

\newcommand{\borel}{\mathcal{B}}

\newcommand{\e}{\mathbb{E}}

\newcommand{\p}{\mathbb{P}}

\newcommand{\abs}[1]{\left|#1\right|}
\newcommand{\norm}[1]{\left|\left|#1\right|\right|}

\newcommand{\ind}[1]{\mathbbm{1}\left[#1\right]}
\newcommand{\dx}{{\rm d}}

\newcommand{\as}{a.\,s.\,}

\newcommand{\wrt}{w.\,r.\,t.\,}

\newcommand{\ppp}{\eta}
\newcommand{\markdist}{\mathbb{Q}}
\newcommand{\markdistlimd}{\mathbb{Q}^{LIM}_{d}}
\newcommand{\markdistlimtwo}{\mathbb{Q}^{LIM}_{2}}
\newcommand{\Wrho}{W_\rho}
\newcommand{\bfx}{{\bf x}}
\newcommand{\bfy}{{\bf y}}
\newcommand{\veps}{\varepsilon}
\newcommand{\sdrho}{\mathrm{s}_{d,\rho}}
\newcommand{\rdrho}{\mathrm{r}_{d,\rho}}
\newcommand{\ubarrho}{\overline{u}_\rho}
\newcommand{\dkr}{\mathsf{d_{KR}}}
\newcommand{\dtv}{\mathsf{d_{TV}}}

\newcommand{\B}{\mathrm{A}} 
\newcommand{\stworho}{\mathrm{s}_{2,\rho}}
\newcommand{\sthreerho}{\mathrm{s}_{3,\rho}}
\newcommand{\Mtilde}{\widetilde{M}}
\newcommand{\Atilde}{\widetilde{A}}
\newcommand{\stworhoAtilde}{\mathrm{s}_{2,\rho,\Atilde}}
\newcommand{\qalpharho}{\mathrm{q}_{\alpha,\rho}}

\newcommand{\keywords}[1]{\par\noindent\textbf{Keywords: }#1}

\newcommand{\msc}[1]{\par\noindent\textbf{MSC: }#1}

\numberwithin{equation}{section}

\title{Point process convergence of large inradii of Poisson--Laguerre tessellations}
\author{Matthias Schulte\footnote{Institute of Mathematics, Hamburg University of Technology, Germany, e-mail: \textit{matthias.schulte@tuhh.de}} \ and Martina Švarc Petráková\footnote{Department of Probability and Mathematical Statistics, Faculty of Mathematics and Physics, Charles University, Czechia, e-mail: \textit{petrakova@karlin.mff.cuni.cz}}}
\date{}

\begin{document}

\maketitle
\begin{abstract}
   In this paper we study a weighted generalization of the Poisson--Voronoi tessellation called the Poisson--Laguerre tessellation, where the nuclei of the generating Poisson process additionally carry independent non-negative random weights. For each cell we can define its inradius as the radius of the largest ball centered at the nucleus and contained in the cell. We consider point processes of nuclei, weights and inradii, where the nuclei are taken from growing observation windows and the processes are suitably rescaled and shifted to see the behavior of large inradii. We prove convergence in distribution to suitable Poisson processes, obtaining as corollaries the asymptotic behavior of the maximal inradii. Our results cover dimension $d \geq 3$ and bounded random weights, dimension $d = 2$ and random weights with suitable finite exponential moments as well as dimension $d \geq 2$ and random weights following a power-law distribution. Particularly, we observe a different behavior of the Poisson--Laguerre tessellation in the planar and higher dimensional cases when the weights are bounded. The proofs of our results are based on suitable Poisson process approximation techniques and a careful investigation of the geometry of large cells.
\end{abstract}

\msc{60D05, 60F05, 60G55, 60G70}

\section{Introduction and main results} \label{Sec:Introduction}

In this paper, we consider for Poisson--Laguerre tessellations the inradii of cells with nuclei within an observation window. For increasing observation windows it is shown that corresponding point processes of large inradii, nuclei and weights converge in distribution to Poisson processes.

A tessellation of $\rd$ is a collection of non-empty closed sets, so-called cells, such that the interiors of the cells are disjoint and non-empty, the union of all cells is $\rd$ and any bounded subset of $\rd$ is intersected by only finitely many cells (i.e., the tessellation is locally finite). Tessellations and, in particular, random tessellations whose generation involves randomness are very interesting geometric objects from a mathematical point of view and arise in a wide-range of applications, see e.g.\ \cite{ar:RJ25} for a recent overview on this topic as well as the monograph \cite{bo:Okabeetal00}, \cite[Chapter 10]{bo:SW08} or \cite[Chapter 5]{bo:Calka10}. 

An important class of tessellations are Voronoi tessellations. Here, a locally finite collection of points $\{x_i\}_{i\in I}$ in $\rd$ is given (i.e., any bounded set contains only finitely many of the points). For each $j\in I$ the Voronoi cell
$$
C(x_j,\{x_i\}_{i\in I}):=\{z\in \mathbb{R}^d: \|z-x_j\|\leq \|z-x_i\| \text{ for all } i\in I\}
$$
is constructed, which contains all points of $\rd$ that are not closer to any point of $\{x_i\}_{i\in I}$ than to $x_j$. The point $x_j$ is called the nucleus of the cell.

Laguerre tessellations are a generalization of Voronoi tessellations, where the nuclei $\{x_i\}_{i\in I}$ are equipped with weights. Let $\{w_i\}_{i\in I}$ be a collection of non-negative real numbers and define $\mathbf{x}_i:=(x_i,w_i)$ for $i\in I$. The so-called Laguerre cell with nucleus $\mathbf{x}_j$ for some $j\in I$ is given by
$$
L(\bfx_j,\{\bfx_i\}_{i\in I}):=\{z\in \rd: \|z-x_j\|^2-w_j^2\leq \|z-x_i\|^2-w_i^2 \text{ for all } i\in I\}
$$
and, under some mild assumptions on $\{w_i\}_{i\in I}$ (see \cite[Theorem 3.1]{ar:LZ08} for more details), the collection $L(\{\bfx_i\}_{i\in I})$ of all non-empty Laguerre cells forms a tessellation, the Laguerre tessellation. Since the weights are subtracted in the definition, Laguerre cells whose nuclei have larger weights are favored and tend to be larger. In case that all weights are identical, the Laguerre tessellation becomes the Voronoi tessellation. Note that in contrast to the Voronoi case the Laguerre cells can be empty and do not need to contain their nuclei.

So far we have considered Voronoi and Laguerre tessellations with respect to deterministic points and weights. Their random counterparts are obtained by assuming that the latter come from point processes, which will always be Poisson processes throughout this paper. The Poisson--Voronoi tessellation is obtained from a stationary Poisson process with intensity $\gamma>0$ in $\mathbb{R}^d$, while the Poisson--Laguerre tessellation $L(\ppp)$ is constructed from a stationary marked Poisson process $\eta$ in $\mathbb{R}^d\times[0,\infty)$ with intensity $\gamma>0$. Here the locations in $\rd$ of the points of $\eta$ are the nuclei and the independent marks in $[0,\infty)$ are the weights. We denote the distribution of the marks of $\ppp$ by $\markdist$ and call a random variable $M$ that is distributed according to $\markdist$ a typical mark. This means that $\eta$ has the intensity measure $\gamma\lambda_d\otimes\markdist$, where $\lambda_d$ is the $d$-dimensional Lebesgue measure. We always assume that 
\begin{equation}\label{assumption:markmoment}
\e M^d = \int_{[0,\infty)} s^d \, \dx \markdist(s) < \infty    
\end{equation}
in order to ensure that the Poisson--Laguerre tessellation is well defined almost surely (see \cite[Theorem 4.1]{ar:LZ08}).

The first systematic study of random Laguerre tessellations generated by a stationary marked Poisson process was conducted in \cite{ar:LZ08}, where explicit formulas for several geometric characteristics were derived and the convergence of Poisson--Laguerre to Poisson--Voronoi tessellations was discussed. Recently, the collection of papers \cite{ar:GKT22a,ar:GKT22b,ar:GKT22c,ar:GKT23} considered as random generator a  Poisson process in $\rd \times \R$ with intensity measure $\lambda_d \times \mathbb{F}$, where $\dx \mathbb{F}(h) = f(h) \dx \lambda_1(h)$ for selected choices of $f$, focusing mostly on the dual tessellation of the Poisson--Laguerre tessellation. The sectional properties of the three considered Poisson--Laguerre tessellations were studied in \cite{ar:GKT24} and quantitative central limit theorems for extreme points were derived in \cite{ar:BG25}, while \cite{ar:BG26} concerns the behavior of typical cells and degrees in the dual tessellations. Furthermore  \cite{ar:GWL25,ar:GWL26} investigate the Laguerre model for a more general class of $f$ - in \cite{ar:GWL25} it was proved that a sectional tessellation of a Poisson--Laguerre tessellation is again a (lower-dimensional) Poisson--Laguerre tessellation, while \cite{ar:GWL26} focuses on the convergence of a sequence of Poisson--Laguerre tessellations to a Poisson--Voronoi or Poisson--Laguerre tessellation, depending on the sequence of the generating Poisson processes. Regarding statistical inference for Poisson--Laguerre tessellations, the papers \cite{ar:FPY20, ar:JJV25} might be of interest. Note that some of these references consider a more general framework, where the random generator for the Poisson--Laguerre tessellation is a general - not necessarily marked - Poisson process in $\rd \times \R$. 

In order to measure the size of a cell of the Poisson--Laguerre tessellation with nucleus $\mathbf{x}=(x,s)\in\eta$, we define its inradius 
$$
r(\bfx,\ppp) := \begin{cases} \sup\{r\geq0: B^d(x,r)\subseteq L(\mathbf{x},\ppp)\}, & \quad x\in L(\bfx,\ppp), \\ 0, & \quad x\notin L(\bfx,\ppp), \end{cases}
$$
with respect to the nucleus, where $B^d(y,r)$ stands for the closed ball in $\rd$ with radius $r\in[0,\infty)$ and center $y\in\mathbb{R}^d$. This means that the inradius is the radius of the largest ball around the nucleus that is contained in the cell if the nucleus belongs to the cell. Note that our definition of the inradius is different from the usual definition of the inradius of a set since we take only balls centered at the nucleus into account.

In the following, we are interested in large inradii of Poisson--Laguerre tessellations. As they are stationary, we can only consider cells whose nuclei are within given observation windows. To this end, we fix a compact convex set $W\subseteq\mathbb{R}^d$ with volume one and define $\Wrho:=\rho W:=\{\rho z: z\in W\}$ for $\rho\geq1$. We aim to study the asymptotic behavior of the collection of all triples of nuclei in $\Wrho$, their marks and the corresponding inradii as $\rho\to\infty$. We can think of it as a point process
$$
\sum_{(x,m)\in\ppp\cap \Wrho\times[0,\infty)} \delta_{(x,m,r((x,m),\eta))}
$$ 
in $\mathbb{R}^d\times[0,\infty)\times[0,\infty)$, where $\delta_y$ stands for the Dirac measure concentrated at a point $y$. We use the standard definition that a point process is a random locally finite counting measure, but, as usual, we often employ notations that treat point processes as sets. In our main results, we will shift and rescale the entries of the points of the above point process in suitable ways. For Poisson--Laguerre tessellations one can expect that the behavior of the cells and, in particular, that of large inradii might heavily depend on the mark distribution $\markdist$. In the sequel, we study three different situations for the generating Poisson process $\ppp$. First, we assume the marks of $\ppp$ to be bounded and observe a different behavior for the planar case and for higher dimensions, see Theorem \ref{thm:bounded}. For the planar case we are able to generalize the result to sufficiently light-tailed marks, as can be seen in Theorem \ref{thm:lighttails}. Finally, we study what happens with marks distributed according to a power-law, obtaining rather different results than in the previous cases, see Theorem \ref{thm:powerlaw}. 

In order to present our main results, we have to fix some notation. Denote by $v_d$ the volume of the $d$-dimensional unit ball and write $\lambda_d\lvert_W$ for the restriction of the $d$-dimensional Lebesgue measure $\lambda_d$ to $W$. For a suitable space $\mathcal{E}$ let $\mathcal{M}_p(\mathcal{E})$ be the space of locally finite counting measures on $\mathcal{E}$ so that point processes on $\mathcal{E}$ are random elements in $\mathcal{M}_p(\mathcal{E})$. When concerning convergence in distribution of point processes, we use such spaces; see Section \ref{subsec:pointproccesandconv} for more details on point processes and their convergence in distribution, which is denoted by $\overset{d}{\longrightarrow}$ in the sequel.

We first provide our main result for the case of bounded marks, where \eqref{assumption:markmoment} is obviously satisfied.

\begin{theorem}\label{thm:bounded}
Let $d \geq 2$ and assume that there exists a constant $\B\in(0,\infty)$ such that ${\mathbb{Q}([0,\B])=1}$ and $\mathbb{Q}([0,\B-\varepsilon])<1$ for all $\varepsilon\in(0,\B)$. Define $c_d := 2^d v_d$. Then there exists a family of real numbers $\{\sdrho\}_{\rho \geq 1}$ such that
\begin{equation} \label{defCR}
\lim_{\rho \rightarrow \infty} \e \sum_{\bfx \in \ppp \cap W_\rho \times [0,\B]} \ind{r(\bfx,\ppp)^d\geq \frac{u + \sdrho}{c_d \gamma}} = e^{-u} \quad  \text{ for all } \, u \in \R.
\end{equation}
For any choice of $\{\sdrho\}_{\rho \geq 1}$ satisfying \eqref{defCR} define the point processes
    \begin{equation} \label{def:xirho}
        \xi_\rho:= \sum_{(x,m) \in \ppp \cap \Wrho \times [0,\B]} \delta_{\left(\frac{1}{\rho} x, m,c_d \gamma r((x,m),\ppp)^d-\sdrho\right)}
    \end{equation}
for $\rho \geq 1$ and let $\zeta_d$ be a Poisson process on $W \times [0,\infty) \times \R$ with intensity measure $\lambda_d\lvert_W \otimes \markdistlimd \otimes \mathbb{K}$, where
$$
\markdistlimd = \begin{cases} \mathbb{E}[\ind{M\in \cdot } e^{2v_2\gamma M^2}]/\mathbb{E}[e^{2v_2\gamma M^2}], & \quad d=2, \\ \delta_{\B}, & \quad d\geq 3, \end{cases}
$$
with $M\sim\markdist$ and $\mathbb{K}$ is the measure on $\R$ such that $\mathbb{K}([a,b])=e^{-a}-e^{-b}$ for $a,b \in \R$ with $a\le b$. Then
\begin{equation} \label{conv:bounded}
    \xi_\rho \underset{\rho \rightarrow \infty}{\overset{d}{\longrightarrow}} \zeta_d \quad \text{ in the space } \mathcal{M}_p(W \times [0,\infty) \times \R).
\end{equation}
\end{theorem}

We call any family $\{\sdrho\}_{\rho \geq 1}$ satisfying \eqref{defCR} a correct shift for the large inradii of $L(\ppp)$. To see the motivation behind the definition of a correct shift, notice the following. Let $\mathbb{L}_\rho$ denote the intensity measure of the point process $\xi_\rho$. Then \eqref{defCR} can be rewritten as
\begin{equation} \label{formula:CRwithIntensityMeasure}
    \lim_{\rho \rightarrow \infty} \mathbb{L}_\rho(W \times [0,\B] \times [u,\infty)) = e^{-u} \quad  \text{ for all } u \in \R.
\end{equation}
The correct shift is unique up to the addition of $\{\alpha_\rho\}_{\rho \geq 1}$ converging to $0$. For dimensions $d \in\{2,3\}$, explicit formulas for choices of correct shifts are given, see Proposition \ref{prop:ExistenceCorResc}, namely 
\begin{equation}\label{CRtwo} 
    \stworho = \log \rho^2 - 2v_2\gamma\e M^2 + \log \e e^{2v_2\gamma M^2}+ \log \gamma. 
\end{equation}
for the planar case and 
\begin{equation}\label{CRthree}
     \sthreerho = \log \rho^3 - 3 (v_3 \gamma)^{\frac{2}{3}} \big(\log \rho^3\big)^{\frac{1}{3}}  \e M^2 +\log\Big( \e \exp\big( 3 (v_3 \gamma)^{\frac{2}{3}} \big(\log \rho^3\big)^{\frac{1}{3}}  M^2  \big) \Big) + \log\gamma
\end{equation}
for the three-dimensional case. Furthermore, in Section \ref{subsec:ProofsExamples} we compute \eqref{CRthree} for selected choices of $\markdist$ and obtain different lower order terms in $\rho$. For higher dimensions, our approach does not provide explicit formulas for correct shifts, but we present an example of a correct shift for a particular choice of $\markdist$ and dimension $d = 4$.

As a corollary of Theorem \ref{thm:bounded}, we get convergence in distribution of the maximal inradius as well as its nucleus and the corresponding mark. We say that a random variable $G$ follows a standard Gumbel distribution if $\p(G\leq u) = e^{-e^{-u}}$ for $u \in \R$ and write $\operatorname{Unif}(W)$ for the uniform distribution on $W$.

\begin{cor}\label{cor:bounded}
    Let $d \geq 2$, let $\mathbb{Q}$ be as in Theorem \ref{thm:bounded} and let $\{\sdrho\}_{\rho \geq 1}$ be any correct shift. Denote by $R_{\max,\rho}$ the maximal inradius of a cell of the Poisson--Laguerre tessellation with nucleus in $\Wrho$ and by $(X_{\max,\rho},M_{\max,\rho})$ the corresponding nucleus and its mark. 
Then
    \begin{equation}\label{conv:boundedcor}
        \left(\frac{1}{\rho}X_{\max,\rho},M_{\max,\rho} ,c_d\gamma R_{\max,\rho}^d-\sdrho\right) \underset{\rho \rightarrow \infty}{\overset{d}{\longrightarrow}} (X,M_d,G)
    \end{equation}
 with independent $X \sim \operatorname{Unif}(W)$, $M_d \sim \markdistlimd$ and $G$ standard Gumbel.    
\end{cor}

Poisson--Voronoi tessellations are contained as a special case in Theorem \ref{thm:bounded} and Corollary \ref{cor:bounded} if $\mathbb{Q}$ is concentrated on some $\B\in(0,\infty)$, i.e., all points of $\eta$ have the same mark $\B$. In this case, we write $\overline{\ppp}$ for a stationary Poisson process in $\mathbb{R}^d$ with intensity $\gamma>0$. We denote by $r(x,\overline{\ppp})$ the inradius of the Poisson--Voronoi cell with nucleus $x\in\overline{\ppp}$ and let $\overline{R}_{\max,\rho}$ be the maximal inradius of a Poisson--Voronoi cell with nucleus in $\Wrho$. It was shown in \cite[Theorem 1, Equation (2a)]{ar:CCh14} that
\begin{equation}\label{eqn:Maximum_Poisson_Voronoi}
c_d\gamma \overline{R}_{\max,\rho}^d-\log (\gamma\rho^d) \underset{\rho \rightarrow \infty}{\overset{d}{\longrightarrow}} G.
\end{equation}
On the level of point process convergence it was derived in \cite[Theorem 3.5]{ar:PS22} that
\begin{equation}\label{eqn:convergence_inradii}
\sum_{x \in \overline{\ppp} \cap\Wrho} \delta_{c_d \gamma r(x,\overline{\ppp})^d-\log(\gamma\rho^d)} \underset{\rho \rightarrow \infty}{\overset{d}{\longrightarrow}} \zeta,
\end{equation}
where $\zeta$ is a Poisson process on $\R$ with intensity measure $\mathbb{K}$, while, for $a\in\R$,
\begin{equation}\label{eqn:convergence_spatial}
\sum_{x \in \overline{\ppp} \cap\Wrho} \ind{c_d \gamma r(x,\overline{\ppp})^d-\log(\gamma\rho^d) \geq a} \delta_{\frac{1}{\rho}x} \underset{\rho \rightarrow \infty}{\overset{d}{\longrightarrow}} \zeta_a
\end{equation}
with a homogeneous Poisson process $\zeta_a$ on $W$ with intensity $e^{-a}$ was established in \cite[Example 4.3]{ar:O25}. The statements \eqref{eqn:Maximum_Poisson_Voronoi}, \eqref{eqn:convergence_inradii} and \eqref{eqn:convergence_spatial} follow directly from Corollary \ref{cor:bounded} and Theorem \ref{thm:bounded} once one has checked that $\{\log(\gamma\rho^d)\}_{\rho\geq 1}$ is a correct shift (for $d \in \{2,3\}$, this is implied by the formulas \eqref{CRtwo} and \eqref{CRthree}). For the convergence in distribution in \eqref{eqn:Maximum_Poisson_Voronoi} a rate of convergence was obtained in \cite[Theorem 3.15]{ar:PS21}. Note that \cite{ar:O25} provides rates of convergence for point process convergence and allows to study more general size functionals than the inradius. Since the inradius of a cell of the Poisson--Voronoi tessellation is half the distance from the nucleus to its nearest neighbor, long edges of nearest neighbor graphs are a related problem. However, studying long edges of the nearest neighbor graph formed by the points of $\overline{\ppp}\cap W_\rho$ is not equivalent to large inradii because potential neighbors outside of the observation window are missing. For $d\ge 3$ this boundary effect becomes non-negligible as observed in \cite{ar:Penrose1999}, where for large nearest neighbor distances analogous results to \eqref{eqn:convergence_inradii} were shown for $W$ being the one- or two-dimensional unit cube or the $d$-dimensional torus for $d\ge1$.

Theorem \ref{thm:bounded} and Corollary \ref{cor:bounded} show that there is a different behavior for the cases $d=2$ and $d\geq3$. For $d\geq 3$ the marks of the nuclei of the cells with large inradii tend to $\B$ as $\rho\to\infty$ because $\markdistlimd$ is concentrated at $\B$. It is therefore clear that in dimensions $d \geq 3$ Theorem \ref{thm:bounded} and Corollary \ref{cor:bounded} cannot be easily generalized to unbounded non-negative marks. In the planar case, the distribution of the marks of nuclei of cells with large inradii converges in distribution to a size-biased version of the mark distribution $\mathbb{Q}$ so that it could be also much smaller than $\B$.  Since the limiting mark distribution $\markdistlimtwo$ remains well-defined under a suitable moment assumption on the typical mark $M \sim \markdist$, it is natural to try to generalize the results of  Theorem \ref{thm:bounded} and Corollary \ref{cor:bounded} to a family of sufficiently light-tailed distributions. Indeed, we have the following result.

\begin{theorem} \label{thm:lighttails}
    Let $d = 2$ and assume that the typical mark $M \sim \markdist$ satisfies $\e e^{(4\gamma v_2+\varepsilon)M^2} < \infty$ for some $\varepsilon > 0$. Then there exists a family of real numbers $\{\stworho\}_{\rho \geq 1}$ such that
\begin{equation*}
\lim_{\rho \rightarrow \infty} \e \sum_{\bfx \in \ppp \cap W_\rho \times [0,\infty)} \ind{r(\bfx,\ppp)^2\geq \frac{u + \stworho}{c_2 \gamma}} = e^{-u} \quad  \text{ for all } \, u \in \R.
\end{equation*}
For $\rho \geq 1$ let $\xi_\rho$ be defined as in \eqref{def:xirho}. Then the results \eqref{conv:bounded} and \eqref{conv:boundedcor} remain true. 
\end{theorem}
Thanks to the moment assumption on the mark distribution $\markdist$, the right-hand side of \eqref{CRtwo} is well defined and, as we will see later, it remains a valid choice for a correct shift in the light-tailed case.

Next we study the situation that the marks follow a power-law distribution. We say that the typical mark $M$ has a regularly varying tail with exponent $\alpha>0$ if its distribution $\markdist$ satisfies 
\begin{equation*}
\markdist((u,\infty)) = l(u) u^{-\alpha}  
\end{equation*}
for all $u>0$, where $l: (0,\infty)\to[0,\infty)$ is a slowly varying function at $\infty$, i.e., $\lim_{v \rightarrow \infty} \frac{l(tv)}{l(v)} = 1$ for any $t >0$. In order to ensure that \eqref{assumption:markmoment} is satisfied, we always require $\alpha>d$.

Our main result is again the convergence in distribution of a process of suitably rescaled inradii, nuclei and their marks. We consider point processes on the space $W\times (0,\infty]^2$ with the product metric, where we use the standard Euclidean metric on $W$ and the metric induced by the mapping $x \mapsto \frac{1}{x}$ on $(0,\infty]$. In this setting, the inradii are rescaled by $\qalpharho:=\gamma^{1/\alpha} q(\rho)$ for $\rho \geq 1$, where $q(\rho)$ is the $(1-\frac{1}{\rho^d})$-quantile of the mark distribution $\markdist$.

\begin{theorem}\label{thm:powerlaw} 
     Let $d \geq 2$ and assume that $M \sim \markdist$ has a regularly varying tail with exponent $\alpha>d$ and that $M>0$ $\p$-a.s. Define the point processes 
     \begin{equation} \label{def:XirhoPL}
     \Xi_\rho := \sum_{\bfx=(x,m) \in \ppp \cap \Wrho \times (0,\infty)} \delta_{\left(\frac{1}{\rho}x,\frac{1}{\qalpharho} m, \frac{1}{\qalpharho}r(\bfx,\ppp)\right)}
\end{equation}
for $\rho\geq1$. Let $\Psi_\alpha$ be a Poisson process on $W \times (0,\infty]^2$ with intensity measure $\mathbb{M}_\alpha = \lambda_d\lvert_W \otimes \mathbb{K}_\alpha$, where $\mathbb{K}_\alpha(\{(x,x): a < x \leq b\}) = a^{-\alpha}-b^{-\alpha}$ for all $0<a<b$ and $\mathbb{K}_\alpha(\{(x,x): x > 0\}^c)=0$. Then
     \begin{equation*}
        \Xi_\rho  \underset{\rho \rightarrow \infty}{\overset{d}{\longrightarrow}} \Psi_\alpha \text{ in the space } \mathcal{M}_p(W \times (0,\infty]^2).
    \end{equation*}
\end{theorem}

Let us note that for a technical reason, we do not allow $M$ to attain the value zero. However, this is not a restriction since the transformation $m \mapsto \sqrt{m^2+t}$ with a fixed $t\in(0,\infty)$ for the marks does not change the Laguerre tessellation and hence ensures that we can work with positive marks without loss of generality. Again, we get a corollary on the cell with the maximal inradius. Recall that a random variable $F$ follows a Fr\'echet distribution with shape parameter $\alpha$  if $\p(F\leq u) = e^{-u^{-\alpha}}$ for any $u \geq 0$. In this case, we write $F\sim\text{{\rm Fr\'echet}}(\alpha)$.

\begin{cor} \label{cor:powerlaw}
     Let $d \geq 2$ and assume that $M$ is as in Theorem \ref{thm:powerlaw}. Denote by $R_{\max,\rho}$ the maximal inradius of a cell of the Poisson--Laguerre tessellation with nucleus in $\Wrho$ and by $(X_{\max,\rho},M_{\max,\rho})$ the corresponding nucleus and its mark. Then
\begin{equation*}
\left(\frac{1}{\rho}X_{\max,\rho},\frac{1}{\qalpharho}M_{\max,\rho} ,\frac{1}{\qalpharho}R_{\max,\rho}\right) \underset{\rho \rightarrow \infty}{\overset{d}{\longrightarrow}} (X,F,F)
\end{equation*}
with independent $X \sim \operatorname{Unif }(W)$ and $F\sim\text{{\rm Fr\'echet}}(\alpha)$.
\end{cor}

Comparing our results for bounded marks in dimension $d\geq 3$, light-tailed marks in dimension $d = 2$ and for marks distributed according to a power-law distribution, we see that the transformations of the inradii and the intensity measures of the limiting Poisson processes are different. The limiting distribution of the maximal inradius is the Gumbel distribution in the first two cases and the Fr\'echet distribution in the last. The behavior for bounded marks is still very similar to that for constant marks, i.e., for Poisson--Voronoi tessellations, but we already observe a different behavior for the planar case and higher dimensional cases, triggering the question what happens to light-tailed but possibly unbounded marks in higher dimensions. Since one can think of the light-tailed and power-law mark distributions as two extremal cases, it also begs the question what one can expect for other classes of mark distributions.

Our proofs for bounded and light-tailed marks and for marks distributed according to a power-law follow completely different strategies. In the case of bounded marks, one of the challenges is to prove the existence of a correct shift for general dimension $d\geq 2$ as we were able to find explicit formulas for correct shifts only in dimensions $d\in\{2,3\}$. To obtain the point process convergence, a general bound for the Poisson process approximation from \cite{ar:BSY22} is applied. Furthermore, for the planar case it is possible to apply approximation arguments to generalize our results from the bounded setting to a setting with light-tailed marks. For marks distributed according to a power-law, we show that large inradii can be approximated by the marks of the nuclei of the corresponding cells. In contrast to the point process of rescaled nuclei, rescaled marks and rescaled inradii with its involved dependence structure, the point process of rescaled nuclei and rescaled marks is a Poisson process and, thus, much easier to study. A similar proof strategy was employed for large degrees and large components of inhomogeneous random graphs in \cite{ar:BS22,ar:LS25}.

This paper is organized as follows. Section \ref{sec:Preliminaries} is dedicated to preliminary results regarding point processes, their convergence and Laguerre tessellations. The proofs of Theorem \ref{thm:bounded} and Corollary \ref{cor:bounded} can be found in Section \ref{sec:ProofBC} together with the proof of the order of any correct shift (see Proposition \ref{prop:ExistenceCorResc}) and some examples of correct shifts for dimensions $d = 3$ and $d = 4$ (see Section \ref{subsec:ProofsExamples}). Section \ref{sec:ProofLT} contains the proof of Theorem \ref{thm:lighttails}, while the proofs of Theorem \ref{thm:powerlaw} and the following corollary are given in Section \ref{sec:ProofConvPL}.

\section{Preliminaries} \label{sec:Preliminaries}

\subsection{Point processes and their convergence} \label{subsec:pointproccesandconv}

Let $(\mathcal{E},\mathsf{d})$ be a complete separable locally compact metric space and let $\mathcal{B}(\mathcal{E})$ be its Borel $\sigma$-algebra. We denote by $\mathcal{M}_p(\mathcal{E})$ the space of locally finite counting measures on $\mathcal{E}$. We equip $\mathcal{M}_p(\mathcal{E})$ with the smallest $\sigma$-algebra such that all mappings $\mathcal{M}_p(\mathcal{E})\ni\nu\mapsto\nu(B)$ for $B\in\mathcal{B}(\mathcal{E})$ are measurable. A point process on $\mathcal{E}$ is a random element in $\mathcal{M}_p(\mathcal{E})$. Recall that a point process $\Gamma$ on $\mathcal{E}$ is a Poisson process with a locally finite intensity measure $\mu$ if
\begin{itemize}
\item $\Gamma(B)$ follows a Poisson distribution with parameter $\mu(B)$ for every $B\in\mathcal{B}(\mathcal{E})$ (where the Poisson distribution with parameter $\infty$ takes almost surely the value $\infty$),
\item $\Gamma(B_1),\hdots,\Gamma(B_k)$ are independent for all pairwise disjoint $B_1,\hdots,B_k\in\mathcal{B}(E)$ and $k\in\n$.
\end{itemize}
For more details on point processes and, in particular, Poisson processes we refer the reader to, for example, \cite{bo:LP17}.

We equip $\mathcal{M}_p(\mathcal{E})$ with the topology of vague convergence, i.e., the smallest topology such that for any continuous bounded function $f: \mathcal{E} \rightarrow [0,\infty)$ with compact support the mapping $\nu \mapsto \int f \,\dx \nu$ is continuous. The vague convergence is denoted by $\overset{v}{\longrightarrow}$ and the corresponding Borel $\sigma$-algebra coincides with the $\sigma$-algebra mentioned above (see \cite[Lemma 4.7]{bo:OKallenMeasures}). By a continuous function $g:\mathcal{M}_p(\mathcal{E})\to\R$ we mean a function that is continuous with respect to the usual topology on $\R$ and the topology of vague convergence on $\mathcal{M}_p(\mathcal{E})$. For point processes $\{\Gamma_\rho\}_{\rho\ge 1}$ and $\Gamma$ on $\mathcal{E}$ we say that $\{\Gamma_\rho\}_{\rho\ge 1}$ converges in distribution to $\Gamma$ and write $\Gamma_\rho\overset{d}{\underset{\rho\to \infty}{\longrightarrow}}\Gamma$ if
$$
\lim_{\rho\to\infty}\e g(\Gamma_\rho) = \e g(\Gamma)
$$
for all continuous bounded functions $g: \mathcal{M}_p(\mathcal{E}) \to \R$.

A family $\mathcal{I} \subset \borel(\mathcal{E})$ is denoted as semi-ring if it is closed under finite intersections and every proper difference of sets from $\mathcal{I}$ is a finite union of disjoint elements of $\mathcal{I}$. Furthermore, a family $\mathcal{J}\subseteq \{B \in \borel(\mathcal{E}): B \text{ is relatively compact}\}$ is called dissecting if any open set is a countable union of elements of $\mathcal{J}$ and any relatively compact set has a finite cover consisting of elements of $\mathcal{I}$, see \cite[Chapter 1]{bo:OKallenMeasures} for more details. In the sequel, we work with dissecting semi-rings. Particularly, for a product space of finitely many spaces the product of corresponding dissecting semi-rings is a dissecting semi-ring, see \cite[Lemma 1.9 (vi)]{bo:OKallenMeasures}. When working with a sequence of Poisson point processes, the convergence in distribution can be verified in terms of their intensity measures.

\begin{lemma}\label{lemma:AuxPoissonConv}
    Let $\{\Gamma_\rho\}_{\rho \geq 1}$ and $\Gamma$ be Poisson processes on $(\mathcal{E},\mathsf{d})$ with intensity measures $\{\mu_\rho\}_{\rho\ge 1}$ and $\mu$. Let $\mathcal{I}\subseteq \{B \in \borel(\mathcal{E}): B \text{ is relatively compact, }  \mu(\partial B) = 0\}$ be a dissecting semi-ring. If
    $$
    \lim_{\rho\to\infty} \mu_\rho(I) = \mu(I) \quad \text{for all} \quad I \in \mathcal{I},
    $$
    then
    $$
    \Gamma_\rho \overset{d}{\underset{\rho \rightarrow \infty}{\longrightarrow}} \Gamma \text{ in } \mathcal{M}_p(\mathcal{E}).
    $$
    \end{lemma}

\begin{proof}
        We can regard $\{\Gamma_\rho\}_{\rho\ge 1}$ and $\Gamma$ as Cox processes with deterministic driving measures $\{\mu_\rho\}_{\rho\ge 1}$ and $\mu$. It follows from \cite[Lemma 4.17]{bo:OKallenMeasures} that 
        \begin{equation*}
            \Gamma_\rho \overset{d}{\underset{\rho \rightarrow \infty}{\longrightarrow}} \Gamma \text{ in } \mathcal{M}_p(\mathcal{E}) \iff \mu_\rho \overset{v}{\underset{\rho \rightarrow \infty}{\longrightarrow}} \mu.
        \end{equation*}
        From \cite[Lemma 4.1]{bo:OKallenMeasures} we know that $\lim_{\rho\to\infty}\mu_\rho(I) = \mu(I)$ for all $I \in \mathcal{I}$ implies the vague convergence $\mu_\rho \overset{v}{\underset{\rho \rightarrow \infty}{\longrightarrow}} \mu$, which completes the proof. 
    \end{proof}

To prove the convergence in distribution of the maximal inradius, its nucleus and mark as stated in Corollaries \ref{cor:bounded} and \ref{cor:powerlaw}, we need to work with the following map. Let $(\mathcal{E},\mathsf{d})$ be either $W\times[0,\infty)\times \R$ with Euclidean metric or $W  \times (0,\infty]^2$ with the metric described before Theorem \ref{thm:powerlaw} .  Define a map $T: \mathcal{M}_p(\mathcal{E}) \rightarrow \R^{d+2}$ by
\begin{equation}\label{def:argmaxmap}
T(\varphi) :=
\begin{cases}
     \underset{(y,m,r) \in \varphi}{\arg \max} r, \, &\varphi \in \mathcal{M}_p(\mathcal{E}) \text{ such that } \underset{(y,m,r) \in \varphi}{\max} r \text{ exists and has a unique maximizer, } \\
     0,\, &\text{ otherwise. }
\end{cases}
\end{equation}

A counting measure $\varphi$ on $\mathcal{E}$ is called simple if $\varphi(\{x\})\le 1$ for all $x\in\mathcal{E}$. As is shown in the next lemma, when we consider a restriction of $T$ to a smaller space, then it is continuous on a suitable set of measures. 

\begin{lemma} \label{lemma:continuity of argmax}
    Fix a non-empty compact set $K \subseteq \mathcal{E}$ and define a set $\mathcal{C}_K \subseteq \mathcal{M}_p(K)$ as 
    \begin{equation} \label{def:SubsetCK}
        \mathcal{C}_K:= \{\varphi \in \mathcal{M}_p(K): \varphi \text{ is simple, }  \underset{(y,m,r) \in \varphi}{\max} r \text{ has a unique maximizer}\}.
    \end{equation}
    Then the map $T: \mathcal{M}_p(K) \rightarrow  \R^{d+2}$ defined in \eqref{def:argmaxmap} is continuous on $\mathcal{C}_K$.
\end{lemma}

\begin{proof} 
    For $y \in K$ and $\varepsilon > 0$ define $f_{y,\varepsilon}: K\to\R$ by 
    $$
    f_{y,\varepsilon}(x) := 
    \begin{cases}
        0  ,  &\mathsf{d}(x,y) > \varepsilon,\\
        1-\frac{\mathsf{d}(x,y)}{\varepsilon}, \quad  &\mathsf{d}(x,y) \leq \varepsilon.
    \end{cases}
    $$
    Then $f_{y,\varepsilon}$ is continuous and has bounded support. Now assume that $\{\mu_n\}_{n\in\n}$ is a sequence in $\mathcal{M}_p(K)$ and $\mu\in\mathcal{C}_K$ such that $\mu_n \overset{v}{\underset{\rho \rightarrow \infty}{\longrightarrow}} \mu$ in $\mathcal{M}_p(K)$. Since $ K$ is compact, it follows from the vague convergence that $\mu_n( K) = \mu( K)$ at least for all $n$ large enough. Furthermore for $y \in \mu$ and all $\varepsilon > 0$ small enough
    \begin{align*}
        \lim_{n \rightarrow \infty} \int_K  f_{y,\varepsilon}(x) \,\mu_n(\dx x) = \int_K f_{y,\varepsilon}(x) \, \mu(\dx x) = 1. 
    \end{align*}
    This means that the points of $\mu_n$ converge to those of $\mu$ as $n\to\infty$. Thus, there exists a sequence $\{x_n\}_{n \in \n}$ such that $x_n \in \mu_n$ for $n\in\n$ and $\mathsf{d}(x_n,T(\mu)) \rightarrow 0$ as $n \rightarrow \infty$, which finishes the proof. 
\end{proof}

\subsection{Poisson process approximation}

Let $(\mathcal{E},\mathsf{d})$ be a complete separable locally compact metric space and denote by $\mathcal{M}_{p,f}(\mathcal{E}) \subseteq \mathcal{M}_p(\mathcal{E})$ the subset of finite counting measures on $(\mathcal{E},\borel(\mathcal{E}))$. The total variation distance $\dtv$ of two finite measures $\nu_1$ and $\nu_2$ on $\mathcal{E}$ is given by
 \begin{equation*}
         \dtv(\nu_1, \nu_2):= \sup_{B \in \borel(\mathcal{E})}\abs{\nu_1(B) - \nu_2(B)}. 
 \end{equation*}
The Kantorovich-Rubinstein distance $\dkr$ between two finite point processes $\xi$ and $\zeta$ on $\mathcal{E}$ is defined as
 \begin{equation*}
         \dkr(\xi, \zeta):= \sup_{h \in \textrm{Lip}(\mathcal{E})}\abs{\e h(\xi) - \e h(\zeta)},
 \end{equation*}
 where $\textrm{Lip}(\mathcal{E}):= \{h: \mathcal{M}_{p,f}(\mathcal{E}) \rightarrow \R, \text{ measurable and } 1\text{-Lipschitz \wrt} \dtv\}$. Bounding the $\dkr$ distance between restrictions of the processes $\xi_\rho$ and suitable Poisson processes is one of the key steps of the proof of Theorem \ref{thm:bounded}. To this end, we rely on a Poisson process approximation result from \cite{ar:BSY22}, which we present next.

Let $\tilde{\eta}$ be a Poisson process on a complete separable locally compact metric space $\mathbb{X}$ with a locally finite intensity measure $\mathbf{K}$. Suppose that $f: \mathbb{X} \times \mathcal{M}_{p}(\mathbb{X})\to\mathcal{E}$ and $g: \mathbb{X} \times \mathcal{M}_{p}(\mathbb{X})\to\{0,1\}$ are measurable functions and that there exist sets $\{S(x)\}_{x\in\mathbb{X}}$ such that, for all $\omega\in \mathcal{M}_{p}(\mathbb{X})$ and $x\in\omega$, 
$$
g(x,\omega) = g(x,\omega\cap S(x))
$$
and
$$
f(x,\omega) = f(x,\omega\cap S(x)) \quad \text{if} \quad g(x,\omega) = 1.
$$
Moreover, we require that $x\in S(x)$ for all $x\in \mathbb{X}$ and that $x\mapsto S(x)$ is measurable in a suitable way (see \cite[Section 4]{ar:BSY22} for details). We define the point process
$$
\xi:= \sum_{x\in\tilde{\eta}} g(x,\tilde{\eta}) \delta_{f(x,\tilde{\eta})}
$$
and denote its intensity measure by $\mathbf{L}$. The next theorem is a simplified version of \cite[Theorem 4.1]{ar:BSY22}.

\begin{theorem}\label{thm:BSY}
Let $\xi$ be the point process introduced above, let the mentioned assumptions on $f$, $g$ and $\{S(x)\}_{x\in\mathbb{X}}$ be satisfied and assume that $\mathbf{L}(\mathcal{E})<\infty$. Denote by $\zeta$ a Poisson process on $\mathcal{E}$ with a finite intensity measure $\mathbf{M}$. Then
$$
\dkr(\xi,\zeta) \leq \dtv(\mathbf{L},\mathbf{M}) + E_2 + E_3
$$
with
\begin{align*}
E_2 & := 2 \int_{\mathbb{X}^2} \ind{S(x)\cap S(y) \neq \emptyset} \e \big[ g(x,\tilde{\eta}+\delta_x) \big] \e \big[ g(y,\tilde{\eta}+\delta_y) \big] \, \dx\mathbf{K}^2(x,y) \\
E_3 & := 2 \int_{\mathbb{X}^2} \ind{S(x)\cap S(y) \neq \emptyset} \e \big[ g(x,\tilde{\eta}+\delta_x+\delta_y)  g(y,\tilde{\eta}+\delta_x+\delta_y) \big] \, \dx\mathbf{K}^2(x,y).
\end{align*}
\end{theorem}

\subsection{Laguerre tessellations and their inradii}\label{subsec:Properties}

Let $\varphi\in\mathcal{M}_p(\rd\times[0,\infty))$ be simple and such that $\varphi(\cdot \times [0,\infty))$, the projection of $\varphi$ to $\rd$, is locally finite. Recall that the Laguerre cell of a point $(x,m)\in \varphi$ is given by
$$
L((x,m), \varphi):= \{z\in\rd: \|z-x\|^2-m^2 \leq \|z-y\|-\tilde{m}^2 \quad \text{for all} \quad (y,\tilde{m})\in\varphi\}.
$$
We denote the collection of Laguerre cells with non-empty interiors by $$L(\varphi):=\{L((x,m),\varphi):{(x,m)\in\varphi, \operatorname{int}L((x,m),\varphi) \neq \emptyset}\}$$ and say that $L(\varphi)$ is a tessellation if the union of the Laguerre cells is $\rd$. Note that $L(\varphi)$ does not need to be a tessellation since there can be points that do not belong to any cell. For a point $z_0\in\rd$ this can happen if there exists a sequence $\{(x_n,m_n)\}_{n\in\n}$ of points of $\varphi$ such that $\lim_{n\to\infty}\|z_0-x_n\|^2-m_n^2=-\infty$. In case of an underlying marked stationary Poisson process $\eta$, the assumption \eqref{assumption:markmoment} on the distribution $\markdist$ of the marks is equivalent to that $L(\eta)$ is almost surely a tessellation (see \cite[Theorem 4.1]{ar:LZ08}).

If the inradius of a Laguerre cell $L((x,m),\varphi)$ is larger than some fixed $u > 0$, there cannot be points with too large weights close to the point $(x,m)$. This key observation is rigorously formulated and proved in the next lemma. Throughout this paper, we write $ \int_a^b \dots \dx \markdist$ for $ \int_{[a,b]} \dots \dx \markdist$ and $ \int_a^\infty \dots \dx \markdist$ for $ \int_{[a,\infty)} \dots \dx \markdist$. 

\begin{lemma} \label{lemma:EmptySetFormula}
Let $\varphi \in \mathcal{M}_p(\rd\times[0,\infty))$ be as described above and such that $L(\varphi)$ is a tessellation, let $(x,s) \in \varphi$, let $u >0$ and define $D(s,u):= \{(y,t) \in \rd \times \rpluszero: t^2 > \norm{y}^2-2u\norm{y}+s^2\}$. Then
\begin{equation}\label{equivalence:inradius/emptyset}
    r((x,s),\varphi) \geq u \quad \iff \quad \varphi(D(s,u)+x) = 0,
\end{equation}
where $D(s,u)+x:= \{(y,t) \in \rd \times \rpluszero: t^2 > \norm{y-x}^2-2u\norm{y-x}+s^2\}$. Furthermore
\begin{align}
\begin{split} \label{VolumeEmptySet}
    (\lambda_d\otimes\markdist) (D(s,u))  \, = v_d \int_{\sqrt{(s^2-u^2)^+}}^\infty &\left(u+\sqrt{u^2-s^2+t^2}\right)^d \dx \markdist(t) 
    \\ &- v_d \int_{\sqrt{(s^2-u^2)^+}}^s \left(u-\sqrt{u^2-s^2+t^2}\right)^d \dx \markdist(t),
    \end{split}
\end{align}
where $(s^2-u^2)^+ := \max\{0,s^2-u^2\}$. 
\end{lemma}

\begin{proof}
 To prove \eqref{equivalence:inradius/emptyset}, first define, for $\bfx=(x,s)\in\rd \times \rpluszero$ and $\bfy=(y,t)\in\rd \times \rpluszero$,
\begin{align*}
        Ra_{\bfx,\bfy}&:= \{z \in \rd: \|z-x\|^2-s^2 = \|z-y\|^2-t^2 \},\\
        H_{\bfx \leq \bfy}&:= \{z \in \rd: \|z-x\|^2-s^2 \le \|z-y\|^2-t^2 \}.
\end{align*}
It can be shown easily that $Ra_{\bfx,\bfy}$ is a hyperplane and $H_{\bfx \leq \bfy}$ is a closed half-space in $\rd$. Furthermore, $Ra_{(x,s),(y,t)}$ is perpendicular to the line passing through $x$ and $y$ (see \cite[Section 3]{ar:LZ08}).

For $(x,s)\in\varphi$ and $u>0$ it follows from the definition of the inradius and the Laguerre cell that
\begin{equation*}
r((x,s),\varphi) \geq u \iff  B^d(x,u) \subseteq L((x,s),\varphi) = \bigcap_{(y,t)\in \varphi \setminus \{(x,s)\}} H_{(x,s)\leq(y,t)}.    
\end{equation*}
The orthogonality of $Ra_{(x,s),(y,t)}$ and the line between $x$ and $y$ implies that
$$
B^d(x,u) \subseteq H_{(x,s)\leq(y,t)} \iff x+\frac{u}{\norm{y-x}}(y-x) \in H_{(x,s)\leq(y,t)}.
$$
Altogether we get that 
\begin{align*}
& r((x,s),\varphi) \geq u \\  &\iff  \bigg\|x+\frac{u}{\norm{y-x}}(y-x) - x\bigg\|^2-s^2 \leq \bigg\|x+\frac{u}{\norm{y-x}}(y-x)-y\bigg\|^2-t^2  \text{ for all } (y,t) \in \varphi\\
 &\iff  u^2-s^2 \leq (u-\|y-x\|)^2-t^2 \text{ for all } (y,t) \in \varphi\\
&\iff t^2 \leq \norm{y-x}^2 -2u\norm{y-x}+s^2 \text{ for all } (y,t) \in \varphi,
\end{align*}
which shows \eqref{equivalence:inradius/emptyset}.

Note that, for $(y,t)\in\rd\times[0,\infty)$,
\begin{align*}
& (y,t)\in D(s,u) \\
& \quad \Longleftrightarrow \quad t^2 > \|y\|^2-2 u\|y\|+s^2  \quad \Longleftrightarrow \quad t^2-s^2+u^2>(\|y\|-u)^2 \\
& \quad \Longleftrightarrow \quad u^2-s^2+t^2\geq 0 \quad \text{and} \quad \|y\|\in\big(u-\sqrt{u^2-s^2+t^2},u+\sqrt{u^2-s^2+t^2}\big) \\
& \quad \Longleftrightarrow \quad t\geq \sqrt{(s^2-u^2)^+} \quad \text{and} \quad \|y\|\in\big(u-\sqrt{u^2-s^2+t^2},u+\sqrt{u^2-s^2+t^2}\big)\cap[0,\infty).
\end{align*}
This yields that
\begin{align*}
& (\lambda_d\otimes\markdist) (D(s,u)) \\
& = \int_0^\infty \int_{\rd} \ind{t\ge \sqrt{(s^2-u^2)^+}} \\
& \hspace{1.2cm} \times \ind{\|y\|\in\big(u-\sqrt{u^2-s^2+t^2},u+\sqrt{u^2-s^2+t^2}\big)\cap[0,\infty)} \, \dx y \, \dx \markdist(t) \\
& = \int_{\sqrt{(s^2-u^2)^+}}^\infty \lambda_d\big(B^d\big(0,u+\sqrt{u^2-s^2+t^2}\big)\big) 
- \lambda_d\big(B^d\big(0,\max\big\{0,u-\sqrt{u^2-s^2+t^2}\big\}\big)\big) \, \dx \markdist(t) \\
& = v_d \int_{\sqrt{(s^2-u^2)^+}}^\infty \big(u+\sqrt{u^2-s^2+t^2}\big)^d \, \dx \markdist(t) -  v_d \int_{\sqrt{(s^2-u^2)^+}}^s \big(u-\sqrt{u^2-s^2+t^2}\big)^d \, \dx \markdist(t),
\end{align*}
where we used in the last step that the volume of the ball that is subtracted becomes zero for $t>s$. This proves \eqref{VolumeEmptySet}.
\end{proof}

\section{The bounded case} \label{sec:ProofBC}

The goal of this section is to prove Theorem \ref{thm:bounded}. We start with a computation for the intensity measure in Section \ref{sec:intensity_measure}.  Section \ref{subsec:ProofExistenceCorResc} contains the proof of the existence of a correct shift and its order as well as additional derivations about the convergence of the intensity measure of the point process $\xi_\rho$. Finally the proofs of Theorem \ref{thm:bounded} and Corollary \ref{cor:bounded} can be found in Section \ref{subsec:ProofConvBounded}. We finish with some explicit formulas for correct shifts for dimension $d = 3$ and one example of a correct shift in dimension $d = 4$ in Section \ref{subsec:ProofsExamples}.

\subsection{Computation of the intensity measure}\label{sec:intensity_measure}

We combine Lemma \ref{lemma:EmptySetFormula} and the Mecke formula to derive the following integral formula, which will allow us to compute the intensity measure of $\xi_\rho$.

\begin{lemma} \label{lem:ExpectationFormula} 
    Suppose that the assumptions of Theorem \ref{thm:bounded} are satisfied and let $B \subseteq W$ and $D \subseteq [0,\B]$ be measurable sets. 
    \begin{itemize}
    \item [a)] For $\Tilde{u} > \B$, one has 
    \begin{align*}
        \e \sum_{\bfx \in \ppp \cap W_\rho \times [0,\B]} &\ind{r(\bfx,\ppp)\geq \Tilde{u}} \ind{\bfx \in \rho B \times D} \\
        &= \gamma \rho^d \lambda_d(B)\int_D \exp\bigg(-v_d \gamma \int_0^\B\Big(\Tilde{u} + \sqrt{\Tilde{u}^2-s^2+t^2}\Big)^d \, \dx \markdist(t)\\ 
        & \hspace{4cm}+v_d \gamma \int_0^s \Big(\Tilde{u} - \sqrt{\Tilde{u}^2-s^2+t^2}\Big)^d \, \dx \markdist(t)\bigg) \, \dx \markdist (s).
    \end{align*}
    \item [b)] Let $\{\Tilde{u}_{\rho}\}_{\rho\ge 1}$ be a family of positive real numbers such that $\lim_{\rho\to\infty} \Tilde{u}_\rho=\infty$. Then
    \begin{align*}
        \lim_{\rho\to\infty} \e &\sum_{\bfx \in \ppp \cap W_\rho \times [0,\B]} \ind{r(\bfx,\ppp)\geq \Tilde{u}_\rho} \ind{\bfx \in \rho B \times D} \\
        &= \lim_{\rho\to\infty} \gamma \rho^d \lambda_d(B)\int_D \exp\bigg(-v_d \gamma \int_0^\B \Big(\Tilde{u}_\rho + \sqrt{\Tilde{u}_\rho^2-s^2+t^2}\Big)^d \, \dx \markdist(t) \bigg) \, \dx \markdist (s)
    \end{align*}
    if the limit on the right-hand side is well defined.
    \end{itemize}
\end{lemma}

\begin{proof}
     It follows from the Mecke formula (see, for example, \cite[Theorem 4.1]{bo:LP17}), 
   translation invariance of $\eta$ and Lemma \ref{lemma:EmptySetFormula}, where due to $\Tilde{u}> \B$ the lower limits of the integrals become 0, that 
   \begin{align*}
         \e \sum_{\bfx \in \ppp \cap W_\rho \times [0,\B]} &\ind{r(\bfx,\ppp)\geq \Tilde{u}} \ind{\bfx \in \rho B \times D} \\
         & = \gamma \int_{\rho B} \int_D \mathbb{P}(r((x,s),\ppp+\delta_{(x,s)})\geq \Tilde{u}) \, \dx \markdist(s) \, \dx x \\
         &= \gamma \lambda_d(\rho B) \int_D \mathbb{P}(r((0,s),\ppp+\delta_{(0,s)})\geq \Tilde{u}) \, \dx \markdist(s) \\
         &= \gamma \lambda_d(\rho B) \int_D \exp(-\gamma (\lambda_d\otimes \markdist)(D(s,\Tilde{u}))) \, \dx \markdist(s) \\
        &= \gamma \rho^d \lambda_d(B) \int_D \exp\bigg(-v_d \gamma \int_0^\B \Big(\Tilde{u} + \sqrt{\Tilde{u}^2-s^2+t^2}\Big)^d \, \dx \markdist(t) \\
        & \hspace{4cm}+v_d \gamma \int_0^s\Big(\Tilde{u} - \sqrt{\Tilde{u}^2-s^2+t^2}\Big)^d \, \dx \markdist(t)\bigg) \, \dx \markdist (s),
     \end{align*}
which is part a). For the situation of part b) one has, for all $s,t\in[0,\B]$ with $t\le s$,
\begin{equation} \label{bound_second_integral}
\big|\Tilde{u}_\rho - \sqrt{\Tilde{u}_\rho^2-s^2+t^2}\big| = \frac{\Tilde{u}_\rho^2 - (\Tilde{u}_\rho^2-s^2+t^2)}{\Tilde{u}_\rho + \sqrt{\Tilde{u}_\rho^2-s^2+t^2}} \leq \frac{s^2-t^2}{2\sqrt{\Tilde{u}_\rho^2-s^2+t^2}}\leq \frac{\B^2}{2\sqrt{\Tilde{u}_\rho^2-\B^2}}
\end{equation}
if $\Tilde{u}_\rho>\B$. Using the abbreviation
$$
J_\rho:=\gamma \rho^d \lambda_d(B)\int_{D} \exp\bigg(-v_d \gamma \int_0^\B \Big(\Tilde{u}_\rho + \sqrt{\Tilde{u}_\rho^2-s^2+t^2}\Big)^d \, \dx \markdist(t)\bigg) \, \dx \markdist (s),
$$
the equality from a) and \eqref{bound_second_integral} lead to
$$
\exp\Bigg( -\frac{v_d\gamma\B^{2d}}{2^d(\Tilde{u}_\rho^2-\B^2)^{d/2}} \Bigg) J_\rho \leq \e \sum_{\bfx \in \ppp \cap \rho B \times D} \ind{r(\bfx,\ppp)\geq \Tilde{u}_\rho} \leq \exp\Bigg( \frac{v_d\gamma\B^{2d}}{2^d(\Tilde{u}_\rho^2-\B^2)^{d/2}} \Bigg) J_\rho
$$
for $\Tilde{u}_\rho>\B$. Now the observation that the arguments of the exponential functions vanish as $\rho\to\infty$ completes the proof of part b).
\end{proof}

\subsection{Existence of a correct shift and convergence of the intensity measure} \label{subsec:ProofExistenceCorResc}

A crucial part of Theorem \ref{thm:bounded} is that one can choose a correct shift such that \eqref{defCR} is satisfied. In the following main result of this section, we establish that there exists at least one correct shift and furthermore that it behaves as $\log \rho^d$ as $\rho\to\infty$.

\begin{prop}\label{prop:ExistenceCorResc}
    Let the assumptions of Theorem \ref{thm:bounded} be satisfied. Then there exists a correct shift $\{\sdrho\}_{\rho \geq 1}$ such that $\lim_{\rho \rightarrow \infty} \frac{\sdrho}{\log \rho^d}= 1$. Furthermore for  $d =2$ the correct shift can be chosen as in \eqref{CRtwo} and for $d = 3$ as in \eqref{CRthree}.
\end{prop}

The next corollary ensures that all correct shifts are asymptotically equivalent and behave as $\log \rho^d$ as $\rho\to\infty$.

\begin{cor} \label{cor:OrderofCR}
Let the assumptions of Theorem \ref{thm:bounded} prevail.
\begin{itemize}
\item [a)] For two correct shifts $\{\sdrho\}_{\rho \geq 1}$ and $\{\tilde{\mathrm{s}}_{d,\rho}\}_{\rho \geq 1}$ one has $\lim_{\rho \rightarrow \infty} \sdrho-\tilde{\mathrm{s}}_{d,\rho} = 0$.
\item [b)] Any correct shift $\{\sdrho\}_{\rho \geq 1}$ satisfies $\lim_{\rho \rightarrow \infty} \frac{\sdrho}{\log \rho^d}= 1$ and, thus, $\lim_{\rho\to\infty} \sdrho=\infty$.
\end{itemize}
\end{cor}

\begin{proof}
We start with proving part a) by contradiction. Assume that there exists a sequence $\{\rho_n\}_{n\in\mathbb{N}}$ with $\rho_n\geq 1$ for $n\in\mathbb{N}$ and $\lim_{n\to\infty}\rho_n=\infty$ and an $\varepsilon>0$ such that $\mathrm{s}_{d,\rho_n} - \tilde{\mathrm{s}}_{d,\rho_n}\geq \varepsilon$ for all $n\in\mathbb{N}$. Then it follows from the definition of a correct shift that
\begin{align*}
1 & = \lim_{n\to\infty} \e \sum_{\bfx \in \ppp \cap W_\rho \times [0,\B]} \ind{r(\bfx,\ppp)^d\geq \frac{\mathrm{s}_{d,\rho_n}}{c_d\gamma}} \\
& \leq \lim_{n\to\infty} \e \sum_{\bfx \in \ppp \cap W_\rho \times [0,\B]} \ind{r(\bfx,\ppp)^d\geq \frac{\tilde{\mathrm{s}}_{d,\rho_n}+\varepsilon}{c_d\gamma}} =\exp(-\varepsilon),
\end{align*}
which is a contradiction. Treating the case $\tilde{\mathrm{s}}_{d,\rho_n} - \mathrm{s}_{d,\rho_n} \geq \varepsilon$ analogously shows a). Part b) is now an immediate consequence of a) and Proposition \ref{prop:ExistenceCorResc}.
\end{proof}

Using similar arguments as in the proof of the previous corollary, one sees that the sum of a correct shift and $\{\alpha_\rho\}_{\rho \geq 1}$ converging to $0$ is again a correct shift, which means that we do not have uniqueness.

The proof of Proposition \ref{prop:ExistenceCorResc} is prepared by the following lemmas. Recall the notation $c_d = 2^dv_d$ and furthermore define
\begin{equation*}
    h_d(z):= \Big(1+\sqrt{1+z}\Big)^d
\end{equation*}
for $z \in [-1,\infty)$. For $i \in \n_0$ we use the abbreviation $c_{i,d}:= h_d^{(i)}(0)$ for the value of the $i$-th derivative of the function $h_d$ at $z=0$.

\begin{lemma} \label{lemma:AlternativeFormula}
     Let the assumptions of Theorem \ref{thm:bounded} prevail, let D  $\subseteq [0,\B]$ be measurable and let $\{\ubarrho\}_{\rho \geq 1}$ be a family of real numbers such that $\lim_{\rho\to\infty} \ubarrho=\infty$. Then 
    \begin{align}
    \begin{split} \label{limit:SumFormula}
        \lim_{\rho\to\infty} \gamma \rho^d &\int_D \exp\bigg(-v_d \gamma \int_0^\B\Big(\ubarrho + \sqrt{\ubarrho^2-s^2+t^2}\Big)^d \, \dx \markdist(t) \bigg) \, \dx \markdist (s) \\
& = \lim_{\rho\to\infty} \gamma \rho^d \int_D \exp\bigg(-v_d \gamma \int_0^\B \sum_{i=0}^{\lfloor \frac{d}{2} \rfloor} \frac{c_{i,d}}{i!} \ubarrho^{d-2i} (t^2-s^2)^{i} \, \dx \markdist(t) \bigg) \, \dx \markdist (s)
\end{split}
    \end{align}
if the limit on at least one side exists. 
\end{lemma}

\begin{proof}
    For $\rho\ge 1$ large enough such that $\ubarrho \geq \B$ we define the function
\begin{equation*}
h_{\rho,d}(z):=\Big( \ubarrho + \sqrt{\ubarrho^2+z} \Big)^d
\end{equation*}
for $z \in [-\B^2,\B^2]$. Taylor expansion leads to
\begin{equation} \label{formula:Taylor}
h_{\rho,d}(z) = \sum_{i=0}^{\lfloor \frac{d}{2} \rfloor} \frac{h_{\rho,d}^{(i)}(0)}{i!} z^{i} + \frac{h_{\rho,d}^{(\lfloor \frac{d}{2} \rfloor+1)}(\tilde{z})}{(\lfloor \frac{d}{2} \rfloor+1)!} z^{\lfloor \frac{d}{2} \rfloor+1}    
\end{equation}
for some $\tilde{z}$ between $0$ and $z$. Obviously, we have
\begin{equation*}
h_{\rho,d}(z)=\Big( \ubarrho + \sqrt{\ubarrho^2+z} \Big)^d = \ubarrho^d \Big( 1 + \sqrt{1+z/\ubarrho^2} \Big)^d = \ubarrho^d h_{d}(z/\ubarrho^2)
\end{equation*}
so that, for $i\in\n_0$, $h_{\rho,d}^{(i)}(z) = \ubarrho^{d-2i} h_{d}^{(i)}(z/\ubarrho^2).$ Thus \eqref{formula:Taylor} can be rewritten as
\begin{equation*}
\Big( \ubarrho + \sqrt{\ubarrho^2+z} \Big)^d = \sum_{i=0}^{\lfloor \frac{d}{2} \rfloor} \frac{c_{i,d}}{i!} \ubarrho^{d-2i} z^{i} + R_\rho(z),
\end{equation*}
where $R_\rho(z):= \frac{h_{d}^{(\lfloor \frac{d}{2} \rfloor+1)}(\tilde{z}/\ubarrho^2)}{(\lfloor \frac{d}{2} \rfloor+1)!} \ubarrho^{d-2\lfloor \frac{d}{2} \rfloor-2} z^{\lfloor \frac{d}{2} \rfloor+1}$ is the remainder from the Taylor expansion. Furthermore, there exist real numbers $\{r_\rho\}_{\rho\ge 1}$ such that $|R_\rho(z)|\le r_\rho$ for all $z\in[-\B^2,\B^2]$ and $\rho\ge1$ and $\lim_{\rho\to\infty}r_\rho=0$. Since one has
$$
\bigg| \Big(\ubarrho + \sqrt{\ubarrho^2-s^2+t^2}\Big)^d - \sum_{i=0}^{\lfloor \frac{d}{2} \rfloor} \frac{c_{i,d}}{i!} \ubarrho^{d-2i} (t^2-s^2)^{i} \bigg|\le r_\rho
$$
for all $s,t\in[0,\B]$ and $\rho\ge 1$ with $\ubarrho\ge\B$, this proves \eqref{limit:SumFormula}.
\end{proof}

 For $\rho \geq 1$ and $d \geq 2$ define the function $G_{d,\rho}: \left[-\log \rho^d,\infty\right) \rightarrow [0,\infty)$ by 
\begin{equation} \label{def:Gfunction}
    G_{d,\rho}(x) := \gamma \rho^d \int_0^\B \exp\bigg(-v_d \gamma \int_0^\B \sum_{i=0}^{\lfloor \frac{d}{2} \rfloor} \frac{c_{i,d}}{i!} 0_\rho(x)^{d-2i} (t^2-s^2)^{i} \, \dx \markdist(t) \bigg) \, \dx \markdist (s)
\end{equation}
    with
\begin{equation}\label{def:0rhox}
0_\rho(x):= \bigg(\frac{\log \rho^d + x}{c_d\gamma }\bigg)^{1/d}.    
\end{equation}
Because of $c_{0,d}=2^d$ and $c_d=2^dv_d$ we obtain
\begin{equation}\label{def:Gfunction_simplified} 
G_{d,\rho}(x) = \gamma \int_0^\B \exp\bigg(-x-v_d \gamma \int_0^\B \sum_{i=1}^{\lfloor \frac{d}{2} \rfloor} \frac{c_{i,d}}{i!} 0_\rho(x)^{d-2i} (t^2-s^2)^{i} \, \dx \markdist(t) \bigg) \, \dx \markdist (s).
\end{equation}
Introducing the notation 
\begin{align*}
\alpha_{i,d}:= \frac{v_d \gamma c_{i,d}}{i!(c_d \gamma)^{(d-2i)/d}} \quad \text{and} \quad
    \beta_{i,d}(s):= \alpha_{i,d} \int_0^{\B} (t^2-s^2)^i \dx \markdist(t) 
\end{align*}
for $i \in \{1,\dots,\lfloor \frac{d}{2} \rfloor\}$ and $s \in [0,\B]$, we can rewrite 
\begin{align}
     G_{d,\rho}(x) & = \gamma \int_0^\B \exp\bigg(-x - \sum_{i=1}^{\lfloor \frac{d}{2} \rfloor} \beta_{i,d}(s) (\log \rho^d + x)^{(d-2i)/d} \bigg) \, \dx \markdist (s). \label{eqn:representation_Gdrho}
\end{align}
Furthermore, let $\hat{\rho}_{d}\in [1,\infty)$ be such that
\begin{equation}\label{eqn:bound_sum_exponent}
    \sum_{i=1}^{\lfloor \frac{d}{2} \rfloor}\abs{\alpha_{i,d}} \B^{2i} \frac{1}{(\log \rho^d)^{2i/d}} < \frac{1}{4}
\end{equation}
for all $\rho \geq \hat{\rho}_{d}$. Note that such $\hat{\rho}_{d}$ exists and depends only on $d,\markdist$ and $\gamma$. The upper bound $\frac{1}{4}$ is arbitrarily chosen - in fact we could have taken any real number in $(0,1/2)$.

\begin{lemma}\label{lem:rdrho}
Let the assumptions of Theorem \ref{thm:bounded} be satisfied.
\begin{itemize}
\item [a)] For every $\rho \geq \max\{\hat{\rho}_d,(1/\gamma)^{4/d}\}$ there exists a unique $\rdrho \in \left(-\frac{1}{2}\log \rho^d,\infty\right)$ such that it satisfies $G_{d,\rho}(\rdrho)= 1$.
\item [b)] Furthermore, one has
$$
\lim_{\rho\to\infty} \frac{\rdrho}{\log \rho^d}=0.
$$
\end{itemize}
\end{lemma}

\begin{proof}
For the proof of a) let $\rho$ be as assumed. Then, using \eqref{eqn:representation_Gdrho}, the bounds $\abs{\beta_{i,d}(s)} \leq \abs{\alpha_{i,d}} \B^{2i}$ for $i \in \{1,\dots \lfloor \frac{d}{2} \rfloor\}$ and \eqref{eqn:bound_sum_exponent}, we obtain
    \begin{align*}
        G_{d,\rho}\bigg(-\frac{1}{2}\log \rho^d\bigg) &= \gamma \int_0^\B \exp\bigg(\frac{1}{2}\log \rho^d - \sum_{i=1}^{\lfloor \frac{d}{2} \rfloor}\beta_{i,d}(s) \Big(\frac{1}{2}\log \rho^d\Big)^{(d-2i)/d}  \bigg) \, \dx \markdist (s)\\
        &\geq \gamma \exp\bigg(\frac{1}{2}\log \rho^d - \log \rho^d\sum_{i=1}^{\lfloor \frac{d}{2} \rfloor}\abs{\alpha_{i,d}} \B^{2i} (\log \rho^d)^{-2i/d}  \bigg) \\
        &> \gamma \exp\bigg(\frac{1}{4}\log \rho^d  \bigg) \geq 1. 
    \end{align*} 
Furthermore, $G_{d,\rho}$ is continuous and  $\lim_{x\rightarrow\infty}G_{d,\rho}(x)= 0$ since $-x$ is the leading term in \eqref{eqn:representation_Gdrho}. Computing the derivative and using the same arguments as for the previous display, one derives for $x \geq -\frac{1}{2}\log \rho^d$,
\begin{align*}
    \frac{\partial G_{d,\rho}}{\partial x}(x) &= \gamma \int_0^\B \bigg(- 1 - \sum_{i=1}^{\lfloor \frac{d}{2} \rfloor} \frac{(d-2i)\beta_{i,d}(s)}{d(\log \rho^d + x)^{2i/d}}\bigg)\exp\bigg(-x - \sum_{i=1}^{\lfloor \frac{d}{2} \rfloor} \beta_{i,d}(s) (\log \rho^d + x)^{(d-2i)/d} \bigg) \, \dx \markdist (s) \\
    &\leq  \bigg(-1 + \sum_{i=1}^{\lfloor \frac{d}{2} \rfloor} \frac{2\abs{\alpha_{i,d}} \B^{2i}}{(\log \rho^d)^{2i/d}}  \bigg) \cdot G_{d,\rho}(x) < 0,
\end{align*}
leading to $G_{d,\rho}$ being strictly decreasing on $[-\frac{1}{2}\log \rho^d, \infty)$, which finishes the proof of part a).

For the proof of part b) we fix some $\varepsilon\in(0,\frac{1}{2})$. From \eqref{eqn:representation_Gdrho} it follows that
$$
\lim_{\rho\to\infty} G_{d,\rho}(-\varepsilon \log \rho^d) = \infty \quad \text{and} \quad \lim_{\rho\to\infty} G_{d,\rho}(\varepsilon \log \rho^d) = 0.
$$
Therefore, there exists a $\rho_0\geq \max\{\hat{\rho}_d,(1/\gamma)^{4/d}\}$ such that, for all $\rho\geq\rho_0$,
\begin{equation}\label{eqn:inequalities_Gdrho}
G_{d,\rho}(\varepsilon \log \rho^d) < 1 < G_{d,\rho}(-\varepsilon \log \rho^d).
\end{equation}
Since $G_{d,\rho}(\rdrho)=1$ and $G_{d,\rho}$ is strictly decreasing on $[-\frac{1}{2}\log \rho^d, \infty)$, \eqref{eqn:inequalities_Gdrho} yields
$$
-\varepsilon \log \rho^d < \rdrho < \varepsilon \log \rho^d
$$
for $\rho\geq\rho_0$. As $\varepsilon$ can be chosen arbitrarily small, this shows part b). 
\end{proof}

\begin{lemma} \label{lemma:SolutionForAllu}
    Let the assumptions of Theorem \ref{thm:bounded} be satisfied and let $\{\mathrm{t}_{d,\rho}\}_{\rho \geq 1}$ be a family of real numbers such that  $ \lim_{\rho \rightarrow \infty} G_{d,\rho}(\mathrm{t}_{d,\rho}) = 1$ and $\lim_{\rho \rightarrow \infty} \log \rho^d + \mathrm{t}_{d,\rho} = \infty$. Then, for any $u \in \R$,
    \begin{equation}\label{eqn:limit_Gdrho}
        \lim_{\rho \rightarrow \infty} G_{d,\rho}(\mathrm{t}_{d,\rho}+u) = e^{-u}
    \end{equation}
and $\{\log \rho^d + \mathrm{t}_{d,\rho}\}_{\rho\ge1}$ is a correct shift.
\end{lemma}
\begin{proof}
    Let $u \in \R$ and recall from \eqref{eqn:representation_Gdrho} that, for $\rho$ sufficiently large,
    \begin{align*}
        G_{d,\rho}(\mathrm{t}_{d,\rho}+u) =  \gamma \int_0^\B \exp\bigg(-\mathrm{t}_{d,\rho} -u - \sum_{i=1}^{\lfloor \frac{d}{2} \rfloor} \beta_{i,d}(s) (\log \rho^d + \mathrm{t}_{d,\rho}+u)^{(d-2i)/d} \bigg) \, \dx \markdist (s).
    \end{align*}
    For $i \in \{1,\dots, \lfloor \frac{d}{2} \rfloor\}$, by the mean value theorem and $\abs{\beta_{i,d}(s)} \leq \abs{\alpha_{i,d}} \B^{2i}$ we get
    \begin{align*}
        \abs{\beta_{i,d}(s)(\log \rho^d + \mathrm{t}_{d,\rho})^{(d-2i)/d} - \beta_{i,d}(s)(\log \rho^d + \mathrm{t}_{d,\rho}+u)^{(d-2i)/d}}\leq \frac{\abs{\alpha_{i,d}}\B^{2i} \abs{u}}{\abs{\log\rho^d + \mathrm{t}_{d,\rho}-\abs{u}}^{\frac{2i}{d}}}
    \end{align*}
    for $\rho$ large enough and $s \in [0,\B]$. Since
   $$
   \Delta_\rho(u):= \sum_{i=1}^{\lfloor \frac{d}{2}\rfloor} \frac{\abs{\alpha_{i,d}}\B^{2i} \abs{u}}{\abs{\log\rho^d + \mathrm{t}_{d,\rho}-\abs{u}}^{\frac{2i}{d}}} \underset{\rho \rightarrow \infty}{\longrightarrow} 0,
   $$ 
   and we can bound
   \begin{align*}
       e^{-u -\Delta_\rho(u)} G_{d,\rho}(\mathrm{t}_{d,\rho})\leq G_{d,\rho}(\mathrm{t}_{d,\rho}+u) \leq e^{-u +\Delta_\rho(u)} G_{d,\rho}(\mathrm{t}_{d,\rho}),
   \end{align*}
    we obtain \eqref{eqn:limit_Gdrho}.

In the following we use the notation $0_\rho$ defined in \eqref{def:0rhox}. For $u\in\R$ it follows from Lemma \ref{lem:ExpectationFormula} b) and Lemma \ref{lemma:AlternativeFormula} with $D = [0,\B]$ and $B=W$, the definition of $G_{d,\rho}$ in \eqref{def:Gfunction} and \eqref{eqn:limit_Gdrho} that
\begin{align*}
& \lim_{\rho \rightarrow \infty} \e \sum_{\bfx \in \ppp \cap W_\rho \times [0,\B]} \ind{r(\bfx,\ppp)^d\geq \frac{u + \log \rho^d + \mathrm{t}_{d,\rho}}{c_d \gamma}} \\
& = \lim_{\rho\to\infty} \gamma \rho^d \int_0^{\B} \exp\bigg(-v_d \gamma \int_0^\B \Big(0_\rho(\mathrm{t}_{d,\rho}+u) + \sqrt{0_\rho(\mathrm{t}_{d,\rho}+u)^2-s^2+t^2}\Big)^d \, \dx \markdist(t) \bigg) \, \dx \markdist (s) \\
& = \lim_{\rho\to\infty} \gamma \rho^d \int_0^{\B} \exp\bigg(-v_d \gamma \int_0^\B \sum_{i=0}^{\lfloor \frac{d}{2} \rfloor} \frac{c_{i,d}}{i!} 0_\rho(\mathrm{t}_{d,\rho}+u)^{d-2i} (t^2-s^2)^{i} \, \dx \markdist(t) \bigg) \, \dx \markdist (s) \\
& = \lim_{\rho\to\infty} G_{d,\rho}(\mathrm{t}_{d,\rho}+u)=e^{-u},
\end{align*}
which means that $\{\log\rho^d + \mathrm{t}_{d,\rho}\}_{\rho\ge 1}$ is a correct shift.
\end{proof}

Now we are ready to prove Proposition \ref{prop:ExistenceCorResc}. 

\begin{proof}[Proof of Proposition \ref{prop:ExistenceCorResc}]
We choose our candidate for the correct shift as 
$$
\sdrho := \begin{cases} \log \rho^d, & \quad  1\leq \rho \leq \max\{\hat{\rho}_d,\frac{1}{\gamma^{4/d}}\},\\ \log \rho^d + \rdrho, & \quad \rho > \max\{\hat{\rho}_d,\frac{1}{\gamma^{4/d}}\}, \end{cases}
$$
with $\rdrho$ as in Lemma \ref{lem:rdrho} a). Since, by Lemma \ref{lem:rdrho}, $\lim_{\rho\to\infty} G_{d,\rho}(\rdrho)=1$ and $\lim_{\rho\to\infty}\log\rho^d + \rdrho = \infty$, it follows from Lemma \ref{lemma:SolutionForAllu} that $\sdrho$ is a correct shift.

It remains to see that the choices \eqref{CRtwo} and \eqref{CRthree} yield  correct shifts for $d = 2$ and $d=3$. If $d \in \{2,3\}$, then the formula for $G_{d,\rho}$ simplifies to
\begin{equation}\label{FctionGd23}
    G_{d,\rho}(x) = \gamma\int_0^\B \exp\bigg(- x -  v_d \gamma \int_0^{\B} c_{1,d} \bigg(\frac{\log \rho^d+x}{c_d \gamma}\bigg)^{(d-2)/d} (t^2-s^2) \dx \markdist(t)  \bigg) \, \dx \markdist (s).
\end{equation}
For $d \in \{2,3\}$ define
\begin{align*}
    \dot{\mathrm{r}}_{d,\rho} &:=  \log\bigg( \int_0^\B \exp\bigg(-v_d \gamma \int_0^\B c_{1,d} \bigg(\frac{\log \rho^d}{c_d\gamma}\bigg)^{\frac{d-2}{d}} (t^2-s^2) \, \dx \markdist(t) \bigg) \, \dx \markdist (s) \bigg) + \log\gamma,
\end{align*}
where one can easily compute that $c_{1,2} = 2$ and $c_{1,3} = 6$ and hence the formula in \eqref{CRtwo} is equal to $\log \rho^2+\dot{\mathrm{r}}_{2,\rho}$ and the formula in \eqref{CRthree} to $\log\rho^3+\dot{\mathrm{r}}_{3,\rho}$. By Lemma \ref{lemma:SolutionForAllu} it is enough to show that $\lim_{\rho \rightarrow \infty} G_{d,\rho}(\dot{\mathrm{r}}_{d,\rho}) = 1$.

For $d = 2$, we immediately see that in fact $\dot{\mathrm{r}}_{2,\rho} = \mathrm{r}_{2,\rho}$, or in other words $G_{2,\rho}(\dot{\mathrm{r}}_{2,\rho})=1$. 

For $d=3$, we must do one more approximation step. It is easy to see that there exists a constant $C_1>0$ depending on $\gamma$ and $\B$ such that 
$\abs{\dot{\mathrm{r}}_{3,\rho}}\leq C_1 (\log \rho^3)^{1/3}$. Therefore we can bound for $\rho$ large enough
\begin{equation} \label{boundd3}
    \abs{(\log \rho^3+\dot{\mathrm{r}}_{3,\rho})^{1/3}-(\log \rho^3)^{1/3}} \leq C_2 (\log \rho^3)^{-1/3}
\end{equation}
for a suitable constant $C_2>0$ not dependent on $\rho$. This yields
\begin{align*}
    \lim_{\rho \rightarrow \infty}& G_{3,\rho}(\dot{\mathrm{r}}_{3,\rho}) \\
    &= \lim_{\rho \rightarrow \infty} 
    \gamma\int_0^\B \exp\bigg(- \dot{\mathrm{r}}_{3,\rho} -  v_3 \gamma \int_0^{\B} c_{1,3} \bigg(\frac{\log \rho^d+\dot{\mathrm{r}}_{3,\rho}}{c_3 \gamma}\bigg)^{1/3} (t^2-s^2) \dx \markdist(t)  \bigg) \, \dx \markdist (s)\\
    &=\lim_{\rho \rightarrow \infty}  \exp(- \dot{\mathrm{r}}_{3,\rho} + \log \gamma )  \int_0^\B \exp\bigg(- v_3 \gamma \int_0^{\B} c_{1,3} \bigg(\frac{\log \rho^3}{c_3 \gamma}\bigg)^{1/3} (t^2-s^2) \dx \markdist(t)  \bigg) \, \dx \markdist (s) = 1,
\end{align*}
which finishes the proof for $d=3$.
\end{proof}

The key ingredients to obtain the explicit formulas for correct shifts in dimensions $d \in\{2,3\}$ are the formula \eqref{FctionGd23} and the bound \eqref{boundd3}. Particularly, instead of looking for an exact solution of $G_{3,\rho}(x)=1$, we approximate the term $(\log \rho^d+x)^{1/3}$ by $(\log \rho^d)^{1/3}$ and look for the solution of the corresponding approximated equation. This approach cannot be carried out in higher dimensions since the order of $\mathrm{r}_{d,\rho}:=\log G_{d,\rho}(0)$ becomes so large that the terms in the sum in \eqref{def:Gfunction_simplified} for $x=\mathrm{r}_{d,\rho}$ cannot be approximated by those for $x=0$.

\begin{lemma}\label{lem:intensity_measure_2}
Let $d=2$ and let the assumptions of Theorem \ref{thm:bounded} be satisfied. Then, for any correct shift $\{\stworho\}_{\rho \geq 1}$, one has
\begin{equation}\label{eqn:limit_measure_two}
\lim_{\rho\to\infty}\e \sum_{\bfx \in \ppp \cap W_\rho \times [0,\B]} \ind{r(\bfx,\ppp)^2\geq \frac{u + \stworho}{c_2 \gamma}} \ind{\bfx \in \rho B \times D} = e^{-u}\lambda_2(B) \markdistlimtwo(D)
\end{equation}
for all $u\in\R$, measurable sets $B\subseteq W$ and $D\subseteq[0,\B]$.
\end{lemma}

\begin{proof}
To prove that \eqref{eqn:limit_measure_two} holds for a general correct shift $\stworho$, recall that according to Corollary \ref{cor:OrderofCR} it holds that $\stworho = \dot{\mathrm{s}}_{2,\rho} + \alpha_\rho$, where $\lim_{\rho \rightarrow \infty} \alpha_\rho = 0$ and $\dot{\mathrm{s}}_{2,\rho}$ is the choice of correct shift from \eqref{CRtwo}. Throughout the proof, we use the abbreviation $u_\rho:=\sqrt{\max\big\{\frac{u + \stworho}{c_2\gamma},0\big\}}$ for $\rho\ge 1$.

Using Lemma \ref{lem:ExpectationFormula} b),  Lemma \ref{lemma:AlternativeFormula} with $c_{0,2} = 2^2$ and $c_{1,2}=2$ and 
$$
e^{-\stworho} = e^{-\dot{\mathrm{s}}_{2,\rho}} e^{-\alpha_\rho} = \frac{\exp(-\alpha_r + 2v_2\gamma \e M^2)}{\gamma \rho^d \e e^{2v_2\gamma M^2}},
$$
it follows that 
\begin{align*}
    \begin{split}
        \lim_{\rho\to\infty} &\e \sum_{\bfx \in \ppp \cap W_\rho \times [0,\B]} \ind{r(\bfx,\ppp)^2\geq \frac{u + \stworho}{c_2 \gamma}} \ind{\bfx \in \rho B \times D} \\
        &= \lambda_2(B)\lim_{\rho\to\infty} \gamma \rho^2 \int_D \exp\bigg(-v_2 \gamma \int_0^\B \Big(u_\rho+ \sqrt{u_\rho^2-s^2+t^2}\Big)^2 \, \dx \markdist(t) \bigg) \, \dx \markdist (s) \\
        &= \lambda_2(B)\lim_{\rho\to\infty} \gamma \rho^2 \int_D \exp\bigg(-v_2 \gamma \int_0^\B \sum_{i=0}^{1} \frac{c_{i,2}}{i!} u_\rho^{2-2i} (t^2-s^2)^{i} \, \dx \markdist(t) \bigg) \, \dx \markdist (s)\\
         &= \lambda_2(B)\lim_{\rho\to\infty} \gamma \rho^2  e^{-u-\stworho}\int_D \exp\bigg(-2v_2 \gamma \int_0^\B    (t^2-s^2) \, \dx \markdist(t) \bigg) \, \dx \markdist (s)\\
        & =  \lambda_2(B)\lim_{\rho\to\infty} e^{-u - \alpha_\rho}\frac{1}{\e e^{2v_2\gamma M^2}} \int_D \exp\big(2 v_2 \gamma s^2 \big) \, \dx \markdist (s) = e^{-u}\lambda_2(B) \markdistlimtwo(D),
    \end{split}
\end{align*}
which is the desired identity.
\end{proof}

\begin{lemma}\label{lem:intensity_measure_d}
Let $d\geq 3$ and let the assumptions of Theorem \ref{thm:bounded} be satisfied. Then
\begin{equation*}
\lim_{\rho\to\infty}\e \sum_{\bfx \in \ppp \cap W_\rho \times [0,\B]} \ind{r(\bfx,\ppp)^d\geq \frac{u + \sdrho}{c_d \gamma}} \ind{\bfx \in \rho B \times [0,a]} = e^{-u}\lambda_d(B) \markdistlimd([0,a])
\end{equation*}
for all $u\in\R$, measurable sets $B\subseteq W$ and $a\geq 0$.
\end{lemma}

\begin{proof}
We use again the abbreviation $u_\rho:=\big(\max\big\{\frac{u + \stworho}{c_d\gamma},0\big\}\big)^{1/d}$ for $\rho\ge 1$ and assume that $\rho$ is sufficiently large so that $u_\rho>\B$. We first prove that 
\begin{equation}\label{eqn:limit_A-varepsilon}
       \lim_{\rho \rightarrow \infty}  \rho^d \exp\bigg(-v_d \gamma \int_0^\B \Big(u_\rho + \sqrt{u_\rho^2-\tilde{s}^2+t^2}\Big)^d \, \dx \markdist(t) \bigg) = 0
\end{equation}
for all $\tilde{s}\in[0,\B)$. Assume for contradiction that there exists an $\tilde{s}_0\in[0,\B)$, a sequence $\{\rho_n\}_{n\in\mathbb{N}}$ with $\lim_{n\to\infty} \rho_n=\infty$ and $L\in (0,\infty)\cup\{\infty\}$ such that 
$$
\lim_{n \rightarrow \infty}  \rho_n^d \exp\bigg(-v_d \gamma \int_0^\B \Big(u_{\rho_n} + \sqrt{u_{\rho_n}^2-\tilde{s}_0^2+t^2}\Big)^d \, \dx \markdist(t) \bigg)  = L.
$$
For $n\in\n$, $t \in [0,\B]$ and $x \in (0,\B)$ we define $F_{\rho_n,t}(x):= \Big(u_{\rho_n} + \sqrt{u_{\rho_n}^2-x^2+t^2}\Big)^d$. Then by the mean value theorem there exists a $c_{\rho_n,t} \in (\tilde{s}_0,\frac{\tilde{s}_0+\B}{2})$ such that 
\begin{align*}
F_{\rho_n,t}\Big(\frac{\tilde{s}_0+\B}{2}\Big) - F_{\rho_n,t}(\tilde{s}_0) & = -\frac{\B-\tilde{s}_0}{2} d c_{\rho_n,t} \frac{\Big(u_{\rho_n} + \sqrt{u_{\rho_n}^2-c_{\rho_n,t}^2+t^2}\Big)^{d-1}}{\sqrt{u_{\rho_n}^2-c_{\rho_n,t}^2+t^2}} \\
& \leq -\frac{\B-\tilde{s}_0}{2} d \tilde{s}_0 \frac{\Big(u_{\rho_n} + \sqrt{u_{\rho_n}^2-\B^2}\Big)^{d-1}}{\sqrt{u_{\rho_n}^2+\B^2}}. 
\end{align*}
For $n$ sufficiently large, there exists a constant $C_d > 0$ such that the right-hand side is less that $-C_d u_{\rho_n}^{d-2}$. This shows that
\begin{align*}
& \lim_{n\to\infty } \rho_n^d \exp\bigg(-v_d \gamma \int_0^\B \Big(u_{\rho_n} + \sqrt{u_{\rho_n}^2-\Big(\frac{\tilde{s}_0+\B}{2}\Big)^2+t^2}\Big)^d \, \dx \markdist(t) \bigg) \\
& =\lim_{n \rightarrow \infty} \rho_n^d \exp\bigg(-v_d \gamma \int_0^\B  F_{\rho_n,t}\Big(\frac{\tilde{s}_0+\B}{2}\Big) \, \dx \markdist(t) \bigg) \\
&\geq \lim_{n \rightarrow \infty}  \rho_n^d \exp\bigg( C_d u_{\rho_n}^{d-2}  \bigg) \exp\bigg(-v_d \gamma \int_0^\B  F_{\rho_n,t}(\tilde{s}_0) \, \dx \markdist(t) \bigg) = \infty.
\end{align*}
By monotonicity we have
\begin{align*}
& \rho_n^d \int_0^{\B} \exp\bigg(-v_d \gamma \int_0^\B \Big(u_{\rho_n} + \sqrt{u_{\rho_n}^2-s^2+t^2}\Big)^d \, \dx \markdist(t) \bigg) \, \dx \markdist (s) \\
& \geq \markdist\Big(\Big(\frac{\B+\tilde{s}_0}{2},\B\Big]\Big) \rho_n^d \exp\bigg(-v_d \gamma \int_0^\B \Big(u_{\rho_n} + \sqrt{u_{\rho_n}^2-\Big(\frac{\tilde{s}_0+\B}{2}\Big)^2+t^2}\Big)^d \, \dx \markdist(t) \bigg), 
\end{align*}
where the right-hand side tends to infinity as $n\to\infty$. Thus, Lemma \ref{lem:ExpectationFormula} b) yields that
$$
\lim_{\rho\to\infty} \e \sum_{\bfx \in \ppp \cap W_{\rho_n} \times [0,\B]} \ind{r(\bfx,\ppp)\geq u_{\rho_n}} =\infty,
$$
which is a contradiction to the fact that $\{\sdrho\}_{\rho\geq 1}$ is a correct shift. This proves \eqref{eqn:limit_A-varepsilon}.

For $a\in[0,\B)$, we have
\begin{align*}
 \rho^d \int_0^{a} \exp\bigg(-v_d \gamma \int_0^\B \Big(u_{\rho} &+ \sqrt{u_{\rho}^2-s^2+t^2}\Big)^d \, \dx \markdist(t) \bigg) \, \dx \markdist (s) \\
& \leq \rho^d \exp\bigg(-v_d \gamma \int_0^\B \Big(u_{\rho} + \sqrt{u_{\rho}^2-a^2+t^2}\Big)^d \, \dx \markdist(t) \bigg) 
\end{align*}
so that, by \eqref{eqn:limit_A-varepsilon}, the left-hand side vanishes as $\rho\to\infty$. Thus, Lemma \ref{lem:ExpectationFormula} b) yields
$$
\lim_{\rho\to\infty}\e \sum_{\bfx \in \ppp \cap W_\rho \times [0,\B]} \ind{r(\bfx,\ppp)^d\geq \frac{u + \sdrho}{c_d \gamma}} \ind{\bfx \in \rho B \times [0,a]} = 0.
$$
For $a\geq \B$ it follows from the Mecke formula, the stationarity of $\eta$ and the fact that $\{\sdrho\}_{\rho\geq 1}$ is a correct shift that
\begin{align*}
& \lim_{\rho\to\infty}\e \sum_{\bfx \in \ppp \cap W_\rho \times [0,\B]} \ind{r(\bfx,\ppp)^d\geq \frac{u + \sdrho}{c_d \gamma}} \ind{\bfx \in \rho B \times [0,a]} \\
& = \lim_{\rho\to\infty} \int_{ W_\rho \times [0,\B]} \p\Big(r(\bfx,\ppp+\delta_{\bfx})^d\geq \frac{u + \sdrho}{c_d \gamma}\Big) \ind{\bfx \in \rho B \times [0,\B]} \, \dx(\lambda_d\otimes\markdist)(\bfx) \\
& = \lim_{\rho\to\infty} \lambda_d(\rho B) \int_{[0,\B]} \p\Big(r((0,t),\ppp+\delta_{(0,t)})^d\geq \frac{u + \sdrho}{c_d \gamma}\Big)  \, \dx\markdist(t) \\
& = \lim_{\rho\to\infty} \frac{\lambda_d(\rho B)}{\lambda_d(W_\rho)} \int_{ W_\rho \times [0,\B]} \p\Big(r(\bfx,\ppp+\delta_{\bfx})^d\geq \frac{u + \sdrho}{c_d \gamma}\Big) \ind{\bfx \in \rho B \times [0,\B]} \, \dx(\lambda_d\otimes\markdist)(\bfx) \\
& = \lambda_d(B) \lim_{\rho\to\infty} \e \sum_{\bfx \in \ppp \cap W_\rho \times [0,\B]} \ind{r(\bfx,\ppp)^d\geq \frac{u + \sdrho}{c_d \gamma}} = \lambda_d(B) e^{-u}.
\end{align*}
Combining the two previous displays completes the proof. 
\end{proof}

Finally, we can present the convergence of the intensity measure $\mathbb{L}_\rho$ of the point process $\xi_\rho$. First, define
\begin{equation*}
    \mathcal{U}_1(\markdist):= \{\left[\alpha, \beta\right): 0\leq \alpha < \beta < \B,\, \markdistlimd(\{\alpha,\beta\}) = 0\}           \cup \{[\alpha, \B]: 0 \leq \alpha  < \B,\, \markdistlimd(\{\alpha\}) =0\}.
\end{equation*}
Furthermore, let $-\infty < a < b < \infty$ and let
\begin{equation} \label{def:DissSemiRingU}
         \mathcal{U}^{a,b}(\markdist):= \borel(W) \times \mathcal{U}_1(\markdist) \times ( \{\left[u, v\right): a\leq u< v< b\} \cup \{[u, b]: a \leq u < b\}).
    \end{equation}
As product of dissecting semi-rings the family $\mathcal{U}^{a,b}(\markdist)$ is a dissecting semi-ring of $W\times[0,\B]\times [a,b]$.

\begin{cor} \label{cor:ConvergenceOfIntMeasures}
Let $d\geq 2$ and let the assumptions of Theorem \ref{thm:bounded} be satisfied. Then
$$
\lim_{\rho\to\infty} \mathbb{L}_\rho(I) = (\lambda_d|_W\otimes\markdistlimd\otimes\mathbb{K})(I)
$$
for all $I \in  \mathcal{U}^{a,b}(\markdist)$.
\end{cor}

\begin{proof}First, realize that thanks to the continuity of $\mathbb{K}$, it is enough to consider sets $I \in \mathcal{U}^{a,b}(\markdist)$ of the form $I = B \times D \times [u,v)$ for some $B \in \borel(W)$ and $D \in\mathcal{U}_1(\markdist)$. Let $I$ be such a set, then clearly 
\begin{align*}
    \mathbb{L}_\rho&(B \times D \times [u,v)) \\
    &= \e \sum_{\bfx \in \ppp \cap \rho B \times D} \ind{r(\bfx,\ppp)^d\geq \frac{u + \sdrho}{c_d \gamma}}  - \e \sum_{\bfx \in \ppp \cap \rho B \times D} \ind{r(\bfx,\ppp)^d\geq \frac{v + \sdrho}{c_d \gamma}}.
\end{align*}
For $d = 2$ the statement is a straightforward consequence of Lemma \ref{lem:intensity_measure_2}.

For $d\geq 3$ we consider the cases $D=[\alpha,\beta)$ with $\alpha<\beta<\B$ and $D=[\alpha,\B]$ with $0<\alpha<\B$.
For the first case choose $\beta'\in (\beta,\B)$. Then one has 
\begin{align*}
    \liminf_{\rho \rightarrow \infty} \e \sum_{\bfx \in \ppp \cap \rho B \times [\alpha,\beta)} \ind{r(\bfx,\ppp)^d\geq \frac{u + \sdrho}{c_d \gamma}} \leq \lim_{\rho \rightarrow \infty} \e \sum_{\bfx \in \ppp \cap \rho B \times [0,\beta']} \ind{r(\bfx,\ppp)^2\geq \frac{u + \stworho}{c_2 \gamma}} = 0
\end{align*}
due to Lemma \ref{lem:intensity_measure_d} and the definition of $\markdistlimd$. Hence
\begin{equation*}
    \lim_{\rho\to\infty} \mathbb{L}_\rho(B \times [\alpha,\beta)\times[u,v)) = 0 = (\lambda_d|_W\otimes\markdistlimd\otimes\mathbb{K})(B \times [\alpha,\beta)\times[u,v)).
\end{equation*}
On the other hand if $D = [\alpha,\B]$, it follows from Lemma \ref{lem:intensity_measure_d} and the previous statement that
\begin{align*}
    &\lim_{\rho\to\infty} \mathbb{L}_\rho(B \times [\alpha,\B]\times[u,v)) \\
    & = \lim_{\rho\to\infty} \mathbb{L}_\rho(B \times [0,\B]\times[u,v)) - \lim_{\rho\to\infty} \mathbb{L}_\rho(B \times [0,\alpha)\times[u,v)) \\
    & = (e^{-u}-e^{-v}) \lambda_d(B) \markdistlimd([0,\B]) - 0 = (e^{-u}-e^{-v}) \lambda_d(B) \markdistlimd([\alpha,\B]) \\
    &= (\lambda_d|_W\otimes\markdistlimd\otimes\mathbb{K})(B \times [\alpha,\beta)\times[u,v)),
\end{align*}
which completes the proof. 
\end{proof}

\subsection{Proof of Theorem \ref{thm:bounded}} \label{subsec:ProofConvBounded}

For $a,b\in\R$ with $a<b$ and $\rho\ge 1$ we let $\xi_\rho^{a,b}$ be the restriction of $\xi_\rho$ to points whose third component belongs to $[a,b]$, i.e., $\xi_\rho^{a,b} = \xi_\rho|_{W\times [0,\B] \times [a,b]}$. We denote the intensity measure of $\xi_\rho^{a,b}$ by $\mathbb{L}_\rho^{a,b}$. The key step in the proof of Theorem \ref{thm:bounded} is the application of Theorem \ref{thm:BSY} to approximate the point processes $\xi_\rho^{a,b}$, $\rho \geq 1$, by a suitable Poisson processes. 

\begin{prop}\label{prop:limitdKR}
Let the assumptions of Theorem \ref{thm:bounded} be satisfied, let $a,b\in\R$ with $a<b$ and let $\zeta_\rho^{a,b}$ be a Poisson process with intensity measure $\mathbb{L}_\rho^{a,b}$ for $\rho\ge 1$. Then
$$
\lim_{\rho\to\infty} \dkr(\xi_\rho^{a,b},\zeta_\rho^{a,b})=0.
$$
\end{prop}

\begin{proof}
By Corollary \ref{cor:OrderofCR} b), there exists a $\rho_0\in(1,\infty)$ such that 
\begin{equation}\label{eqn:assumptions_rho_a_A}
a+\sdrho > c_d \gamma \B^d \quad \text{and} \quad b+\sdrho \ge \sqrt{3}^d  c_d \gamma \B^d
\end{equation}
for all $\rho\ge \rho_0$. Throughout this proof, we assume that $\rho\geq \rho_0$ and use the abbreviations
$$
a_{\rho} := \Bigl( \frac{a+\sdrho}{c_d \gamma}\Bigr)^{\frac{1}{d}} \quad \text{and} \quad b_{\rho} := \Bigl( \frac{b+\sdrho}{c_d \gamma}\Bigr)^{\frac{1}{d}} 
$$
for $\rho\ge\rho_0$. We note that
$$
\xi_\rho^{a,b} = \sum_{\bfx\in\eta} g_\rho(\bfx,\eta) \delta_{f_\rho(\bfx,\eta)}
$$
for $\rho\ge \rho_0$ with the functions $g_\rho$ and $f_\rho$ given by
$$
g_\rho(\bfx,\varphi) := \ind {c_d \gamma r(\bfx, \varphi)^d - \sdrho \in [a,b]} \ind {\bfx \in \Wrho \times [0,\B]}
$$
and
$$
f_\rho(\bfx,\varphi) := \left(\frac{1}{\rho}\cdot x,m, c_d \gamma r(\bfx, \varphi)^d - \sdrho\right)
$$
for $\varphi\in \mathcal{M}_{p}(\rd\times[0,\B])$, $\bfx=(x,m)\in\varphi$ and $\rho\ge\rho_0$. From Lemma \ref{lemma:EmptySetFormula} we know that
$$
c_d \gamma r(\bfx, \varphi)^d - \sdrho \geq v \iff \varphi(D(m,((v+\sdrho)/c_d \gamma)^{1/d})+x)=0
$$
for $v\in\R$ such that $(v+\sdrho)/(c_d \gamma) > 0$. We can pass to strict inequality on the left-hand side of the equivalence by considering the closure $\Bar{D}(\cdot, \cdot)$ of the set $D(\cdot, \cdot)$ on the right-hand side. Furthermore, it holds that $D(m,z)\subseteq D(m,z')$ whenever $z \leq z'$. Altogether, we can conclude that $g_\rho(\bfx,\varphi)$ and $f_\rho(\bfx,\varphi)$ are determined by $\varphi\cap (\Bar{D}(m,b_\rho)+x)$ in the sense that 
\begin{equation}\label{eqn:f_g_determined}
g_\rho(\bfx,\varphi) \cdot \delta_{f_\rho(\bfx,\varphi)} = g_\rho(\bfx,\varphi\cap (\Bar{D}(m,b_\rho)+x)) \cdot \delta_{f_\rho(\bfx,\varphi\cap (\Bar{D}(m,b_\rho)+x))}
\end{equation}
for $\varphi\in \mathcal{M}_{p}(\rd\times[0,\B])$, $\bfx\in\varphi$ and $\rho\ge\rho_0$. Since, by \eqref{eqn:assumptions_rho_a_A}, $\B \le \frac{1}{\sqrt{3}}b_\rho$, we obtain
\begin{align*}
\Bar{D}(m,b_\rho) & = \{(y,t) \in \rd \times [0,\B]: (\norm{y}-b_\rho)^2+m^2 \leq b_\rho^2 + t^2\} \\
& \subseteq \{(y,t) \in \rd \times [0,\B]: (\norm{y}-b_\rho)^2 \leq b_\rho^2 + \B^2\} \subseteq B^d(0, 3b_\rho)\times [0,\B].
\end{align*}
Thus, we define $S_{\rho}(\bfx) := B^d(x,3 b_{\rho}) \times [0,\B]$ for $\bfx=(x,m)\in \rd\times [0,\B]$ and $\rho\geq \rho_0$. Together with \eqref{eqn:f_g_determined}, we see that the assumptions of Theorem \ref{thm:BSY} are satisfied. Since the point processes $\xi_\rho^{a,b}$ and $\zeta^{a,b}_\rho$ have both the intensity measure $\mathbb{L}_\rho^{a,b}$, the bound in Theorem \ref{thm:BSY} simplifies to the two terms
\begin{align*}
    E_{2,\rho} &:= 2 \gamma^2 \int\displaylimits_{(\Wrho \times [0,\B])^2} \e \bigl[ \ind{r(\bfx,\ppp+\delta_\bfx) \in [a_{\rho},b_{\rho}]}\bigr] \cdot \e \bigl[ \ind{r(\bfy,\ppp+\delta_\bfy) \in [a_{\rho},b_{\rho}]}\bigr] \\
    & \hspace{7cm} \times \ind{S_{\rho}(\bfx) \cap S_{\rho}(\bfy) \neq \emptyset} \, \dx (\lambda_d \otimes \markdist)^{\otimes 2}(\bfx,\bfy),\\
   E_{3,\rho} &:= 2 \gamma^2 \int\displaylimits_{(\Wrho \times [0,\B])^2}  \e \bigl[ \ind{r(\bfx,\ppp+\delta_\bfx+\delta_\bfy) \in [a_{\rho},b_{\rho}]} \ind{r(\bfy,\ppp+\delta_\bfx+\delta_\bfy) \in [a_{\rho},b_{\rho}]}\bigr]  \\
    & \hspace{7cm} \times \ind{S_{\rho}(\bfx) \cap S_{\rho}(\bfy) \neq \emptyset} \, \dx (\lambda_d \otimes \markdist)^{\otimes 2}(\bfx,\bfy).
\end{align*}
In the following, we establish that $\lim_{\rho\to\infty}E_{2,\rho}=0$ and $\lim_{\rho\to\infty}E_{3,\rho}=0$. 
First, we consider $E_{2,\rho}$. To this end, we define
$$
I_{\rho,2}:=\int_{\Wrho  \times [0,\B]} \e \bigl[ \ind{r(\bfy,\ppp+\delta_{\bfy}) \in \left[a_\rho,b_\rho\right]}\bigr] (\lambda_d \otimes \markdist)(\dx\bfy)
$$
for $\rho\geq\rho_0$. For a fixed $\bfx=(x,s) \in \Wrho\times [0,\B]$, the translation invariance of $\eta$ leads to
\begin{align*}
    \int_{(\Wrho\cap B^d(x,6b_\rho))  \times [0,\B]}  &\e \bigl[ \ind{r(\bfy,\ppp+\delta_\bfy) \in \left[a_\rho,b_\rho\right]}\bigr]  (\lambda_d \otimes \markdist)(\dx\bfy) \\
    &= \lambda_d(\Wrho \cap B^d(x,6b_\rho)) \int_{[0,\B]} \e \bigl[ \ind{r((0,t),\ppp+\delta_{(0,t)}) \in \left[a_\rho,b_\rho\right]}\bigr] \markdist(\dx t) \\
    &\leq \frac{6^d v_d b_\rho^d}{\rho^d} \rho^d \int_{[0,\B]} \e \bigl[ \ind{r((0,t),\ppp+\delta_{(0,t)}) \in \left[a_\rho,b_\rho\right]}\bigr] \markdist(\dx t) \\
    & =  \frac{6^d v_d b_\rho^d}{\rho^d} \int_{\Wrho  \times [0,\B]} \e \bigl[ \ind{r(\bfy,\ppp+\delta_{\bfy}) \in \left[a_\rho,b_\rho\right]}\bigr] (\lambda_d \otimes \markdist)(\dx\bfy) \\    
    &= \frac{6^d v_d b_\rho^d}{\rho^d}  I_{\rho,2}.
\end{align*}
Therefore, we obtain
\begin{equation*}
     E_{\rho,2} \leq 2 \gamma^2 \frac{6^d v_d b_\rho^d }{\rho^d}  I_{\rho,2} \int_{\Wrho \times [0,\B]}  \e \bigl[ \ind{r(\bfx,\ppp+\delta_\bfx) \in \left[a_\rho,b_\rho\right]}\bigr]  (\lambda_d \otimes \markdist)(\dx\bfx) = 2 \gamma^2 \frac{6^d b_\rho^d v_d}{\rho^d} (I_{\rho,2})^2.
\end{equation*}
It follows from the Mecke formula and \eqref{defCR} that
\begin{align*}
   \lim_{\rho\to\infty} I_{\rho,2} & = \lim_{\rho\to\infty} \int_{\Wrho  \times [0,\B]} \e \bigl[ \ind{r(\bfx,\ppp+\delta_{\bfx}) \in \left[a_\rho,b_\rho\right]}\bigr] (\lambda_d \otimes \markdist)(\dx\bfx) \\
   & = \lim_{\rho\to\infty} \e \sum_{\bfx \in \ppp \cap W_\rho \times [0,\B]} \ind{r(\bfx,\ppp)\in [a_\rho,b_\rho]} = e^{-a}-e^{-b}.
\end{align*}
Combining this with $\lim_{\rho\to\infty} \frac{b_\rho^d}{\rho^d} = \lim_{\rho\to\infty} \frac{\log \rho^d}{\rho^d}=0$, which follows from Corollary \ref{cor:OrderofCR} b), we see that $\lim_{\rho\to\infty} E_{\rho,2}=0$.

Next we study $E_{3,\rho}$. For a fixed $\bfx=(x,s) \in \Wrho \times [0,\B]$ we define
\begin{equation*}
    E_{\rho,3}(\bfx) := \int_{B^d(x,6b_\rho)\times [0,\B]}  \e \bigl[ \ind{r(\bfx,\ppp+\delta_\bfx+\delta_\bfy) \geq a_\rho} \ind{r(\bfy,\ppp+\delta_\bfx+\delta_\bfy) \geq a_\rho} \bigr] \, \dx (\lambda_d\otimes \markdist)(\bfy).
\end{equation*}
If the inradii of $\bfx=(x,s)$ and $\bfy=(y,t)$ are both at least $a_\rho$, we must have $\norm{x-y}\geq 2a_\rho$. Thus, it is sufficient to integrate over $(B^d(x,6b_\rho)\setminus B^d(x,2a_\rho))\times [0,\B]$. Denoting by $B^d(x,r,R) := B^d(x,R) \setminus B^d(x,r)$ the annulus with center $x$ and radii  $r < R$, we see that
$$
E_{\rho,3}(\bfx) = \int_{B^d(x,2a_\rho,6b_\rho)\times [0,\B]}  \e \bigl[ \ind{r(\bfx,\ppp+\delta_\bfx+\delta_\bfy) \geq a_\rho} \ind{r(\bfy,\ppp+\delta_\bfx+\delta_\bfy) \geq a_\rho} \bigr] \, \dx (\lambda_d\otimes \markdist)(\bfy).
$$
Using that the inradius associated with a point does not decrease if other points are removed as well as stationarity, Lemma \ref{lemma:EmptySetFormula} and the fact that $D(s,u) \subseteq D(s',u)$, whenever $s'<s$, we can bound
\begin{align*}
    E_{\rho,3}(\bfx)
    &\leq \int_{B^d(x,2a_\rho,6b_\rho)\times [0,\B]}  \e \bigl[ \ind{r(\bfx,\ppp+\delta_\bfx) \geq a_\rho} \ind{r(\bfy,\ppp+\delta_\bfy) \geq a_\rho}\bigr] \, \dx (\lambda_d\otimes \markdist)(\bfy)\\
    &= \int_{B^d(0,2a_\rho,6b_\rho)\times [0,\B]}  \e \bigl[ \ind{\ppp(D(s,a_\rho) = 0)} \ind{\ppp(D(t,a_\rho)+y)=0} \bigr] \, \dx (\lambda_d\otimes \markdist)((y,t))\\
    &= \int_{B^d(0,2a_\rho,6b_\rho)\times [0,\B]}  \exp{(-\gamma(\lambda_d\otimes\markdist)(D(s,a_\rho) \cup (D(t,a_\rho)+y)))} \, \dx (\lambda_d\otimes \markdist)((y,t))\\
    &\leq \int_{B^d(0,2a_\rho,6b_\rho)}   \exp{(-\gamma(\lambda_d\otimes\markdist)(D(\B,a_\rho) \cup (D(\B,a_\rho)+y)))} \, \dx y. 
\end{align*}
For any $y\in B^d(0,2a_\rho,6b_\rho)$ we obtain
\begin{align*}
    &(\lambda_d\otimes\markdist)(D(\B,a_\rho) \cup (D(\B,a_\rho)+y)) \\
    &= \int_{[0,\B]} \lambda_d(\{z \in \rd: u^2-\B^2> \min\{\norm{z}^2-2a_\rho\norm{z},\norm{y-z}^2-2a_\rho\norm{y-z}\}\}) \, \dx \markdist(u)\\
    &\geq \lambda_d(\{z \in \rd: -\B^2> \min\{\norm{z}^2-2a_\rho\norm{z},\norm{y-z}^2-2a_\rho\norm{y-z}\}\})\\
    &= \lambda_d\Big(B^d\Big(0,a_\rho - \sqrt{a_\rho^2-\B^2},a_\rho + \sqrt{a_\rho^2-\B^2})\cup B^d(y,a_\rho - \sqrt{a_\rho^2-\B^2},a_\rho + \sqrt{a_\rho^2-\B^2}\Big)\Big).
\end{align*}
Since $\norm{y} \ge 2a_\rho$, one has $\lambda_d(B^d(y,r)\cap B^d(0,r)) \leq \frac{1}{2}\lambda_d(B(0,r))$ for any $r\in[0,2 a_\rho]$. This implies that
$$
(\lambda_d\otimes\markdist)(D(\B,a_\rho) \cup (D(\B,a_\rho)+y)) \geq \frac{3}{2}v_d \Big(a_\rho+\sqrt{a_\rho^2 -\B^2}\Big)^d-2v_d\Big(a_\rho-\sqrt{a_\rho^2 -\B^2}\Big)^d.
$$
Due to
$$
\lim_{\rho\to\infty} \frac{a_\rho+\sqrt{a_\rho^2 -\B^2}}{a_\rho}=2 \quad \text{and} \quad  \lim_{\rho\to\infty} \frac{a_\rho+\sqrt{a_\rho^2 -\B^2}}{a_\rho}=0,
$$
we have, for $\rho$ sufficiently large,
$$
\gamma (\lambda_d\otimes\markdist)(D(\B,a_\rho) \cup (D(\B,a_\rho)+y)) \geq \frac{5}{4} 2^dv_d \gamma a_\rho^d = \frac{5}{4}2^dv_d \gamma  \frac{a+\sdrho}{c_d \gamma} = \frac{5(a+\sdrho)}{4}.
$$
This yields that, for $\rho$ sufficiently large,
\begin{align*}
    \begin{split}
         E_{\rho,3} &\leq 2 \gamma^2 \int_{\Wrho \times [0,\B]} E_{\rho,3}(\bfx) \, \dx (\lambda_d \otimes \markdist)(\bfx)\\
         &\leq 2 \gamma^2 \int_{\Wrho \times [0,\B]} \int_{B^d(0,2a_\rho,6b_\rho)}  \exp\Big(-\frac{5(a+\sdrho)}{4}\Big) \, \dx y \, \dx (\lambda_d \otimes \markdist)(\bfx)\\
         &\leq 2  \gamma^2 \rho^d v_d(6^db_\rho^d-2^da_\rho^d) \exp\Big(-\frac{5(a+\sdrho)}{4}\Big) \leq 2 \gamma^2 6^dv_d \cdot \rho^d  b_\rho^d \exp\Big(-\frac{5(a+\sdrho)}{4}\Big).
    \end{split}
\end{align*}
As, by Corollary \ref{cor:OrderofCR} b), $\lim_{\rho\to\infty} \frac{\sdrho}{\log \rho^d}=1$ and furthermore $\lim_{\rho\to\infty}\frac{\sdrho}{b_\rho^d}=c_d \gamma$, the right-hand side vanishes as $\rho\to\infty$. This shows $\lim_{\rho\to\infty} E_{3,\rho}=0$ and completes the proof of the proposition.
\end{proof}

We are now ready to prove our main result for bounded marks. 

\begin{proof}[Proof of Theorem \ref{thm:bounded}] The existence of a correct shift follows from Proposition \ref{prop:ExistenceCorResc}. 

Let $-\infty<a<b<\infty$ and recall that $\xi^{a,b}_\rho$ is the restriction of $\xi_\rho$ to $W \times [0,\B] \times [a,b]$ and $\mathbb{L}^{a,b}_\rho$ is its intensity measure. We have shown in Proposition \ref{prop:limitdKR} that 
 \begin{equation}\label{formula:convergencedKR}
 \lim_{\rho\to\infty}\dkr(\xi^{a,b}_\rho,\zeta^{a,b}_\rho)=0,
\end{equation} 
where $\zeta^{a,b}_\rho$ is a Poisson process with intensity measure $\mathbb{L}^{a,b}_\rho$.

Next we show that the Poisson processes $\zeta^{a,b}_\rho$ converge to the restriction $\zeta^{a,b}_d$ of the Poisson process $\zeta_d$ to $W\times[0,\B]\times [a,b]$ as $\rho\to\infty$. Recall the definition \eqref{def:DissSemiRingU} of the dissecting semi-ring $\mathcal{U}^{a,b}(\markdist)$. Since, by Corollary \ref{cor:ConvergenceOfIntMeasures}, $\lim_{\rho\to\infty} \mathbb{L}_\rho(I) = (\lambda_d|_W\otimes\markdistlimd\otimes\mathbb{K})(I)$ for $I\in \mathcal{U}^{a,b}(\markdist)$, it follows from Lemma \ref{lemma:AuxPoissonConv} that
\begin{equation}\label{formula:convergencePoissons}
 \zeta^{a,b}_\rho \underset{\rho \rightarrow \infty}{\overset{d}{\longrightarrow}} \zeta^{a,b}_d \text{ in } \mathcal{M}_p(W \times [0,\B]\times\R). 
\end{equation}

The convergences \eqref{formula:convergencedKR} and \eqref{formula:convergencePoissons} imply that $\xi^{a,b}_\rho \underset{\rho \rightarrow \infty}{\overset{d}{\longrightarrow}} \zeta^{a,b}_d$ in $\mathcal{M}_p(W \times [0,\B]\times\R)$. Since this convergence holds for any $-\infty<a<b<\infty$, it follows from the characterization of convergence in distribution of random measures in \cite[Theorem 23.16]{bo:OKallen21}) that $\xi_\rho \underset{\rho \rightarrow \infty}{\overset{d}{\longrightarrow}} \zeta$ in $\mathcal{M}_p(W \times [0,\B]\times\R)$.  
\end{proof}

Let $\xi_\rho^{-b,b}$ and $\zeta_d^{-b,b}$ denote the restrictions of $\xi_\rho$ and $\zeta_d$ to $W \times [0,\B] \times [-b,b]$. 
The main idea of the proof of Corollary \ref{cor:bounded} is to establish convergence in distribution for these restricted point processes by Lemma \ref{lemma:continuity of argmax} and then to use an approximation argument. 

\begin{proof}[Proof of Corollary \ref{cor:bounded}] 
 Consider the mapping $T$ defined in \eqref{def:argmaxmap} and notice that in fact $T(\xi_\rho) = (\frac{1}{\rho}X_{\max,\rho},M_{\max,\rho},c_d\gamma R_{\max,\rho}^d-\sdrho)$ a.s. Let $b > 0$ and let $\mathcal{C}_b:= \mathcal{C}_{W \times [0,\B] \times [-b,b]}$ be as defined in \eqref{def:SubsetCK}. Then $\p(\zeta_d^{-b,b} \in \mathcal{C}_b)=1$. Therefore, it follows from Theorem \ref{thm:bounded}, Lemma \ref{lemma:continuity of argmax} and the continuous mapping theorem (see \cite[Theorem 5.27]{bo:OKallen21}) that for any $b > 0$,
 \begin{equation} \label{formula:conv in dist bounded}
 T(\xi_\rho^{-b,b}) \underset{\rho \rightarrow \infty}{\overset{d}{\longrightarrow}} T(\zeta_d^{-b,b}).
 \end{equation}
 Furthermore it holds for any $b > 0$ that 
 \begin{equation*}
     \{T(\zeta_d^{-b,b}) \neq T(\zeta_d)\} =\{\zeta_d(W \times [0,\B] \times (b,\infty))\geq 1\}\cup \{\zeta_d(W \times [0,\B] \times [-b,b])=0\}
 \end{equation*} 
 which implies that 
 \begin{equation} \label{limitPois}
        \limsup_{b \rightarrow \infty} \p(T(\zeta_d^{-b,b}) \neq T(\zeta_d)) \leq \lim_{b \rightarrow \infty} 1-\exp{(-e^{-b})}+\exp{(-e^b+e^{-b})} = 0. 
    \end{equation}
Similarly, it holds for any $\rho \geq 1$ and any $b > 0$ that 
 \begin{equation*}
     \{T(\xi_\rho^{-b,b}) \neq T(\xi_\rho)\} =\{\xi_\rho(W \times [0,\B] \times (b,\infty))\geq 1\}\cup \{\xi_\rho(W \times [0,\B] \times [-b,b])=0\}.
 \end{equation*} 
 Furthermore, we can bound 
 \begin{align*}
     \p(\xi_\rho(W \times [0,\B] \times (b,\infty))\geq 1) \leq \e \xi_\rho(W \times [0,\B] \times (b,\infty)) =  \mathbb{L}_\rho(W \times [0,\B] \times (b,\infty)).
 \end{align*}
 It follows from the convergence in Theorem \ref{thm:bounded} that, for any $b > 0$,
 \begin{equation*}
     \lim_{\rho \rightarrow \infty} \p(\xi_\rho(W \times [0,\B] \times [-b,b])=0)=\p(\zeta_d(W \times [0,\B] \times [-b,b])=0) = \exp{(-e^{b}+e^{-b})}.
 \end{equation*}
Together with \eqref{formula:CRwithIntensityMeasure} we can bound
 \begin{align*}
     \limsup_{\rho \rightarrow \infty}  \p(T(\xi_\rho^{-b,b})  \neq T(\xi_\rho))  
     &\leq \exp{(-e^{b}+e^{-b})} + \limsup_{\rho \rightarrow \infty} \mathbb{L}_\rho(W \times [0,\B] \times [b,\infty)) \\
     & = \exp{(-e^{b}+e^{-b})} + e^{-b}
 \end{align*}
so that
    \begin{equation}\label{limitxirho}
        \lim_{b \rightarrow \infty} \limsup_{\rho \rightarrow \infty} \p(T(\xi_\rho^{-b,b}) \neq T(\xi_\rho)) = 0.
    \end{equation}
 Using the limits \eqref{limitPois} and \eqref{limitxirho} and the convergence \eqref{formula:conv in dist bounded}, we can prove that for any $f: \R^{d+2} \rightarrow \R$ bounded continuous  $\abs{\e f(T(\xi_\rho)) - \e f(T(\zeta_d))} \rightarrow 0$ as $\rho \rightarrow \infty$, which implies that also $T(\xi_\rho) \underset{\rho \rightarrow \infty}{\overset{d}{\longrightarrow}} T(\zeta_d)$. 
 
 To see that in fact $ T(\zeta_d) = \underset{(y,m,r) \in \zeta_d}{\arg \max} r$ has the desired distribution, realize that $\zeta_d$ can be seen as a Poisson process on $\R$ marked by i.i.d.\ copies of $(X,M_d) \sim \lambda_d \big|_W \otimes \markdistlimd$. From the formula for the intensity measure of $\zeta_d$ derivw
 $$\p(\underset{(X,M,R) \in \zeta_d}{\max} R \leq a) = \p(\zeta_d(W \times [0,\B]\times(a,\infty))=0) = e^{-e^{-a}}$$
 for any $a \in \R$, which finishes the proof. 
\end{proof}

\subsection{Examples of correct shifts} \label{subsec:ProofsExamples}

As stated in Proposition \ref{prop:ExistenceCorResc}, we have explicit formulas for correct shifts in dimensions $d \in\{ 2,3\}$. While the dependence on $\rho$ and the choice of $\markdist$ is quite transparent in the formula  \eqref{CRtwo} for $d = 2$, the formula \eqref{CRthree} for $d = 3$ is more complicated. In this section, 
 we present valid choices for correct shifts for selected examples of $\markdist$, which particularly demonstrate that the dependence on $\rho$ may differ, although not in the leading term. Throughout this section, we consider mark distributions $\markdist_1,\hdots,\markdist_3$. For $i\in\{1,\hdots,3\}$ we write $M_i$ for a mark distributed according to $\markdist_i$.

\begin{example} \label{example:CRdim3discrete}
Let $d = 3$ and let $\markdist_1:=\sum_{i=1}^K p_i\delta_{\B_i}$ with $K\in\n$, $0<\B_1<\hdots<\B_K$ and $p_1,\hdots,p_k\in[0,1]$ such that $\sum_{i=1}^Kp_i=1$ and $p_K>0$. Then a correct shift can be chosen as
    \begin{equation} \label{CRdim3discrete}
        \log{\rho^3} + 3(v_3 \gamma)^{\frac{2}{3}}(\B_K^2-\e M_1^2)(\log{\rho^3})^{1/3} + \log p_K.
    \end{equation}
For a fixed $\B\in(0,\infty)$ let $\markdist_2(B):=\frac{2}{\B^2}\int_B s \, \dx s$ for $B\in\mathcal{B}(\R)$. Then a correct shift can be chosen as 
\begin{equation} \label{CRdim3density}
    \log \rho^3 + 3(v_3 \gamma)^{\frac{2}{3}}(\B^2-\e M_2^2)( \log \rho^3)^{1/3} - \frac{1}{3} \log \log \rho^3-\log (3(v_3 \gamma)^{\frac{2}{3}}\B^2).
\end{equation}
\end{example}

Let us note that these choices are not obtained by strictly computing \eqref{CRthree}, but rather its polished versions without any vanishing terms.

\begin{proof}[Proof of \eqref{CRdim3discrete}]
It follows from Proposition \ref{prop:ExistenceCorResc} that a correct shift can be chosen as 
\begin{equation*}
    \log \rho^3 -3(v_3 \gamma)^{\frac{2}{3}}  \e M_1^2\big(\log \rho^3\big)^{\frac{1}{3}} + \log\bigg( \e \exp\big(3(v_3 \gamma)^{\frac{2}{3}} \big(\log \rho^3\big)^{\frac{1}{3}} M_1^2  \big) \bigg) + \log\gamma.
\end{equation*}
Define $\alpha_\rho:=3(v_3 \gamma)^{\frac{2}{3}} \big(\log \rho^3\big)^{\frac{1}{3}}$. Then we can write
\begin{align*}
    \log\big(\e \exp\big(\alpha_\rho M_1^2  \big)\big) &= \log\bigg(p_K \exp(\alpha_\rho A_K^2)(1+\sum_{i=1}^{K-1} \frac{p_i}{p_K} \exp\big(\alpha_\rho (A_i^2-A_K^2)\big)\bigg) \\
    &= \log p_K+ \alpha_\rho A_K^2+\log\bigg(1+\sum_{i=1}^{K-1} \frac{p_i}{p_K} \exp\big(\alpha_\rho (A_i^2-A_K^2)\big)\bigg).
\end{align*}
Clearly the last term converges to $0$ as $\rho \rightarrow \infty$ and can therefore be omitted from the correct shift. Thus, \eqref{CRdim3discrete} is a valid choice of a correct shift for the mark distribution $\markdist_1$.
\end{proof}

\begin{proof}[Proof of \eqref{CRdim3density}]
    Define $\alpha_\rho:=3(v_3 \gamma)^{\frac{2}{3}} \big(\log \rho^3\big)^{\frac{1}{3}}$. As in the previous example, it is enough to compute 
    \begin{align*}
        \log\big(\e \exp\big(\alpha_\rho M_2^2  \big)\big) &= \log \bigg( \frac{e^{\B^2\alpha_\rho }-1}{\B^2\alpha_\rho}\bigg) = \B^2 \alpha_\rho - \log(\B^2\alpha_\rho ) + \log(1-e^{-\alpha_\rho \B^2})\\
        &= 3(v_3 \gamma)^{\frac{2}{3}} \big(\log \rho^3\big)^{\frac{1}{3}} \B^2 - \frac{1}{3}\log(\log \rho^3) - \log(3(v_3 \gamma)^{\frac{2}{3}} \B^2) + \log(1-e^{-\alpha_\rho \B^2}).
    \end{align*}
    Again, the last term vanishes for $\rho \rightarrow \infty$, which validates the choice \eqref{CRdim3density} as a correct shift for $\markdist_2$. 
\end{proof}

 We also present a correct shift for a particular choice of $\markdist$ in dimension $d = 4$, demonstrating a possible strategy which could be employed to compute correct shifts for other suitable choices of $\markdist$ in higher dimensions. 

\begin{example} \label{example:CRdim4discrete}
    Let $d = 4$ and $\markdist_3 := (1-p)\delta_0 + p\delta_\B$ with $p \in (0,1]$ and $\B\in(0,\infty)$. Then a correct shift can be chosen as
\begin{equation} \label{CRdim4discrete}
    \log \rho^4 + 4(v_4\gamma)^{1/2}(\B^2-\e M_3^2)(\log \rho^4)^{1/2} + 2v_4\gamma \B^4(3-4p)(1-p)+ \log p + \log \gamma.
\end{equation}
\end{example}

Let us note that our examples are in accordance with the previous results for the Poisson--Voronoi tessellation from \cite{ar:CCh14,ar:O25,ar:PS22}) discussed in the introduction, which can be seen by taking $p_K =1$ in \eqref{CRdim3discrete} and $p=1$ in \eqref{CRdim4discrete}.

\begin{proof}[Proof of \eqref{CRdim4discrete}]
Define
\begin{equation*}
    \mathrm{r}_{\rho}:= 4(v_4\gamma)^{1/2}(\B^2-\e M_3^2)(\log \rho^4)^{1/2} +2v_4\gamma \B^4(3-4p)(1-p)+ \log p + \log \gamma.
 \end{equation*}
Using Lemma \ref{lemma:SolutionForAllu}, we can see that it is enough to show that 
\begin{equation*}
    \lim_{\rho \rightarrow \infty} G_{4,\rho}(\mathrm{r}_{\rho})=1
\end{equation*}
with $G_{4,\rho}$ defined in \eqref{def:Gfunction}. Computing $c_{0,4}=c_{1,4}=16$ and $c_{2,4}=4$, we can write
\begin{align*}
    G_{4,\rho}(x) &= \gamma \int_{0}^\B \exp\bigg(-x- \int_0^\B 4(v_4\gamma)^{1/2}(\log\rho^4+x)^{1/2}(t^2-s^2) + 2v_4\gamma(t^2-s^2)^2 \, \markdist_3(\dx t)\bigg) \, \markdist_3(\dx s)\\
    &=\gamma \exp\left(-x-4(v_4\gamma)^{1/2}(\log\rho^4+x)^{1/2}\e M_3^2 - 2v_4\gamma \e M_3^4\right) \\
    &\hspace{2.5cm} \times \int_{0}^\B \exp\bigg((4(v_4\gamma)^{1/2}(\log\rho^4+x)^{1/2} + 4v_4\gamma \e M_3^2) s^2 - 2v_4\gamma s^4\bigg) \, \markdist_3(\dx s)\\
    &=\gamma \exp\left(-x-4(v_4\gamma)^{1/2}(\log\rho^4+x)^{1/2}\e M_3^2 - 2v_4\gamma \e M_3^4\right) \\
    &\hspace{2.5cm} \times \left[p\exp\bigg((4(v_4\gamma)^{1/2}(\log\rho^4+x)^{1/2} + 4v_4\gamma \e M_3^2) \B^2 - 2v_4\gamma \B^4\bigg) + (1-p)\right] \\
    &=\gamma p \exp\left(-x+4(v_4\gamma)^{1/2}(\log\rho^4+x)^{1/2}(\B^2-\e M_3^2) - 2v_4\gamma (\e M_3^4 + \B^4) + 4v_4\gamma \e M_3^2\B^2\right) \\
    &\hspace{3cm}  + (1-p) \gamma \exp\left(-x-4(v_4\gamma)^{1/2}(\log\rho^4+x)^{1/2}\e M_3^2 - 2v_4\gamma \e M_3^4\right).
\end{align*}
Define $\alpha:= 4(v_4\gamma)^{1/2}(\B^2-\e M_3^2)$ and $\beta:= - 2v_4\gamma (\e M_3^4 + \B^4) + 4v_4\gamma \e M_3^2\B^2$ and notice that $\beta + \frac{\alpha^2}{2}=2v_4\gamma \B^4(3-7p+4p^2)=2v_4\gamma \B^4(3-4p)(1-p)$ so that
$r_\rho = \alpha(\log\rho^4)^{1/2} + \beta + \frac{\alpha^2}{2}+\log p + \log\gamma$. Thus, we get
\begin{align*}
    \lim_{\rho \rightarrow \infty} G_{4,\rho}(\mathrm{r}_{\rho})  
    &= \lim_{\rho \rightarrow \infty} \gamma p \exp\left(-r_\rho + \alpha(\log\rho^4+r_\rho)^{1/2} +\beta\right)\\
    &= \lim_{\rho \rightarrow \infty} \gamma p \exp\left(-r_\rho + \alpha(\log\rho^4)^{1/2} +\beta+ \alpha(\log\rho^4+r_\rho)^{1/2}-\alpha(\log\rho^4)^{1/2}\right)\\
    &= \lim_{\rho \rightarrow \infty} \gamma p \exp\left(-r_\rho + \alpha(\log\rho^4)^{1/2} +\beta+ \frac{\alpha r_\rho}{(\log\rho^4+r_\rho)^{1/2}+(\log\rho^4)^{1/2}}\right)\\
    &= \lim_{\rho \rightarrow \infty} \gamma p \exp\left(-r_\rho + \alpha(\log\rho^4)^{1/2} +\beta+ \frac{\alpha^2}{2}+\log p + \log \gamma\right)=1,
\end{align*}
which proves \eqref{CRdim4discrete}.
\end{proof}

\section{Proof of Theorem \ref{thm:lighttails}} \label{sec:ProofLT}
In this section we prove Theorem \ref{thm:lighttails}, which generalizes the results for bounded marks to a light-tailed case for the planar tessellation. Set $d = 2$ and recall that we assume that the typical mark $M \sim \markdist$ satisfies 
\begin{equation}\label{as:LightTailedMoments}
\e e^{(4\gamma v_2+\varepsilon)M^2} < \infty    
\end{equation}
for some $\varepsilon >0$.

The proof of Theorem \ref{thm:lighttails} is based on an approximation argument and is structured as follows. First, we establish in Lemma \ref{lemma:ConvIntMesureLT} that the intensity measure 
\begin{equation*}
    \mathbb{L}_\rho(\cdot):= \e \xi_\rho(\cdot) = \e \sum_{(x,m_x) \in \ppp \cap W_\rho \times [0,\infty)} \delta_{(\frac{1}{\rho}x, m_x, 4v_2\gamma r((x,m_x),\ppp)^2 - \stworho)} (\cdot)
\end{equation*}
of the (correctly shifted) point process $\xi_\rho$, where the family $\stworho$ is chosen as in \eqref{CRtwo}, converges to the intensity measure $\mathbb{M}_2:= \lambda_d\lvert_W \otimes \markdistlimtwo \otimes \mathbb{K}$ of the limiting point process $\zeta_2$ on a suitable class of subsets of $W \times [0,\infty) \times \R$. Notice that the assumption \eqref{as:LightTailedMoments} ensures that $\stworho$ and $\markdistlimtwo$ are both well defined. 

Next, in Lemmas \ref{lemma:ConvXiRhoAtilde} and \ref{lemma:ConvPoisLT}, we consider the convergence of the point processes with the marks of $\eta$ truncated by $\Atilde$ shifted by the family $\stworho$, which is not the correct shift since it does not take into account the truncation of the marks. We show that the truncation effect on the limiting Poisson process vanishes with $\Atilde \rightarrow \infty$. Finally, we prove in Lemmas \ref{lemma:auxboundsLT2} and \ref{lemma:auxboundsLT} that the expected number of points with large inradius and marks smaller than some fixed $\Mtilde$ whose large inradius is bounded by a neighbor with a large mark vanishes.

From now on, assume that $\stworho$ is chosen as in \eqref{CRtwo}. We define
$$
u_\rho := \max\left\{0,\frac{u + \stworho}{4\gamma v_2}\right\}^{1/2}
$$
for all $u\in\R$ and $\rho \geq 1$. Before we start with the result about the intensity measures, we prove two auxiliary inequalities. 

\begin{lemma}\label{lem:bounds_squares}
For $\rho\ge 1$ and $u\in\R$ such that $u_\rho\ge 1$, $s\in[0,u_\rho]$ and $t\in[0,\infty)$ one has
\begin{align}
 \label{bound+}
\Big(u_\rho+\sqrt{u_\rho^2-s^2+t^2}\Big)^2 =4u_\rho^2-s^2+t^2 + \frac{2u_\rho(-s^2+t^2)}{\sqrt{u_\rho^2-s^2+t^2}+u_\rho} &\geq 4u_\rho^2  -3 s^2 
\end{align}
and
\begin{align}
\label{bound-}
\Big(u_\rho-\sqrt{u_\rho^2-s^2+t^2}\Big)^2 &\le \frac{(s^2-t^2)^2}{u_\rho^2} \leq s^2 + t^4.
\end{align}
\end{lemma}

\begin{proof}
We obtain 
\begin{align*}
\begin{split} 
\Big(u_\rho+\sqrt{u_\rho^2-s^2+t^2}\Big)^2 & =4u_\rho^2-s^2+t^2+2u_\rho \Big(\sqrt{u_\rho^2-s^2+t^2}-u_\rho\Big) \\
& = 4u_\rho^2-s^2+t^2 + \frac{2u_\rho(-s^2+t^2)}{\sqrt{u_\rho^2-s^2+t^2}+u_\rho} \\
& \geq 4u_\rho^2-s^2+t^2  -2 s^2 + \frac{2u_\rho t^2}{\sqrt{u_\rho^2+t^2}+u_\rho}\geq 4u_\rho^2  -3 s^2 
\end{split}
\end{align*}
and
\begin{align*}
\begin{split}
\Big(u_\rho-\sqrt{u_\rho^2-s^2+t^2}\Big)^2 &= \frac{(s^2-t^2)^2}{\Big(u_\rho+\sqrt{u_\rho^2-s^2+t^2}\Big)^2} \leq \frac{(s^2-t^2)^2}{u_\rho^2} \leq \frac{s^4}{u_\rho^2} + \frac{t^4}{u_\rho^2} \le s^2 + t^4,
\end{split}
\end{align*}
which finishes the proof. 
\end{proof}

The following lemma implies that the choice \eqref{CRtwo} yields a correct shift in the setting of Theorem \ref{thm:lighttails}. 

\begin{lemma} \label{lemma:ConvIntMesureLT}
    Assume that $d = 2$ and that \eqref{as:LightTailedMoments} holds. Let $ B \subseteq W$ and $D \subseteq [0,\infty)$ be measurable and let $E \in \{[u,v],[u,\infty): -\infty < u < v < \infty\}$. Then 
    \begin{equation*}
        \lim_{\rho \rightarrow \infty} \mathbb{L}_{\rho}(B \times D \times E) = \mathbb{M}_2(B \times D \times E). 
    \end{equation*}
\end{lemma}

\begin{proof}
It is enough to consider sets of the form $B \times D \times [u,\infty)$, where $u \in \R$ and $B \subseteq W$ and $D \subseteq \rpluszero$ are measurable. We always assume that $\rho$ is sufficiently large such that $u_\rho>0$. Using the Mecke equation, stationarity and Lemma \ref{lemma:EmptySetFormula}, we can write
\begin{align*} 
    \mathbb{L}_\rho ( B \times D \times \left[u,\infty\right)) & = \e \sum_{\bfx \in \ppp \cap W_\rho \times [0,\infty)} \ind{4v_2\gamma r((x,m_x),\ppp)^2 - \stworho\ge u} \ind{\bfx\in \rho B\times D} \\
    & = \gamma \int_{\rho B} \int_D \p(4v_2\gamma r((x,t),\ppp+\delta_{(x,t)})^2 - \stworho\ge u) \, \dx \markdist(s) \, \dx x \\
    & = \gamma \lambda_2(B) \rho^2 \int_D \p(r((0,t),\ppp+\delta_{(0,t)})\ge u_\rho) \, \dx \markdist(s) \\
    &=\gamma \lambda_2(B) \rho^2 \int_D \exp(-\gamma \lambda_2\otimes\markdist(D(s,u_\rho))) \, \dx \markdist(s).
\end{align*}
Next, we split 
\begin{equation*}
\int_D e^{-\gamma \lambda_2\otimes\markdist(D(s,u_\rho))} \dx \markdist(s) = \int_{D\cap[0,u_\rho]} e^{-\gamma \lambda_2\otimes\markdist(D(s,u_\rho))} \dx \markdist(s) + \int_{D\cap(u_\rho,\infty)} e^{-\gamma \lambda_2\otimes\markdist(D(s,u_\rho))} \dx \markdist(s).
\end{equation*}
Let $r_2:= \stworho- \log \rho^2 $ and note that $r_2$ does not depend on $\rho$. The Markov inequality and assumption \eqref{as:LightTailedMoments} imply that 
\begin{align*}
    \begin{split}
        0 &\leq  \limsup_{\rho \rightarrow \infty} \gamma \rho^2  \int_{D\cap(u_\rho,\infty)} \exp(-\gamma \lambda_2\otimes\markdist(D(s,u_\rho)))\, \dx \markdist(s) \\
        &\leq  \limsup_{\rho \rightarrow \infty} \gamma \rho^2  \p(M> u_\rho) \leq \lim_{\rho \rightarrow \infty} \gamma \rho^2 \exp\big(-(4\gamma v_2 + \varepsilon)u_\rho^2\big) \e \exp\big((4\gamma v_2 + \varepsilon)M^2\big) \\
        &= \gamma  \e \exp\big((4\gamma v_2 + \varepsilon)M^2\big) \lim_{\rho \rightarrow \infty}  \exp(-u - r_2- \varepsilon u_\rho^2) =0.
    \end{split}
\end{align*}
Therefore, we have
\begin{align*}
    \lim_{\rho \rightarrow \infty} &\mathbb{L}_{\rho}(B \times D \times [u,\infty)) = \lim_{\rho \rightarrow \infty} \gamma \lambda_2(B) \rho^2 \int_{D\cap[0,u_\rho]} \exp(-\gamma \lambda_2\otimes\markdist(D(s,u_\rho))) \, \dx \markdist(s).
\end{align*}
Furthermore, using Lemma \ref{lemma:EmptySetFormula} and the bounds \eqref{bound+} and \eqref{bound-} from Lemma \ref{lem:bounds_squares}, we can write for $u_\rho\ge 1$ and $0\leq s \leq u_\rho$,
\begin{align*} 
\begin{split} 
     \rho^2 &\exp(-\gamma \lambda_2\otimes\markdist(D(s,u_\rho)))\\
     &= \rho^2 \exp\bigg(-\gamma v_2\int_0^\infty \Big(u_\rho+\sqrt{u_\rho^2-s^2+t^2}\Big)^2 \dx \markdist(t) \\
     & \hspace{3cm}+ \gamma v_2\int_0^s \Big(u_\rho-\sqrt{u_\rho^2-s^2+t^2}\Big)^2 \dx \markdist(t)\bigg) \\
     &\leq \rho^2 \exp\big(-4\gamma v_2 u_\rho^2 + \gamma v_2 \e  M^4 + 4 \gamma v_2 s^2\big)  = \exp\big(-u-r_2 + \gamma v_2 \e  M^4 + 4 \gamma v_2 s^2\big).
\end{split}
\end{align*}
The assumption \eqref{as:LightTailedMoments} implies that the right-hand side is integrable with respect to $\markdist$ so that the dominated convergence theorem leads to 
\begin{align*}
    \lim_{\rho \rightarrow \infty} &\mathbb{L}_{\rho}(B \times D \times [u,\infty)) \\
    &= \gamma \lambda_2(B) \int_{D} \lim_{\rho \rightarrow \infty} \ind{s\leq u_\rho} \rho^2   \exp \left(-\gamma v_2\int_0^\infty \Big(u_\rho+\sqrt{u_\rho^2-s^2+t^2}\Big)^2 \, \dx \markdist(t)\right. \\ 
    & \hspace{6cm} \left. + \gamma v_2\int_0^s \Big(u_\rho-\sqrt{u_\rho^2-s^2+t^2}\Big)^2 \dx \markdist(t)\right) \, \dx \markdist(s).
\end{align*}

Now we use the bounds \eqref{bound+} and \eqref{bound-} to compute the limit in the integral. Fix $s \in \rpluszero$. By \eqref{bound-}, we can bound 
\begin{align*}
   1 & \leq  \exp \bigg(\gamma v_2\int_0^s \Big(u_\rho-\sqrt{u_\rho^2-s^2+t^2}\Big)^2 \,\dx \markdist(t)\bigg)  \leq \exp \bigg(\gamma v_2\int_0^s (s^2-t^2)^2  \dx \markdist(t)/u_\rho^2\bigg) \underset{\rho \rightarrow \infty}{\longrightarrow} 1.
\end{align*}
From \eqref{bound+} and the dominated convergence therem, which is applicable due to $\abs{\frac{2u_\rho(-s^2+t^2)}{\sqrt{u_\rho^2-s^2+t^2}+u_\rho}} \leq 2\abs{t^2-s^2}$ for any $t \in \rpluszero$, it follows that
\begin{align*}
    \begin{split}
        \lim_{\rho \rightarrow \infty} &\ind{s\leq u_\rho} \rho^2   \exp\bigg(-\gamma v_2\int_{0}^\infty \Big(u_\rho+\sqrt{u_\rho^2-s^2+t^2}\Big)^2 \, \dx \markdist(t)\bigg) \\
        &= \lim_{\rho \rightarrow \infty} \ind{s\leq u_\rho} \rho^2  \exp\bigg(- 4 \gamma v_2 u_\rho^2 + \gamma v_2 s^2 -\gamma v_2\int_{0}^\infty t^2 + \textstyle\frac{2u_\rho(-s^2+t^2)}{\sqrt{u_\rho^2-s^2+t^2}+u_\rho} \, \dx \markdist(t)\bigg) \\
        &= \exp\big(- u - r_2 + \gamma v_2 s^2 - \gamma v_2 \e M^2\big) \lim_{\rho \rightarrow \infty} \ind{s\leq u_\rho}  \exp\bigg(-\gamma v_2\int_{0}^\infty \textstyle\frac{2u_\rho(-s^2+t^2)}{\sqrt{u_\rho^2-s^2+t^2}+u_\rho} \, \dx \markdist(t)\bigg) \\
        &= \exp\big(- u - r_2 +\gamma v_2 s^2 - \gamma v_2 \e M^2\big)  \exp\bigg(-\gamma v_2\int_{0}^\infty \lim_{\rho \rightarrow \infty} \textstyle\frac{2u_\rho(-s^2+t^2)}{\sqrt{u_\rho^2-s^2+t^2}+u_\rho} \, \dx \markdist(t)\bigg)\\
        &= \exp\big(- u - r_2 + 2\gamma v_2 s^2 - 2\gamma v_2 \e M^2\big) = \frac{1}{\gamma \e e^{2v_2\gamma M^2}}\exp\big(- u + 2\gamma v_2 s^2 \big) .
    \end{split}
\end{align*}
Altogether we get that 
\begin{align*}
    \lim_{\rho \rightarrow \infty} \mathbb{L}_{\rho}(B \times D \times [u,\infty)) 
    &= \lambda_2(B) e^{- u} \int_{D}  \frac{1}{\e e^{2v_2\gamma M^2}} e^{2\gamma v_2 s^2} \dx \markdist(s) = \lambda_2(B) \markdistlimtwo(D) e^{- u},
\end{align*}
which finishes the proof. 
\end{proof}

For $\Atilde>0$ let $\markdist_{\Atilde}$ be the distribution of the truncated mark $M_{\Atilde} := \min \{M, \Atilde\}$ and let
$$
\markdist_{2,\Atilde}^{LIM} := \frac{ \mathbb{E}[\ind{M_{\Atilde}\in \cdot } e^{2v_2\gamma M_{\Atilde}^2}]}{\mathbb{E}[e^{2v_2\gamma M_{\Atilde}^2}]}
$$
and define $\mathbb{M}_{\Atilde}:= \lambda_d\lvert_W \otimes \markdist_{2,\Atilde}^{LIM} \otimes \mathbb{K}_{\Atilde}$ with $\mathbb{K}_{\Atilde}([a,b)):= \mathbb{K}([a-d_{\Atilde},b-d_{\Atilde}))$, where
$$
d_{\Atilde}:= 2v_2\gamma(\e M_{\Atilde}^2-\e M^2) + \log \e e^{2v_2\gamma M^2} - \log \e e^{2v_2\gamma M_{\Atilde}^2}.
$$
Furthermore, let $\ppp_{\Atilde}$ be the version of $\ppp$, where marks larger than $\Atilde$ are replaced by $\Atilde$, i.e., $\ppp_{\Atilde}$ has the truncated mark distribution $\markdist_{\Atilde}$. Denote by $\zeta_{\Atilde} $ a Poisson process on $W \times [0,\infty)\times\R$ with intensity measure $\mathbb{M}_{\Atilde}$.
 
\begin{lemma} \label{lemma:ConvXiRhoAtilde}
    Assume that $d = 2$ and that \eqref{as:LightTailedMoments} holds. Let $\Atilde > \Mtilde > 0$ and $u \in \R$. Then 
    \begin{align}\label{ConvXiRhoAtilde}
    \xi_{\rho,\Atilde}:= \sum_{(x,m_x) \in \ppp_{\Atilde} \cap W_\rho \times [0,\Atilde]} \delta_{(\frac{1}{\rho}x, m, 4v_2\gamma r((x,m_x),\ppp_{\Atilde})^2 - \stworho)} &\overset{d}{\underset{\rho \rightarrow \infty}{\longrightarrow}} \zeta_{\Atilde} 
    \end{align}
    and
    \begin{align}\label{ConvLRhoAtilde}
    \mathbb{L}_{\rho,\Atilde}(W\times[0,\Mtilde]\times[u,\infty)):= \e \xi_{\rho,\Atilde}(W\times[0,\Mtilde]\times[u,\infty)) &\underset{\rho \rightarrow \infty}{\longrightarrow}  \mathbb{M}_{\Atilde}(W\times[0,\Mtilde]\times[u,\infty)).
\end{align}
\end{lemma}

\begin{proof}
    We know from Proposition \ref{prop:ExistenceCorResc} that the correct shift for $\ppp_{\Atilde}$ can be chosen as 
\begin{equation*}
    \stworhoAtilde:= \log \rho^2 - 2v_2\gamma\e M_{\Atilde}^2 + \log \e e^{2v_2\gamma M_{\Atilde}^2}+ \log \gamma.
\end{equation*} 
Denoting the correctly rescaled process of inradii corresponding to $\ppp_{\Atilde}$ by
\begin{equation*}
    \widetilde{\xi}_{\rho, \Atilde} := \sum_{(x,m_x) \in\ppp_{\Atilde} \cap W_\rho \times [0,\Atilde]} \delta_{(\frac{1}{\rho}x, m_x, 4v_2\gamma r((x,m_x),\ppp_{\Atilde} )^2 - \stworhoAtilde)},
\end{equation*}
 it is clear that $\xi_{\rho,\Atilde} = \widetilde{\xi}_{\rho, \Atilde} - (0,0,\stworho-\stworhoAtilde) = \widetilde{\xi}_{\rho, \Atilde} - (0,0,d_{\Atilde})$, where $d_{\Atilde}$ does not depend on $\rho$. Therefore, Theorem \ref{thm:bounded} implies \eqref{ConvXiRhoAtilde} and Lemma \ref{lem:intensity_measure_2} yields \eqref{ConvLRhoAtilde}
. 
\end{proof}

Recall that $\zeta_2$ denotes a Poisson process on $W \times [0,\infty)\times\R$ with intensity measure $\mathbb{M}_2=\lambda_d\lvert_W \otimes \markdistlimtwo \otimes \mathbb{K}$. We define $\mathcal{U}_1(\markdist):= \{D \subset [0,\infty): D\text{ relatively compact}, \markdist(\partial D)=0\}$. 

\begin{lemma} \label{lemma:ConvPoisLT}
Assume that $d = 2$ and that \eqref{as:LightTailedMoments} holds. Then
    \begin{equation} \label{ConvMAtildetoM2}
        \lim_{\Atilde \to \infty}\mathbb{M}_{\Atilde}(B\times D \times [u,\infty)) =\mathbb{M}_2(B\times D \times [u,\infty)),
    \end{equation}
    for $B \in \borel(W)$, $D \in \mathcal{U}_1(\markdist)$ and $u \in \R$ and
    \begin{equation} \label{auxconvd}
        \zeta_{\Atilde} \overset{d}{\underset{\Atilde \rightarrow \infty}{\longrightarrow}} \zeta_2 \quad \text{ in the space } \mathcal{M}_p(W \times [0,\infty)\times \R).
    \end{equation}
\end{lemma}

\begin{proof}
It holds that $\mathbb{M}_{\Atilde}(B\times D \times [u,\infty)) = \lambda_d(B) \markdist_{2,\Atilde}^{LIM}(D)e^{-u+d_{\Atilde}}$. Thanks to \eqref{as:LightTailedMoments}, the dominated convergence theorem can be used to show that
$d_{\Atilde}\rightarrow 0$ and $\markdist_{2,\Atilde}^{LIM}(D) \rightarrow \markdistlimtwo(D)$ as $\Atilde \rightarrow \infty$, which proves \eqref{ConvMAtildetoM2}. Since $\borel(W)\times  \mathcal{U}_1(\markdist)\times\{ [a,b): -\infty < a < b < \infty\}$ is a dissecting semi-ring, \eqref{auxconvd} is a consequence of Lemma \ref{lemma:AuxPoissonConv} and \eqref{ConvMAtildetoM2}. 
\end{proof}

Let $u \in \R$ and recall the notation $u_\rho := \max\{0,\frac{u + \stworho}{4\gamma v_2}\}^{1/2}$. 

\begin{lemma}\label{lemma:auxboundsLT2}
    Assume that $d = 2$ and that \eqref{as:LightTailedMoments} holds. Let $\Mtilde > 0$ and $u \in \R$. Then
    \begin{equation*}
        \lim_{\Atilde \rightarrow \infty} \lim_{\rho \rightarrow \infty} \e \sum_{\bfx \in \ppp \cap W_\rho \times [0,\Mtilde]} \ind{ r(\bfx,\ppp_{\Atilde}) \geq u_\rho > r(\bfx,\ppp)} = 0.
    \end{equation*}
\end{lemma}

\begin{proof}
For $\Atilde>\Mtilde$ we have $r(\bfx,\ppp_{\Atilde}) \geq r(\bfx,\ppp)$ for all $\bfx \in \ppp \cap W_\rho \times [0,\Mtilde]$ so that we can write
\begin{align*}
     \e &\sum_{\bfx \in \ppp \cap W_\rho \times [0,\Mtilde]} \ind{ r(\bfx,\ppp_{\Atilde}) \geq u_\rho > r(\bfx,\ppp)} \\
    & = \e \sum_{\bfx \in \ppp_{\Atilde} \cap W_\rho \times [0,\Mtilde]} \ind{ r(\bfx,\ppp_{\Atilde}) \geq u_\rho} - \e \sum_{\bfx \in \ppp \cap W_\rho \times [0,\Mtilde]} \ind{ r(\bfx,\ppp) \geq u_\rho} \\ 
     & = \e \xi_{\rho,\Atilde}(W\times[0,\Mtilde]\times[u,\infty)) - \e \xi_{\rho}(W\times[0,\Mtilde]\times[u,\infty)). 
\end{align*}
It follows from Lemma \ref{lemma:ConvIntMesureLT} that $\lim_{\rho \rightarrow \infty} \e \xi_{\rho}(W\times[0,\Mtilde]\times[u,\infty)) =  \mathbb{M}_2(W\times[0,\Mtilde]\times[u,\infty))$. Furthermore, \eqref{ConvLRhoAtilde} and \eqref{ConvMAtildetoM2} imply that also $\lim_{\Atilde \rightarrow \infty} \limsup_{\rho \rightarrow \infty} \e \xi_{\rho,\Atilde}(W\times[0,\Mtilde]\times[u,\infty)) = \mathbb{M}_2(W\times[0,\Mtilde]\times[u,\infty))$, which finishes the proof.
\end{proof}

\begin{lemma}\label{lemma:auxboundsLT}
    Assume that $d = 2$ and that \eqref{as:LightTailedMoments} holds. Let $\Mtilde > 0$ and $u \in \R$. Then
    \begin{equation*}
        \lim_{\Atilde \rightarrow \infty} \limsup_{\rho \rightarrow \infty} \e \sum_{\bfx \in \ppp \cap W_\rho \times [0,\Mtilde]} \ind{ r(\bfx,\ppp_{\Atilde}) > r(\bfx,\ppp) \geq u_\rho} = 0.
    \end{equation*}
\end{lemma}

\begin{proof}
Throughout the proof we assume that $\rho$ is large enough such that $u_\rho > \Mtilde$ and that $\Atilde > \Mtilde$. Using the Mecke equation and stationarity, we can write
 \begin{align*}
    \begin{split} 
         \e \sum_{\bfx \in \ppp \cap W_\rho \times [0,\Mtilde]} &\ind{r(\bfx,\ppp_{\Atilde}) > r(\bfx,\ppp) \geq u_\rho} \\
         &= \gamma \rho^2 \int_0^{\Mtilde} \p( r((0,m),\ppp_{\Atilde}+\delta_{(0,m)}) > r((0,m),\ppp+\delta_{(0,m)}) \geq u_\rho ) \, \dx \markdist(m).
    \end{split}
    \end{align*}
Denote by $\ppp_{\leq \Atilde}$ and $\ppp_{> \Atilde}$ the restrictions of $\ppp$ to $\rd \times [0,\Atilde]$ and $\rd \times (\Atilde,\infty)$, which are again marked Poisson processes. For $m \in [0,\Mtilde]$ define 
$$
R_{m,\leq \Atilde}:= r((0,m),\ppp_{\leq \Atilde}+\delta_{(0,m)}) \quad 
\text{and} \quad
R_{m,> \Atilde}:=r((0,m),\ppp_{> \Atilde}+\delta_{(0,m)}),
$$
which are independent due to the independence of $\ppp_{\leq \Atilde}$ and $\ppp_{> \Atilde}$. Let
$$
\widetilde{R}_{m,> \Atilde}:=r((0,m),\ppp_{\Atilde,=}+\delta_{(0,m)}),
$$
where $\ppp_{\Atilde,=}$ are the points of $\ppp_{\Atilde}$ whose marks were greater than $\Atilde$ before the truncation. Note that
\begin{equation}\label{eqn:identity_Rs}
r((0,m),\ppp+\delta_{(0,m)}) = \min\{ R_{m,\leq \Atilde}, R_{m,> \Atilde} \}
\end{equation}
and
$$
r((0,m),\ppp_{\Atilde}+\delta_{(0,m)}) = \min\{ R_{m,\leq \Atilde}, \widetilde{R}_{m,> \Atilde} \}
$$
since the inradius is determined by the closest hyperplane generating the cell. Moreover, one has $R_{m,> \Atilde} < \widetilde{R}_{m,> \Atilde}$ if $R_{m,> \Atilde}>0$ because truncating marks strictly increases the inradius of $(0,m)$. This means that $r((0,m),\ppp_{\Atilde}+\delta_{(0,m)} > r((0,m),\ppp+\delta_{(0,m)}) \geq u_\rho$ holds if and only if $R_{m,\leq \Atilde} > R_{m,> \Atilde} \geq u_\rho$. Therefore, we obtain
\begin{align*}
 & \p( r((0,m),\ppp_{\Atilde}+\delta_{(0,m)}) > r((0,m),\ppp+\delta_{(0,m)}) \geq u_\rho ) = \p(R_{m,\leq \Atilde} > R_{m,> \Atilde} \geq u_\rho) \\
 & = \p( R_{m,\leq \Atilde} > R_{m,> \Atilde} \geq u_\rho, R_{m,> \Atilde}<  u_\rho + \Atilde/ u_\rho) + \p(R_{m,\leq \Atilde} > R_{m,> \Atilde} \geq u_\rho + \Atilde/ u_\rho) \\
 & \leq \p( R_{m,> \Atilde} \geq u_\rho, R_{m,\leq \Atilde} \geq u_\rho, R_{m,> \Atilde}<  u_\rho + \Atilde/ u_\rho) + \p(R_{m,\leq \Atilde} > R_{m,> \Atilde} \geq u_\rho + \Atilde/ u_\rho) \\
 & = \p( R_{m,> \Atilde} \geq u_\rho, R_{m,\leq \Atilde}\geq u_\rho ) \p(R_{m,> \Atilde}<  u_\rho + \Atilde/ u_\rho \, | \, R_{m,> \Atilde}\ge  u_\rho ) \\
 & \quad + \p(R_{m,\leq \Atilde} > R_{m,> \Atilde} \geq u_\rho + \Atilde/ u_\rho) \\
 & \le \p(r((0,m),\ppp+\delta_{(0,m)})\ge u_\rho) \p(R_{m,> \Atilde}<  u_\rho + \Atilde/ u_\rho \, | \, R_{m,> \Atilde}\ge  u_\rho )\\
 & \quad + \p( r((0,m),\ppp+\delta_{(0,m)})\ge u_\rho+\Atilde/u_\rho),
\end{align*}
where we used the independence of $R_{m,\leq \Atilde}$ and $R_{m,> \Atilde}$ for the conditional probability and \eqref{eqn:identity_Rs} in the last step.
This yields
\begin{align*}
& \e \sum_{\bfx \in \ppp \cap W_\rho \times [0,\Mtilde]} \ind{ r(\bfx,\ppp_{\Atilde}) > r(\bfx,\ppp) \geq u_\rho} \\
& \leq \gamma \rho^2 \int_0^{\Mtilde} \p(r((0,m),\ppp+\delta_{(0,m)})\ge u_\rho) \p(R_{m,> \Atilde}<  u_\rho + \Atilde/ u_\rho \, | \, R_{m,> \Atilde}\ge  u_\rho ) \, \dx \markdist(m) \\
& \quad + \gamma \rho^2 \int_0^{\Mtilde} \p( r((0,m),\ppp+\delta_{(0,m)})\ge u_\rho+\Atilde/u_\rho) \, \dx \markdist(m) \\
& =: I_1(\rho,\Atilde) + I_2(\rho,\Atilde).
\end{align*}

Using the Mecke equation and stationarity, the second integral can be rewritten as
\begin{align*}
I_2(\rho,\Atilde) &= \e \sum_{\bfx \in \ppp \cap W_\rho \times [0,\Mtilde]} \ind{ r(\bfx,\ppp) \ge  u_\rho + \Atilde/ u_\rho} \\
& \le \e \sum_{\bfx \in \ppp \cap W_\rho \times [0,\Mtilde]} \ind{ r(\bfx,\ppp)^2 \ge  u_\rho^2 + 2\Atilde}\\
& = \e \sum_{\bfx \in \ppp \cap W_\rho \times [0,\Mtilde]} \ind{ 4v_2\gamma r(\bfx,\ppp)^2-\stworho \ge u + 8v_2\gamma\Atilde} = \mathbb{L}_{\rho}(W\times[0,\Mtilde]\times[u+ 8v_2\gamma\Atilde,\infty)).
\end{align*}
Now Lemma \ref{lemma:ConvIntMesureLT} yields
$$
\limsup_{\rho\to\infty} I_2(\rho,\Atilde) \leq \mathbb{M}(W\times[0,\Mtilde]\times[u+8v_2\gamma\Atilde,\infty)) \leq e^{-u-8v_2\gamma\Atilde}
$$
so that
$$
\lim_{\Atilde\to\infty} \limsup_{\rho\to\infty} I_2(\rho,\Atilde) = 0.
$$

Next we consider $I_1(\rho,\Atilde)$. Since $\ppp_{>\Atilde}$ has the intensity $\gamma\p(M>\Atilde)$ and the mark distribution $\markdist(\cdot\cap(\Atilde,\infty))/\p(M>\Atilde)$, we derive from Lemma \ref{lemma:EmptySetFormula} that, for $m\in[0,\Mtilde]$ and $a\in(m,\infty)$,
$$
\p(R_{m,> \Atilde} \ge a) = \p( r((0,m), \ppp_{>\Atilde}) \ge a) = \exp\bigg( -v_2\gamma  \int_{\Atilde}^\infty \Big( a +\sqrt{a^2-m^2+t^2} \Big)^2 \, \dx\markdist(t)\bigg).
$$
This means that
\begin{align*}
& \p( R_{m,> \Atilde} < u_\rho + \Atilde/ u_\rho \, | \, R_{m,> \Atilde} \ge u_\rho ) \\
& = \frac{\p( u_\rho \le R_{m,> \Atilde} < u_\rho + \Atilde/ u_\rho)}{\p(R_{m,> \Atilde} \ge u_\rho)} \\
& = 1 - \exp\bigg( -v_2\gamma  \int_{\Atilde}^\infty \Big( u_\rho + \Atilde/ u_\rho +\sqrt{(u_\rho + \Atilde/ u_\rho)^2-m^2+t^2} \Big)^2 - \Big( u_\rho +\sqrt{u_\rho^2-m^2+t^2} \Big)^2 \, \dx\markdist(t) \bigg) \\
& \le v_2\gamma  \int_{\Atilde}^\infty \Big( u_\rho + \Atilde/ u_\rho +\sqrt{(u_\rho + \Atilde/ u_\rho)^2-m^2+t^2} \Big)^2 - \Big( u_\rho +\sqrt{u_\rho^2-m^2+t^2} \Big)^2 \, \dx\markdist(t).
\end{align*}
Since
\begin{align*}
& \Big( u_\rho + \Atilde/ u_\rho +\sqrt{(u_\rho + \Atilde/ u_\rho)^2-m^2+t^2} \Big)^2 - \Big( u_\rho +\sqrt{u_\rho^2-m^2+t^2} \Big)^2 \\
& \le \Big( 2u_\rho + \Atilde/ u_\rho +\sqrt{(u_\rho + \Atilde/ u_\rho)^2-m^2+t^2} + \sqrt{u_\rho^2-m^2+t^2} \Big) \\
& \quad \times \Big( \Atilde/ u_\rho + \sqrt{(u_\rho + \Atilde/ u_\rho)^2-m^2+t^2} - \sqrt{u_\rho^2-m^2+t^2} \Big) \\
& \le \Big( 2u_\rho + \Atilde/ u_\rho +2\sqrt{(u_\rho + \Atilde/ u_\rho)^2+t^2} \Big) 
\bigg( \frac{\Atilde}{u_\rho}+\frac{(u_\rho + \Atilde/ u_\rho)^2 - u_\rho^2}{\sqrt{(u_\rho + \Atilde/ u_\rho)^2-m^2+t^2} + \sqrt{u_\rho^2-m^2+t^2}}\bigg) \\
& \le \Big( 2u_\rho + \Atilde/ u_\rho +2(u_\rho + \Atilde/ u_\rho)+2t \Big) 
\bigg( \frac{\Atilde}{u_\rho} + \frac{2\Atilde + (\Atilde/ u_\rho)^2}{2 \sqrt{u_\rho^2-\Mtilde^2}}\bigg),
\end{align*}
 we see that
\begin{align*}
& \p( R_{m,> \Atilde} < u_\rho + \Atilde/ u_\rho \, | \, R_{m,> \Atilde} \ge u_\rho ) \\
& \le v_2\gamma \Big( (4u_\rho+3\Atilde/u_\rho) \p(M\ge \Atilde) +2\e\left[ \ind{M\ge \Atilde} M\right] \Big) \bigg( \frac{\Atilde}{u_\rho}+\frac{2\Atilde + (\Atilde/ u_\rho)^2}{2 \sqrt{u_\rho^2-\Mtilde^2}}\bigg) =: U_\rho(\Atilde)
\end{align*}
for all $m\in[0,\Mtilde]$. Inserting this into $I_1(\rho,\Atilde)$ leads to
\begin{align*}
I_1(\rho,\Atilde) & \le U_\rho(\Atilde) \gamma \rho^2 \int_0^{\Mtilde} \p( r((0,m),\ppp+\delta_{(0,m)}) \ge u_\rho) \, \dx \markdist(m) = U_\rho(\Atilde) \mathbb{L}_{\rho}(W\times[0,\Mtilde]\times[u,\infty)) \\
& \le U_\rho(\Atilde) \mathbb{L}_{\rho}(W\times[0,\infty)\times[u,\infty)),
\end{align*}
where the integral is computed similarly as $I_2(\rho,\Atilde)$ above. Note that
$$
\lim_{\rho\to\infty} U_\rho(\Atilde) = 8 v_2\gamma \p(M\ge \Atilde)\Atilde.
$$
Together with Lemma \ref{lemma:ConvIntMesureLT} we get
$$
\limsup_{\rho\to\infty} I_1(\rho,\Atilde) \le 8 v_2\gamma e^{-u} \p(M\ge \Atilde)\Atilde.
$$
By \eqref{as:LightTailedMoments} the right-hand side vanishes as $\Atilde\to\infty$ so that
$$
\lim_{\Atilde\to\infty} \limsup_{\rho\to\infty} I_2(\rho,\Atilde) = 0,
$$
which completes the proof.
\end{proof}

Now we are ready to prove our main result.  

\begin{proof}[Proof of Theorem \ref{thm:lighttails}]
First, it follows from Lemma \ref{lemma:ConvIntMesureLT} that $\stworho$ chosen as in \eqref{CRtwo} is a correct shift, proving the existence claim. Furthermore, it can be proved by the same argument as Corollary \ref{cor:OrderofCR} a) that any other correct shift $\{\tilde{\mathrm{s}}_{2,\rho}\}_{\rho \geq 1}$ satisfies $\lim_{\rho \rightarrow \infty} \stworho - \tilde{\mathrm{s}}_{2,\rho} = 0$. Therefore, to prove Theorem \ref{thm:lighttails}, it is enough to consider \eqref{conv:bounded} and \eqref{conv:boundedcor} for the correct shift $\stworho$ chosen as in \eqref{CRtwo}.

In the sequel the notation $\lvert_{\Mtilde,u}$ indicates the restriction of a point process on $W \times [0,\infty) \times \R$ to $W \times [0,\Mtilde] \times [u,\infty)$. It follows from \cite[Theorem 23.16]{bo:OKallen21} that in order to prove 
    \begin{equation}\label{convlighttails}
        \xi_\rho \, \overset{d}{\underset{\rho \rightarrow \infty}{\longrightarrow}} \zeta_2 \text{ in the space } \mathcal{M}_p(W\times[0,\infty)\times\R),
    \end{equation} it is enough to show that, for any $\Mtilde > 0$ and any $u \in \R$,
    \begin{equation}\label{convlighttails_restricted}
       \xi_\rho\lvert_{\Mtilde,u} \, \overset{d}{\underset{\rho \rightarrow \infty}{\longrightarrow}} \zeta_2\lvert_{\Mtilde,u} \text{ in the space } \mathcal{M}_p(W\times[0,\infty)\times\R).
    \end{equation}    
    Fix $\Mtilde > 0$ and $u \in \R$ and let $\Atilde > \Mtilde$. Whenever $\bfx = (x,m_x) \in \ppp$ is such that $m_x < \Atilde$, it holds that $r(\bfx,\ppp) \leq r(\bfx, \ppp_{\Atilde})$. Since $\Mtilde < \Atilde$, we get that 
\begin{align} \label{eventbound}
\begin{split}
    \{\xi_\rho\lvert_{\Mtilde,u} \neq \xi_{\rho,\Atilde}\lvert_{\Mtilde,u}\} &\subseteq \{\exists\, \bfx \in \ppp\cap W_\rho \times [0,\Mtilde]: r(\bfx, \ppp_{\Atilde}) > r(\bfx,\ppp) \geq u_\rho\} \\
    & \qquad \qquad \cup \{\exists\, \bfx \in \ppp\cap W_\rho \times [0,\Mtilde]: r(\bfx, \ppp_{\Atilde}) \geq u_\rho > r(\bfx,\ppp) \}.
    \end{split}
\end{align}
We can bound 
\begin{align*}
    \p(\exists\, \bfx \in \ppp\cap W_\rho \times [0,\Mtilde] : r(\bfx, \ppp_{\Atilde}) \geq u_\rho > r(\bfx,\ppp))  \leq \e \sum_{\bfx \in \ppp \cap W_\rho \times [0,\Mtilde]} \ind{ r(\bfx,\ppp_{\Atilde}) \geq u_\rho > r(\bfx,\ppp)}.
\end{align*}
Therefore, it follows from Lemma \ref{lemma:auxboundsLT2} that
\begin{equation*}
    \lim_{\Atilde \rightarrow \infty} \limsup_{\rho \rightarrow \infty}  \p(\exists\, \bfx \in \ppp\cap W_\rho \times [0,\Mtilde] : r(\bfx, \ppp_{\Atilde}) \geq u_\rho > r(\bfx,\ppp)) = 0.
\end{equation*} 
Similarly, Lemma \ref{lemma:auxboundsLT} yields that 
\begin{equation*}
    \lim_{\Atilde \rightarrow \infty} \limsup_{\rho \rightarrow \infty} \p(\exists\, \bfx \in \ppp\cap W_\rho \times [0,\Mtilde]: r(\bfx, \ppp_{\Atilde}) > r(\bfx,\ppp) \geq u_\rho) = 0,
\end{equation*}
which together with \eqref{eventbound} implies $\lim_{\Atilde \rightarrow \infty} \limsup_{\rho \rightarrow \infty}  \p(\xi_\rho\lvert_{\Mtilde,u} \neq \xi_{\rho,\Atilde}\lvert_{\Mtilde,u})=0$. Combining this with
$$
\xi_{\rho,\Atilde}\lvert_{\Mtilde,u} \overset{d}{\underset{\rho \rightarrow \infty}{\longrightarrow}} \zeta_{\Atilde}\lvert_{\Mtilde,u} \quad \text{and} \quad \zeta_{\Atilde}\lvert_{\Mtilde,u} \overset{d}{\underset{\Atilde \rightarrow \infty}{\longrightarrow}} \zeta_2\lvert_{\Mtilde,u},
$$
which follow from \eqref{ConvXiRhoAtilde} in Lemma \ref{lemma:ConvXiRhoAtilde} and \eqref{auxconvd} in Lemma \ref{lemma:ConvPoisLT}, a standard approximation argument (see \cite[Theorem 5.29]{bo:OKallen21}) yields \eqref{convlighttails_restricted} and, thus, \eqref{convlighttails}. 

The proof of the convergence in distribution of the maximal inradius has the same structure as in the bounded case - first we prove, by Lemma \ref{lemma:continuity of argmax}, the convergence in distribution of the maximal inradius for processes restricted to compact sets and then use approximation arguments. 

Let $b > 0$ and define $\mathcal{C}_b:= \mathcal{C}_{W\times[0,b]\times[-b,b]}$ as in \eqref{def:SubsetCK}. Denote by $\xi_\rho^{b}$ and $\zeta_2^b$ the restrictions of $\xi_\rho$ and $\zeta_2$ to $W\times[0,b]\times[-b,b]$. It follows from \eqref{convlighttails}, Lemma \ref{lemma:continuity of argmax} and the continuous mapping theorem (see  \cite[Theorem 5.27]{bo:OKallen21}) that 
\begin{equation}\label{ArgmaxConvb}
    T(\xi_\rho^b) \overset{d}{\underset{\rho \rightarrow \infty}{\longrightarrow}} T(\zeta_2^b).
\end{equation}
It holds that 
\begin{align*}
    \{T(\zeta_2^b) \neq T(\zeta_2)\}\subseteq \{\zeta_2(W\times[0,\infty)\times(b,\infty))\geq 1 \}  &\cup \{\zeta_2(W\times(b,\infty)\times[-b,b])\geq 1\} \\
    & \cup \{\zeta_2(W\times[0,b]\times[-b,b])=0\}
\end{align*}
and analogously for $\xi_\rho$ and $\xi_\rho^b$ for any $\rho \geq 1$. Furthermore, one has
\begin{align*}
        0\leq \limsup_{b \rightarrow \infty} \markdistlimtwo((b,\infty)) \cdot e^b & = \limsup_{b \rightarrow \infty} \frac{\e [\ind{M > b} \cdot e^{2v_2\gamma M^2 + b}]}{\e[e^{2v_2\gamma M^2}]} \\
        & \leq \limsup_{b \rightarrow \infty} \frac{\e [\ind{M > b} \cdot e^{2v_2\gamma M^2 + M}]}{\e[e^{2v_2\gamma M^2}]}  = 0
\end{align*}
by the dominated convergence theorem since $\e [e^{2v_2\gamma M^2 + M}] < \infty$ by \eqref{as:LightTailedMoments}. Therefore, we obtain
\begin{align}
\begin{split}\label{limit1}
    0&\leq \limsup_{b \rightarrow \infty} \p(T(\zeta_2^b) \neq T(\zeta_2)) \\ &\leq 
    \lim_{b \rightarrow \infty} \p(\zeta_2(W\times[0,\infty)\times(b,\infty))\geq 1) \\
    & \hspace{3cm}+ \p(\zeta_2(W\times(b,\infty)\times[-b,b])\geq 1) + \p(\zeta_2(W\times[0,b]\times[-b,b])=0)\\
    &= \lim_{b \rightarrow \infty} 1-e^{-e^{-b}} + 1-e^{-\markdistlimtwo((b,\infty)) \cdot (e^b-e^{-b})} + e^{-\markdistlimtwo([0,b]) \cdot (e^{b}-e^{-b})} = 0.
\end{split}
\end{align}
Similarly, using Lemma \ref{lemma:ConvIntMesureLT} and the convergence \eqref{convlighttails}, we get that 
\begin{align}
\begin{split}\label{limit2}
    0&\leq \limsup_{b \rightarrow \infty} \limsup_{\rho \rightarrow\infty} \p(T(\xi_\rho^b) \neq T(\xi_\rho)) \\ 
    &\leq \limsup_{b \rightarrow \infty} \limsup_{\rho \rightarrow\infty} \p(\xi_\rho(W\times[0,\infty)\times(b,\infty))\geq 1) \\
    & \hspace{3cm} + \p(\xi_\rho(W\times(b,\infty)\times[-b,b])\geq 1) +  \p(\xi_\rho(W\times[0,b]\times[-b,b])=0)\\
    &\leq \limsup_{b \rightarrow \infty} \lim_{\rho \rightarrow\infty} \mathbb{L}_\rho(W\times[0,\infty)\times[b,\infty)) \\
    & \hspace{3cm} + \mathbb{L}_\rho(W\times(b,\infty)\times[-b,b]) +\p(\xi_\rho(W\times[0,b]\times[-b,b])=0)\\
    &= \lim_{b \rightarrow \infty} \mathbb{M}_2(W\times[0,\infty)\times[b,\infty))  \\
    & \hspace{3cm} +  \mathbb{M}_2(W\times(b,\infty)\times[-b,b]) + \p(\zeta_2(W\times[0,b]\times[-b,b])=0) \\
    &= \lim_{b \rightarrow \infty} e^{-b} + \markdistlimtwo((b,\infty)) \cdot (e^b-e^{-b}) + e^{-\markdistlimtwo([0,b]) \cdot (e^{b}-e^{-b})}=0.
\end{split}
\end{align}
The convergence \eqref{ArgmaxConvb} together with the limits \eqref{limit1} and \eqref{limit2} imply that 
\begin{equation*}
    T(\xi_\rho) \overset{d}{\underset{\rho \rightarrow \infty}{\longrightarrow}} T(\zeta_2)
\end{equation*}
and the distribution of $T(\zeta_2)$ can be determined as in the proof of Corollary \ref{cor:bounded}. 
\end{proof}

\section{Proof of Theorem \ref{thm:powerlaw}} \label{sec:ProofConvPL}

In this section we prove the convergence in distribution of the suitably rescaled point process of large inradii, their locations and marks in the case with heavy-tailed marks, adapting the proof strategy from \cite{ar:BS22} and \cite{ar:LS25} for statistics of random graphs with weights to a Laguerre setting. The idea is that there is a one-to-one correspondence between large inradii and points with large weights. The main step of the proof is to show that the correctly rescaled process of inradii behaves asymptotically in the same way as the (correctly rescaled) process of weights.

Let $\rho \geq 1$ and recall the notation $\qalpharho = \gamma^{1/\alpha} q(\rho)$, where $q(\rho)$ denotes the $(1-\frac{1}{\rho^d})$-quantile of the mark distribution $\markdist$. It follows from \cite[Theorem 3.6 and Remark 3.3]{bo:Resnick07} that for any $u > 0$ we get
\begin{equation}\label{limit:pltaild}
    \lim_{\rho \rightarrow \infty} \rho^d \markdist ((\qalpharho u, \infty)) = \lim_{\rho \rightarrow \infty} \rho^d \markdist ((q(\rho) \gamma^{1/\alpha}  u, \infty)) =  \frac{1}{\gamma u^\alpha}.
\end{equation}

From $\ppp \cap W_\rho \times (0,\infty)$ we construct the point processes
\begin{align*}
    \Theta_\rho &:= \sum_{(x,m_x) \in \ppp \cap \Wrho \times (0,\infty)} \delta_{(\frac{1}{\rho} x, \frac{1}{\qalpharho}m_x, \frac{1}{\qalpharho}m_x)}  
\end{align*}
for $\rho\ge 1$. The proof of Theorem \ref{thm:powerlaw} is based on the following lemma, which shows that the point processes $\Theta_\rho$ and $\Xi_\rho$ (recall the notation \eqref{def:XirhoPL}) behave asymptotically the same. 

\begin{lemma} \label{lemma:MainConvergenceP}
    Let $u > 0$ and let $B \subseteq W \times (0,\infty]$ be measurable. Then 
    \begin{equation} \label{conv:L1}
        \e \abs{\Theta_\rho(B \times (u,\infty])-\Xi_\rho(B \times (u,\infty])} \underset{\rho \rightarrow \infty}{\longrightarrow}0.
    \end{equation}
    Particularly, one has for all $0< a<b\leq \infty$ that
    \begin{equation}\label{conv:inP}
       \abs{\Theta_\rho(B\times(a,b])-\Xi_\rho(B\times(a,b])} \overset{\p}{\underset{\rho \rightarrow \infty}{\longrightarrow}}0.
    \end{equation}
\end{lemma}

\begin{proof}
For simplicity, we write $r_x(\ppp) := r((x,m_x),\ppp)$ in the following. Let $\veps \in (0,u)$. Then we can bound
\begin{align*}
& \abs{\Theta_\rho(B \times(u,\infty])-\Xi_\rho(B \times(u,\infty])} \\
& = \abs{ \sum_{(x,m_x) \in \ppp \cap \Wrho \times (0,\infty)} \ind{ \left(\frac{1}{\rho} x, \frac{1}{\qalpharho}m_x\right)\in B }  \bigg( \ind{\frac{1}{\qalpharho}m_x>u} - \ind{r_x(\eta)>u} \bigg) } \\
& \le\sum_{(x,m_x) \in \ppp \cap \Wrho \times (0,\infty)} \big( \ind{ m_x > \qalpharho u,\, r_x(\ppp) \leq \qalpharho u} + \ind { m_x \leq \qalpharho u,\, r_x(\ppp) > \qalpharho u,} \big) \\
& \le I_{1,\rho}(\veps) + I_{2,\rho}(\veps) + I_{3,\rho}(\veps),
\end{align*}
    where
    \begin{align*}
        I_{1,\rho}(\veps) &:= \sum_{(x,m_x) \in \ppp \cap \Wrho \times (0,\infty)} \ind{ m_x \in (\qalpharho(u-\veps),\qalpharho(u+\veps))}, \\
        I_{2,\rho}(\veps) &:= \sum_{(x,m_x) \in \ppp \cap \Wrho \times (0,\infty)} \ind{ m_x > \qalpharho u,\, r_x(\ppp) \leq \qalpharho u,\, \abs{m_x-r_x(\ppp)} > \veps \qalpharho }, \\
        I_{3,\rho}(\veps) &:= \sum_{(x,m_x) \in \ppp \cap \Wrho \times (0,\infty)} \ind { m_x \leq \qalpharho u,\, r_x(\ppp) > \qalpharho u,\, \abs{m_x-r_x(\ppp)} > \veps\qalpharho }. 
    \end{align*}
Now we show that $\lim_{\veps \rightarrow 0^+} \limsup_{\rho \rightarrow \infty}\e I_{j,\rho}(\veps) = 0$ for $j \in \{1,2,3\}$, which finishes the proof.

Using the Mecke formula, we obtain
\begin{equation*}
    \e I_{1,\rho}(\veps) = \gamma\rho^d \markdist ((\qalpharho(u-\veps),\qalpharho(u+\veps))).
\end{equation*} 
From \eqref{limit:pltaild} we get that 
\begin{align*}
    \lim_{\veps \rightarrow 0^+} \lim_{\rho \rightarrow \infty} \e I_{1,\rho}(\veps) 
    =  \lim_{\veps \rightarrow 0^+}  \frac{1}{(u-\veps)^\alpha} - \frac{1}{ (u+\veps)^\alpha} = 0. 
\end{align*}

Next we consider $I_{3,\rho}(\veps)$. Let $M \sim \markdist$ be the typical mark and independent of $\ppp$, and write for simplicity $r(s,\ppp) := r((0,s),\ppp + \delta_{(0,s)})$ for any $s \geq 0$. Let $s \geq 0$. Recall that $B^d(0,r,R)$ denotes the anulus with center $0$ and inradii $r < R$ and recall the notation for the set $D(\cdot,\cdot)$ from Lemma \ref{lemma:EmptySetFormula}. Because of $B^d(0,s,s+\veps \qalpharho)\times (0,\infty) \subseteq D(s,s+\varepsilon \qalpharho)$, by Lemma \ref{lemma:EmptySetFormula}, the implication  
\begin{equation}\label{impl1}
    r(s,\ppp) > s + \veps \qalpharho \implies \ppp\bigl(B^d(0,s,s+\veps \qalpharho)\times (0,\infty) \bigr) = 0
\end{equation}
holds \as Using the Mecke formula together with \eqref{impl1}, we can show that 
\begin{align*}
    \e I_{3,\rho}(\veps) &= \gamma \rho^d \p ( M \leq \qalpharho u,\, r(M,\ppp) > \qalpharho u,\, \abs{M-r(M,\ppp)} > \veps \qalpharho) \\
    &= \gamma\rho^d \p ( M \leq \qalpharho u,\, r(M,\ppp) > \qalpharho u,\, r(M,\ppp) - M > \veps \qalpharho ) \\
    &\leq \gamma\rho^d \p ( M \leq \qalpharho u,\,  r(M,\ppp) > M + \veps \qalpharho ) = \gamma\rho^d \int_{0}^{u\qalpharho } \p (r(s,\ppp) > s + \veps \qalpharho ) \, \dx \markdist(s) \\
    &\leq \gamma\rho^d  \int_{0}^{u\qalpharho } \p (\ppp\bigl(B^d(0,s,s+\veps \qalpharho )\times (0,\infty) \bigr) = 0) \, \dx \markdist(s) \\
    &= \gamma\rho^d  \int_{0}^{u\qalpharho } \exp{(-\gamma v_d((s+\veps \qalpharho )^d-s^d))} \, \dx \markdist(s) \\ 
    &\leq \gamma\rho^d  \int_{0}^{u\qalpharho } \exp{(-\gamma v_d(\veps \qalpharho )^d)} \, \dx \markdist(s) \leq \gamma \exp{(- \gamma v_d \veps^d \qalpharho ^d +  \log(\rho^d))}.
\end{align*}
It is therefore enough to prove that for any $\theta > 0$ we have that $\lim_{\rho \rightarrow \infty} \theta q(\rho)^d - \log(\rho^d) = \infty$. However, this follows from the fact that $q(\cdot)$ is a regularly varying function with index $\frac{d}{\alpha}$ (see \cite[Lemma 3.3]{ar:BS22} or \cite[Remark 3.3]{bo:Resnick07}) and \cite[Proposition 2.6 (i)]{bo:Resnick07}. 

It remains to consider $I_{2,\rho}(\varepsilon)$. Let $s\geq \qalpharho u$.  Recalling Lemma \ref{lemma:EmptySetFormula}, it \as holds that
\begin{align*}
r(s,\ppp)< s-\veps \qalpharho & \implies \exists(y,m_y) \in \ppp \,:\, (y,m_y)\in D(s,s-\veps \qalpharho) \\
& \implies \exists(y,m_y) \in \ppp \,:\, m_y^2 > \|y\|^2 - 2 (s-\veps \qalpharho) \|y\| + s^2.
\end{align*}
Furthermore, we can write 
\begin{align*}
\|y\|^2 - 2 (s-\veps \qalpharho) \|y\| + s^2 &= (\|y\|-s)^2 + 2 \veps \qalpharho \|y\| = (\|y\|- (s-\veps \qalpharho) )^2 +2 s \veps \qalpharho  - (\veps \qalpharho)^2,
\end{align*}
which implies that 
\begin{equation}\label{impl2}
    r(s,\ppp) < s-\veps \qalpharho  \implies \exists(y,m_y) \in \ppp \,:\, m_y^2 \geq 2s\veps \qalpharho  - (\veps \qalpharho )^2 \text{ and } B^d(0,s) \cap B(y,m_y) \neq \emptyset
\end{equation}
holds \as The Mecke formula and \eqref{impl2} lead to
\begin{align*}
    \e I_{2,\rho}(\veps) &= \gamma \rho^d \p ( M > \qalpharho u,\, r(M,\ppp) \leq \qalpharho u,\, \abs{M-r(M,\ppp)} > \veps \qalpharho ) \\
    &= \gamma \rho^d \p ( M > \qalpharho u,\, r(M,\ppp) \leq \qalpharho u,\, M - \veps \qalpharho  > r(M,\ppp))\\
    &\overset{\eqref{impl2}}{\leq} \gamma \rho^d \p (M > \qalpharho u,\, \exists(y,m_y) \in \ppp \,:\, m_y^2 \geq 2M\veps \qalpharho  - (\veps \qalpharho )^2 \text{ and } B^d(0,M) \cap B(y,m_y) \neq \emptyset)\\
    &= \gamma \rho^d \int_{u \qalpharho }^\infty \p (\exists(y,m_y) \in \ppp \,:\, m_y^2 \geq 2s\veps \qalpharho  - (\veps \qalpharho )^2 \text{ and } B^d(0,s) \cap B(y,m_y) \neq \emptyset)\, \dx \markdist(s).
\end{align*}
    Now, fix $s \geq u\qalpharho $ and recall that $\veps < u$. Therefore, we have $2s\veps \qalpharho  - (\veps \qalpharho )^2\geq (\veps \qalpharho )^2$ so that we can bound 
\begin{align*}
    \p (\exists(y,m_y) \in \ppp \,&:\, m_y^2 \geq 2s\veps \qalpharho  - (\veps \qalpharho )^2 \text{ and } B^d(0,s) \cap B(y,m_y) \neq \emptyset) \\
    &\leq  \p (\exists(y,m_y) \in \ppp \,:\, m_y \geq \veps \qalpharho  \text{ and } B^d(0,s) \cap B(y,m_y) \neq \emptyset)\\
    &\leq \e \sum_{(y,m_y) \in \ppp}\ind{m_y \geq \veps \qalpharho ,\, B^d(0,s) \cap B(y,m_y) \neq \emptyset}.
\end{align*}
Using the Mecke formula again and the Jensen inequality, we obtain 
\begin{align*}
    \e &\sum_{(y,m_y) \in \ppp}\ind{ m_y \geq \veps \qalpharho ,\, B^d(0,s) \cap B(y,m_y) \neq \emptyset} \\
    &= \gamma \int_{\veps \qalpharho }^\infty \int_{\rd} \ind {B^d(0,s)\cap B^d(y,t) \neq \emptyset}\,\dx y \, \dx \markdist(t)   = \gamma \int_{\veps \qalpharho }^\infty v_d(s+t)^d \, \dx \markdist(t)\\
    &\leq \gamma \int_{\veps \qalpharho }^\infty 2^{d-1}v_d(s^d+t^d) \, \dx \markdist(t) = 2^{d-1}\gamma v_d\bigl( s^d\markdist((\veps \qalpharho ,\infty)) + \e\bigl[ M^d \ind{M\geq \veps \qalpharho }\bigr]\bigr).
\end{align*}
Putting everything together leads to
\begin{align*}
    \e I_{2,\rho}(\veps) &\leq \gamma^2\rho^d \int_{u \qalpharho }^\infty  2^{d-1}v_d\bigl( s^d\markdist((\veps \qalpharho ,\infty)) + \e\bigl[ M^d \ind{M\geq \veps \qalpharho } \bigr]\bigr) \, \dx \markdist(s)\\
    & = 2^{d-1} \gamma^2 \rho^d v_d \Bigl( \markdist((\veps \qalpharho ,\infty)) \e\bigl[ M^d \ind{M\geq u \qalpharho } \bigr] +  \e\bigl[ M^d \ind{M\geq \veps \qalpharho } \bigr] \markdist((u \qalpharho ,\infty)\Bigr)\\
    &\leq 2^d\gamma^2 \rho^d v_d \markdist((\veps \qalpharho ,\infty)) \e\bigl[ M^d \ind{M\geq \veps \qalpharho } \bigr].
\end{align*} 
Together with \eqref{limit:pltaild}, we get for fixed $\veps \in (0,u)$ that
\begin{equation*}
    0 \leq \limsup_{\rho \rightarrow \infty} \e I_{2,\rho}(\varepsilon) \leq \limsup_{\rho \rightarrow \infty} 2^d\gamma^2 v_d \rho^d \markdist((\veps \qalpharho ,\infty)) \e\bigl[ M^d \ind{M\geq \veps \qalpharho } \bigr] = 0,
\end{equation*}
where the last equality holds since, thanks to the assumption that $\alpha > d$, we can use the dominated convergence theorem to show that $\lim_{\rho \rightarrow \infty} \e\bigl[ M^d \ind{M\geq \veps \qalpharho } \bigr] = 0$. This proves \eqref{conv:L1}. Now \eqref{conv:inP} follows from \eqref{conv:L1} and the Markov inequality since we can bound
\begin{align*}
    &\abs{\Theta_\rho(B\times(a,b])-\Xi_\rho(B\times(a,b])} \\
    & \hspace{2cm} \leq \abs{\Theta_\rho(B\times(a,\infty])-\Xi_\rho(B\times(a,\infty])} + \abs{\Theta_\rho(B\times(b,\infty])-\Xi_\rho(B\times(b,\infty])}
\end{align*}
for all $0< a<b\leq \infty$.
\end{proof}

Recall that $\Psi_\alpha$ denotes the Poisson process on $W \times (0,\infty]^2$ with intensity measure $\mathbb{M}_\alpha = \lambda_d\lvert_W \otimes \mathbb{K}_\alpha $ where $\mathbb{K}_\alpha(\{(x,x): a < x \leq b\}) = a^{-\alpha}-b^{-\alpha}$ for all $0<a<b$ and $\mathbb{K}_\alpha(\{(x,x): x > 0\}^c)=0$.

\begin{proof}[Proof of Theorem \ref{thm:powerlaw}]
We start by proving that  
\begin{equation} \label{PoissonConvPL}
    \Theta_\rho \overset{d}{\underset{\rho \rightarrow \infty}{\longrightarrow}} \Psi_\alpha \text{ in } \mathcal{M}_p(W\times (0,\infty]^2).
\end{equation}  
For $\rho \geq 1$, $\Theta_\rho$ is a Poisson process with intensity measure $\mathbb{L}_\rho$ satisfying
$$
\mathbb{L}_\rho(B \times (a_1,b_1]\times (a_2,b_2]) = \gamma \lambda_d(B) \rho^d \markdist((\qalpharho a,\qalpharho b])
$$
for $B \subseteq W$ measurable, $0<a_i <b_i\leq \infty$ for $i \in\{ 1,2\}$ and $a:= \max \{a_1,a_2\} < b:= \min\{b_1,b_2\}$. It follows from \eqref{limit:pltaild} and from the formula for the intensity measure $\mathbb{M}_\alpha$ of $\Psi_\alpha$ that
\begin{align*}
    \mathbb{L}_\rho(B \times (a_1,b_1]\times (a_2,b_2]) 
   \underset{\rho \rightarrow \infty}{\longrightarrow}  \lambda_d(B) \cdot (a^{-\alpha}-b^{-\alpha}) = \mathbb{M}_\alpha(B \times (a_1,b_1]\times(a_2,b_2]).
\end{align*}
Define $\mathcal{I}_1:= \{(a,b]: 0 < a < b \leq \infty\}$ and $\mathcal{I}:= \{B \subseteq W: \text{ measurable}\} \times \mathcal{I}_1^2$. Then $\mathcal{I}$ is a dissecting semi-ring of the space $W\times(0,\infty]^2$. Lemma \ref{lemma:AuxPoissonConv} yields that \eqref{PoissonConvPL} holds. 

Let $n \in \n$ and $I_1,\dots,I_n \in \mathcal{I}$ be fixed in the following. Due to \eqref{PoissonConvPL} \cite[Theorem 23.16]{bo:OKallen21} implies that
\begin{equation} \label{AuxConv1}
    (\Theta_\rho(I_1),\dots, \Theta_\rho(I_n)) \overset{d}{\underset{\rho \rightarrow \infty}{\longrightarrow}} (\Psi_\alpha(I_1),\dots, \Psi_\alpha(I_n)).
\end{equation}
It follows from Lemma \ref{lemma:MainConvergenceP} that 
\begin{equation}\label{AuxConv2}
    (\Theta_\rho(I_1),\dots, \Theta_\rho(I_n)) - (\Xi_\rho(I_1),\dots, \Xi_\rho(I_n)) \overset{\p}{\underset{\rho \rightarrow \infty}{\longrightarrow}} 0.
\end{equation}
Now, \eqref{AuxConv1} and \eqref{AuxConv2} yield 
\begin{equation*} 
    (\Xi_\rho(I_1),\dots, \Xi_\rho(I_n)) \overset{d}{\underset{\rho \rightarrow \infty}{\longrightarrow}} (\Psi_\alpha(I_1),\dots, \Psi_\alpha(I_n)).
\end{equation*}
Using \cite[Theorem 23.16]{bo:OKallen21} again finishes the proof. 
\end{proof}

The proof of Corollary \ref{cor:bounded} follows the same structure as in the bounded case. Let $\Xi_\rho^{b}$ and $\Psi_\alpha^{b}$ denote the restrictions of $\Xi_\rho$ and $\Psi_\alpha$ to $W \times [b,\infty]^2$, which is a compact subset of $W \times (0,\infty]$. 
Again, we prove the convergence for these restricted point processes and then use an approximation argument. 

\begin{proof}[Proof of Corollary \ref{cor:powerlaw}] 
 Define $\mathcal{C}_b := \mathcal{C}_{W \times [b,\infty]^2}$ for $b > 0$ (see \eqref{def:SubsetCK}) and recall the mapping $T$ defined in \eqref{def:argmaxmap}. It holds that $\p(\Psi_\alpha^{b} \in \mathcal{C}_b)=1$ and Theorem \ref{thm:powerlaw}, Lemma \ref{lemma:continuity of argmax} and the continuous mapping theorem (see \cite[Theorem 5.27]{bo:OKallen21}) imply that for any $b > 0$,
 \begin{equation} \label{formula:conv in dist PL}
 T(\Xi_\rho^{b}) \underset{\rho \rightarrow \infty}{\overset{d}{\longrightarrow}} T(\Psi_\alpha^{b}).
 \end{equation}
It also holds that 
 \begin{equation}  \label{limitPoisPL}
        \lim_{b \rightarrow 0_+} \p(T(\Psi_\alpha^{b}) \neq T(\Psi_\alpha))= \lim_{b \rightarrow 0_+} \p(\Psi_\alpha(W \times [b,\infty]^2)=0)= \lim_{b \rightarrow 0_+} e^{-{b^{-\alpha}}} = 0. 
    \end{equation}
Furthermore $\{T(\Xi_\rho^{b}) \neq T(\Xi_\rho)\} \subseteq \{\Xi_\rho(W \times [b,\infty]^2)=0\} \cup  \{\Xi_\rho(W \times (0,b)\times [b,\infty])\geq 1\}$ holds for $\rho \geq 1$ and $b > 0$. Theorem \ref{thm:powerlaw} and \eqref{conv:L1} in Lemma \ref{lemma:MainConvergenceP}  together with the fact that $\Theta_\rho(W \times (0,b)\times [b,\infty] = 0$ \as imply that
\begin{align}
 \begin{split} \label{limitXirhoPL}
      0 &\leq \limsup_{b \rightarrow 0_+} \limsup_{\rho \rightarrow \infty} \p(T(\Xi_\rho^{b}) \neq T(\Xi_\rho)) \\
&\leq \limsup_{b \rightarrow 0_+} \limsup_{\rho \rightarrow \infty} \p(\Xi_\rho(W \times [b,\infty]^2)=0) + \p(\Xi_\rho(W \times (0,b)\times [b,\infty])\geq 1) \\
&\leq \limsup_{b \rightarrow 0_+} \limsup_{\rho \rightarrow \infty} \p(\Xi_\rho(W \times [b,\infty]^2)=0) \\
&\hspace{3cm} + \e\abs{\Xi_\rho(W \times (0,b)\times [b,\infty]) - \Theta_\rho(W \times (0,b)\times [b,\infty])} \\
      &= \lim_{b \rightarrow 0_+} \p(\Psi_\alpha(W \times [b,\infty]^2)=0) + 0=  \lim_{b \rightarrow 0_+}e^{- {b^{-\alpha}}} = 0.
 \end{split}
 \end{align}
 Similarly as in Corollary \ref{cor:bounded}, the limits \eqref{limitPoisPL} and \eqref{limitXirhoPL} and the convergence \eqref{formula:conv in dist PL} yield 
 $$
 T(\Xi_\rho) \underset{\rho \rightarrow \infty}{\overset{d}{\longrightarrow}} T(\Psi_\alpha).
 $$ 
To see that in fact $ T(\Psi_\alpha) = \underset{(y,m,r) \in \Psi_\alpha}{\arg \max} r$ has the desired distribution, we can regard $\Psi_\alpha$ as a Poisson process on $(0,\infty]^2$ with intensity measure $ \mathbb{K}_\alpha$ concentrated on the diagonal and i.i.d.\ copies of $ X \sim \text{Unif}(W) $ as marks. Therefore
 $$\p(\underset{(X,M,R) \in \Psi_\alpha}{\max} R \leq a) = \p(\Psi_\alpha(W \times (a,\infty]^2)=0) = e^{- a^{-\alpha}}$$
 for any $a \geq 0$ and the proof is finished. 
\end{proof}

\subsubsection*{Acknowledgements.} MŠP was supported by the
Charles University Grant Agency (project no.\,70524), the Charles University Research Center program No.\,UNCE/24/SCI/022 and the Mobility Fund of Charles University (project FM/c/2024-2-050). MS was funded by the Deutsche Forschungsgemeinschaft (DFG, German Research Foundation) via the Priority Programme 2265 {\em Random Geometric Systems}, Project 531541167.

\end{document}